\documentclass[11pt,a4paper,oneside]{article}
\usepackage[fleqn, leqno]{amsmath}
\usepackage{amssymb}
\usepackage{amsthm}
\usepackage{eucal}
\usepackage{makeidx}
\usepackage{hyperref}
\usepackage{xcolor}
\makeindex

\usepackage{fancyhdr}
\fancypagestyle{plain}{%
\fancyhead{} 
}

\theoremstyle{plain}
\newtheorem{theorem}{Theorem}[section]
\newtheorem{proposition}[theorem]{Proposition}

\newtheorem{lemma}[theorem]{Lemma}
\newtheorem{conjecture}[theorem]{Conjecture}

\theoremstyle{definition}
\newtheorem{definition}[theorem]{Definition}

\theoremstyle{remark}

\newcommand{\N}{\mathbb{N}}
\newcommand{\Z}{\mathbb{Z}}

\begin{document}


\title{Dynamic Structures of 2-adic Fibonacci Polynomials}


\author{Myunghyun Jung, Donggyun Kim
  and Kyunghwan Song}

\date{}

\maketitle \setlength{\hangindent}{0pt}

\begin{abstract}
	We prove recurrence relations and modulo periodic properties of multiple derivatives of Fibonacci polynomials over the ring of 2-adic integers. We apply the obtained results to present the dynamic structures of Fibonacci polynomials by investigating minimal decompositions which consist of minimal subsystems and attracting basins.
\end{abstract}



\tableofcontents


\section{Introduction}

Let $\Z_p$ be the ring of $p$-adic integers and $f$ a polynomial over $\Z_p$. The dynamical system $(\Z_p, f)$ has a general theorem of the minimal decomposition due to Fan and Liao in 2011 \cite{FanLiao2011Minimal}.

\begin{theorem}[\cite{FanLiao2011Minimal}
               ]
\label{thm-decomposition}
Let $f \in \mathbb{Z}_p[x]$ be a polynomial of integral coefficients with degree at least  $2$. We have the following decomposition
\[
\mathbb{Z}_p = \mathcal{P} \bigsqcup \mathcal{M} \bigsqcup \mathcal{B}
\]
where
\begin{enumerate}
\item $ \mathcal{P}$ is the finite set consisting of all periodic points of $f$,
\item $\mathcal{M}= \bigsqcup_i \mathcal{M}_i$ is the union of all (at most countably many) clopen invariant sets such that each $\mathcal{M}_i$ is a finite union of balls and each subsystem $f: \mathcal{M}_i \to \mathcal{M}_i$ is minimal, and
\item each point in $\mathcal{B}$ lies in the attracting basin of a periodic orbit or of a minimal subsystem.
\end{enumerate}
\end{theorem}
	
This decomposition is usually referred to as a \emph{minimal decomposition} and the invariant subsets $\mathcal{M}_i$ are called {\em minimal components }.
	
Although we have a general decomposition theorem, it is still difficult to describe  the explicit minimal decomposition for  a given polynomial,  as there have been few works done on them. Multiplications on $\Z_p$ for $p \geq 3$ were studied by Coelho and Parry in 2001 \cite{CoelhoParry2001}, and affine maps and quotient maps of affine maps for all $p$ were studied by Fan, Li, Yao and Zhou in 2007 \cite{FanLiYaoZhou2007Affine} and by Fan, Fan, Liao and Wang in 2014 \cite{FanFanLiaoWang2014Homographic}, respectively. The square map on $\Z_p$ for all primes $p$ and quadratic maps for $p=2$ were investigated by Fan and Liao in 2016 \cite{FanLiao2016Square} and  2011 \cite{FanLiao2011Minimal}, respectively. These functions are all of low degrees. Recently, the Chebyshev polynomials, whose degrees increase infinitely, on $\Z_2$ were investigated by Fan and Liao in 2016 \cite{FanLiao2016Chebyshev}.

In this paper, we study the dynamic structures of Fibonacci polynomials over the ring of 2-adic integers. The Fibonacci polynomials $F_m(x)$ are defined for nonnegative integers $m$ by the recurrence relation, see (\ref{eq:Fibonacci polynomial_recurrence relation}) bellow. The dynamic structures of the Fibonacci polynomials $F_m(x)$ are described by investigating minimal decompositions which consist of minimal subsystems and attracting basins. The results show that there are full of varieties of minimal decomposition patterns.  The degrees of the Fibonacci polynomials increase infinitely same as the Chebyshev polynomials, but the pattern of minimal decompositions of $F_m$ appears differently, for example a Fibonacci polynomial has  sufficiently small minimal components, see Theorem \ref{thm:m=4+12q,q=2+64d} and \ref{thm:m=8+12q,q=61+64d}. 

The rest of the paper is organized as follows. In section 2, we review basic definitions and properties of $p$-adic polynomials and Fibonacci polynomials $F_m(x)$. In addition, we present recurrence relations of multiple derivatives of $F_m(x)$ and for a fixed $x$, modulo periodic properties of $F_m(x)$'s and their multiple derivatives for nonnegative integers $m$, which are essential tools for investigation of main results. In next sections, we study the minimal decompositions for $F_m(x)$ with $m$ odd in section 3, with $m \equiv 0 \pmod{12}$ in section 4, with $m \equiv 2 \pmod{12}$ in section 5, with $m \equiv 4 \pmod{12}$ in section 6, with $m \equiv 6 \pmod{12}$ in section 7, with $m \equiv 8 \pmod{12}$ in section 8 and with $m \equiv 10 \pmod{12}$ in section 9.

\section{Preliminary}
\subsection{$p$-adic spaces}
We first review the basic definitions and properties of the $p$-adic space. Let $\N$ be the set of positive integers, $\Z$ the set of integers and $p$ a prime number. We consider a family of rings $\{ \Z/p^k \Z\}$ for all $k\in \N$ with canonical projections $\{ \varphi_{i j} : \Z/p^j \Z \to \Z/p^i \Z \}$ for all integers $i$ and $j$ with $i\le j$. The \emph{ring of $p$-adic integers}, denoted $\Z_p$, is the projective limit of the projective system of rings $\{\Z/p^k \Z\}$ with projections $\{\varphi_{i j} \}$.

Then the ring $\Z_p$ is the set of all sequence $(x_k)$ with the property $x_k \in \Z/p^k \Z$ and $\varphi_{i j}(x_j)=x_i$ for all $i \le j$. A point $x_k$ in $\Z/p^k \Z$ can be represented by $x_k=a_0 + a_1 p+a_2 p^2 + a_3 p^3 + \cdots+a_{k-1}p^{k-1}$ with unique $a_i \in\{0,1,\cdots,p-1\}$ and the point $x = (x_k)$ in $\Z_{p}$ by $x: a_0 , a_0 + a_1 p, a_0 + a_1 p+a_2 p^2, a_0 + a_1 p+a_2 p^2 + a_3 p^3, \cdots $ as a sequence form, or simply $x=a_0 + a_1 p+a_2 p^2 + a_3 p^3 + \cdots$.

An element $x$ can be written as $x=p^\nu (b_0 + b_1 p+b_2 p^2 + b_3 p^3 + \cdots)$ with $\nu\in \N\cup \{0\}$ and $b_i \in\{0,1,\cdots,p-1\}$. In this case, we say that $p^\nu$ \emph{divides} $x$ and denote $p^\nu |x$.

The \emph{$p$-adic order} or \emph{$p$-adic valuation} for $\Z_p$ is defined as the function $\nu_p : \Z_p \rightarrow \mathbb{Z} \cup{\infty}$
$$
\nu_p (x) = \begin{cases}\textrm{max}\{\nu \in \N\cup \{0\} : p^\nu | x \}&\textrm{if}~x \neq 0,\\
\infty &\textrm{if}~x =0,\end{cases}
$$

and the \emph{$p$-adic metric} of $\Z_p$ is defined as the function $|\cdot |_p : \Z_p \rightarrow \mathbb{R}$
$$
|x|_p = \begin{cases}p^{-\nu_p (x)} &\textrm{if}~x \neq 0,\\
0 &\textrm{if}~x =0.\end{cases}
$$
With the $p$-adic metric, $\Z_p$ becomes a compact space.

Note that $\Z/p^k\Z$ can be viewed as a subset of $\Z_{p}$,
\[ \Z/p^k\Z=\{x\in\Z_{p}: a_{i}=0 \text{ for every } i\ge k\}, \]
and for $x=a_0 + a_1 p+a_2 p^2 + a_3 p^3 + \cdots$ in $\Z_p$, the element  $x \pmod{ p^k}$ means $a_0 + a_1 p+a_2 p^2 + a_3 p^3 + \cdots+a_{k-1}p^{k-1}$, which is now in $\Z_p / p^k \Z_p$.
It can be checked that $\Z/p^k\Z$ is isomorphic to $\Z_p / p^k \Z_p$.

\begin{proposition}
Let $x$ and $y$ be elements of $\Z_p$. Then, $\nu_p$ has the properties:
\begin{enumerate}
  \item $\nu_p(xy)=\nu_p(x)+\nu_p(y).$
  \item $\nu_p(x+y)\geq \min\{\nu_p(x),\nu_p(y)\}$, with equality when $\nu_p(x)$ and $\nu_p(y)$ are unequal.
\end{enumerate}
\end{proposition}

\subsection{Dynamics of $p$-adic polynomials}

Let $f$ be a polynomial over $\Z_p$ and let $l\geq1$ be a positive integer. Let $\sigma=\{x_1,\cdots,x_k\}$ with $x_i \in \Z/p^l\Z$ be a cycle of $f$ of length $k$ (also called a $k$-cycle), i.e.,
$$f(x_1)=x_2,\cdots,f(x_i)=x_{i+1},\cdots,f(x_k)=x_1~(\mathrm{mod}~p^l).$$
In this case, we also say that the cycle $\sigma$ is \emph{at level l}. Let
\[ X_\sigma:=\bigcup_{i=1}^k X_i~\mathrm{where}~X_i:=\{x_i+p^lt+p^{l+1}\Z:~t=0,\cdots,p-1\}. \]
Then
\[ f(X_i)\subset X_{i+1}~(1\leq i\leq k-1)~\mathrm{and}~f(X_k)\subset X_1, \]
and so $f$ is invariant on the clopen set $X_\sigma$. The cycles in $X_\sigma$ of $f$ (mod $p^{l+1}$) are called \emph{lifts of $\sigma$} (from level $l$ to level $l+1$). Note that the length of a lift of $\sigma$ is a multiple of $k$.

Let $\mathbb{X}_i:=x_i+p^l\Z_p$ and let $\mathbb{X}_{\sigma}:=\bigcup_{i=1}^{k}\mathbb{X}_i$. For $x\in\mathbb{X}_{\sigma}$, denote
\begin{align}
&a_l(x):=\prod_{j=0}^{k-1}f'\big(f^j(x)\big) \label{a_l}
 \mathrm{~and~}\\
&b_l(x):=\frac{f^k(x)-x}{p^l} \label{b_l}.
\end{align}
Sometimes we abbreviate $a_l(x)$ and $b_l(x)$ to $a_l$ and $b_l$, respectively. The following definition and proposition were treated by Fan et al \cite{FanFares2011, FanLiYaoZhou2007Affine}. Hereafter we fix the prime $p=2$.

\begin{definition}\label{def:movement}

We say $\sigma$ \emph{strongly grows at level l} if $a_l\equiv1$ (mod 4) and $b_l\equiv1$ (mod 2).

We say $\sigma$ \emph{strongly splits at level l} if $a_l\equiv1$ (mod 4) and $b_l\equiv0$ (mod 2).

We say $\sigma$ \emph{weakly grows at level l} if $a_l\equiv3$ (mod 4) and $b_l\equiv1$ (mod 2).

We say $\sigma$ \emph{weakly splits at level l} if $a_l\equiv3$ (mod 4) and $b_l\equiv0$ (mod 2).

We say $\sigma$ \emph{grows tails at level l} if $a_l\equiv0$ (mod 2).
\end{definition}

\begin{proposition}\label{prop:grow tails}
Let $\sigma=\{x_1,\cdots,x_k\}$ be a cycle of $f$ at level $l$.
\begin{enumerate}
\item If $\sigma$ is a strongly growing cycle at level $l\geq2$, then $f$ restricted onto the invariant clopen set $\mathbb{X}_{\sigma}$ is minimal.

\item If $\sigma$ is a strongly splitting cycle at level $l\geq2$ and $\nu_2(b_l(x))=s\geq1$ for all $x\in\mathbb{X}_\sigma$, then the lifts of $\sigma$ will keep splitting until to the level $l+s$ at which all lifts strongly grow.

\item Let $\sigma$ be a cycle of $f$ at level $n\geq2$. If $\sigma$ weakly grows then the lift of $\sigma$ strongly splits.

\item Let $\sigma$ be a weakly splitting cycle of $f$ at level $n\geq2$. Then one lift behaves the same as $\sigma$ while the other weakly grows and then strongly splits.

\item If $\sigma$ is a growing tails cycle at level $l\geq1$, then $f$ has a $k$-periodic orbit in the clopen set $\mathbb{X}_{\sigma}$ lying in its attracting basin.
\end{enumerate}
\end{proposition}

\subsection{Fibonacci polynomials over $\Z_2$}

The \emph{Fibonacci polynomials} are defined for nonnegative integers $m$ by the recurrence relation
\begin{equation}\label{eq:Fibonacci polynomial_recurrence relation}
F_m(x)=x\, F_{m-1}(x)+ F_{m-2}(x),\ \text{with}\ F_0(x)=0\ \text{and}\ F_1(x)=1.
\end{equation}
The first few Fibonacci polynomials are
\begin{flalign}
F_2(x) & = x,\notag \\
F_3(x) & = 1+x^2,\notag \\
F_4(x) & = 2x+x^3,\notag \\
F_5(x) & = 1+3x^2+x^4,\notag\\
F_6(x) & = 3x+4x^3+x^5.\label{Fibonacci_6}
\end{flalign}
	
The Fibonacci polynomial $F_m(x)$ is also given by the explicit summation formula
\begin{equation}\label{eq:Fibonacci polynomial_2}
F_m(x)=\sum_{j=0}^{\lfloor \frac{m-1}{2} \rfloor}
\tbinom{m-1-\lfloor \frac{m-1}{2} \rfloor+j}{\lfloor \frac{m-1}{2} \rfloor-j}x^{m-1 -2\lfloor \frac{m-1}{2} \rfloor+2j},
\end{equation}
where $\lfloor\cdot\rfloor$ is the floor function and $\tbinom{\cdot}{\cdot}$ is the binomial coefficient. There are concrete forms of the Fibonacci polynomials depending on the parity of $m$:
\begin{equation}\label{eq:Fibonacci polynomial_even&odd}
F_m(x)=
\begin{cases}
\sum_{j=0}^{\lfloor \frac{m}{2} \rfloor-1}
\tbinom{\lfloor \frac{m}{2} \rfloor+j}{1+2j}x^{1+2j}, &\text{if $m$ even;}\\
\sum_{j=0}^{\lfloor \frac{m}{2} \rfloor}
\tbinom{\lfloor \frac{m}{2} \rfloor+j}{2j}x^{2j}, &\text{if $m$ odd.}
\end{cases}
\end{equation}

We present the basic properties of Fibonacci polynomials. For general information, see, for example, \cite{Koshy2018book, Posamentier2007Fibonaccibook}.
\begin{proposition}
Let $F_m$ be the Fibonacci polynomial. Then
\begin{align}
&F_{m+n}(x)=F_{m+1}(x)F_n(x)+F_m(x)F_{n-1}(x),\\
&F_{2m}(x)=F_m(x)\big(F_{m+1}(x)+F_{m-1}(x)\big)~\mathrm{and}\label{Fib:eq variant1}\\
&F_{2m+1}(x)=F_{m+1}(x)^2+F_m(x)^2\label{Fib:eq variant2}.
\end{align}
\end{proposition}

\begin{definition}
A sequence $\{f_m\}_m$ is called \emph{periodic of period} $d$ if there is the least positive integer $d$ such that $f_{m+d}=f_m$ for every $m$.
\end{definition}

Let $s$ be an element of $\Z_2$. For convenience, we say that $s$ is an \emph{odd number} if $s$ is in $1+2\Z_2$, and an \emph{even number} if $s$ is in $2\Z_2$.

We present recurrence relations of multiple derivatives of $F_m(x)$ and for a fixed $x$, modulo periodic properties of $F_m(x)$'s and their multiple derivatives for nonnegative integers $m$, which are essential tools for investigation of main results.

\begin{proposition}\label{prop:Fib periodic}
Let $F_m$ be the Fibonacci polynomial.
\begin{enumerate}
\item For any odd number $s$, $F_{3\cdot 2^l}(s)\equiv 0$ and $F_{3\cdot 2^l+1}(s)\equiv 1+2^{l+1}$ (mod $2^{l+2}$) for any positive integer $l$.

\item For any odd number $s$, the sequence $\{F_m(s)~(\mathrm{mod}~2^l)\}_m$ is periodic of period $3\cdot 2^{l-1}$ for any positive integer $l$.

\item For any even number $s$, the sequence $\{F_m(s)~(\mathrm{mod}~2^l)\}_m$ is periodic of period 2 if $l\leq \nu_2(s)$ and periodic of period $2^{l+1-\nu_2(s)}$ if $l>\nu_2(s)$.
\end{enumerate}
\end{proposition}

\begin{proof}
1. Let $s$ be an odd number. We use induction on positive integers $l$. Let $l=1$. Then,
\begin{align*}
&F_6(s)=3s+4s^3+s^5\equiv 3s+4s+s \equiv 0~(\mathrm{mod}~2^3)~\mathrm{and}\\
&F_7(s)=1+6s^2+5s^4+s^6\equiv 1+6+5+1\equiv 1+2^2~(\mathrm{mod}~2^3).
\end{align*}
Here, we used the fact that $s^2\equiv 1$ (mod $2^3$).

For the induction hypothesis, we assume that $F_{3\cdot 2^l}(s)\equiv 0$ and $F_{3\cdot 2^l+1}(s)\equiv 1+2^{l+1}$ (mod $2^{l+2}$) for an integer $l\geq 1$. By the recurrence relation (\ref{eq:Fibonacci polynomial_recurrence relation}), $F_{3\cdot 2^l-1}(s)\equiv 1+2^{l+1}$ (mod $2^{l+2}$). Then, by expression (\ref{Fib:eq variant1}),
\begin{align*}
F_{3\cdot 2^{l+1}}(s)&=F_{3\cdot 2^l}(s)\big(F_{3\cdot 2^l+1}(s)+F_{3\cdot 2^l-1}(s)\big)\\
&\equiv \big(0~(\mathrm{mod}~2^{l+2})\big)\big((1+2^{l+1}~(\mathrm{mod}~2^{l+2})) +(1+2^{l+1}~(\mathrm{mod}~2^{l+2}))\big)\\
&\equiv \big(0~(\mathrm{mod}~2^{l+2})\big)\big(2~(\mathrm{mod}~2^{l+2})\big)\\
&\equiv 0~(\mathrm{mod}~2^{l+3}).
\end{align*}
From expression (\ref{Fib:eq variant2}),
\begin{align*}
F_{3\cdot 2^{l+1}+1}(s)&=F_{3\cdot 2^{l}+1}(s)^2+F_{3\cdot 2^{l}}(s)^2\\
&\equiv \big(1+2^{l+1}~(\mathrm{mod}~2^{l+2})\big)^2+\big(0~(\mathrm{mod}~2^{l+2})\big)^2\\
&\equiv 1+2^{l+2}~(\mathrm{mod}~2^{l+3}),
\end{align*}
which completes the induction step.

2. Let $s$ be an odd number. We use induction on positive integers $l$. Let $l=1$. Then, the Fibonacci polynomial becomes that $F_0(s)=0,~F_1(s)=1$ and
\begin{align*}
F_m(s)=s\cdot F_{m-1}(s)+F_{m-2}(s)\equiv F_{m-1}(s)+F_{m-2}(s)~(\mathrm{mod}~2)
\end{align*}
for $m\geq 2$. Therefore, the sequence
\[ \{F_m(s)~(\mathrm{mod}~2)\}_m=\{0,1,1,0,1,1,0\dots\} \]
is periodic of period 3.

Let $l=2$. Similarly, the Fibonacci polynomial becomes that $F_0(s)=0,~F_1(s)=1$ and
\begin{align*}
F_m(s)\equiv s\cdot F_{m-1}(s)+F_{m-2}(s)~(\mathrm{mod}~2)
\end{align*}
for $m\geq 2$. Therefore, the sequence
\[ \{F_m(s)~(\mathrm{mod}~2^2)\}_m=\{0,1,s,2,2+s,1,0,1,s,2,2+s,1,0,\dots\} \]
is periodic of period 6.

For the induction hypothesis, we assume that the sequence $\{F_m(s)~(\mathrm{mod}~2^l)\}_m$ is periodic of period $3\cdot 2^{l-1}$ for an integer $l\geq2$. We know that $F_0(s)=0$ and $F_1(s)=1$. Since the period is $3\cdot 2^{l-1}$, we obtain
\[ F_{3\cdot 2^{l-1}}(s)\equiv 0~\mathrm{and}~F_{3\cdot 2^{l-1}+1}(s)\equiv 1~(\mathrm{mod}~2^l). \]
Further, by the recurrence relation (\ref{eq:Fibonacci polynomial_recurrence relation}), $F_{3\cdot 2^{l-1}-1}(s)\equiv 1~(\mathrm{mod}~2^l)$. Then, by expression (\ref{Fib:eq variant1}),
\begin{align*}
F_{3\cdot 2^l}(s)&=F_{3\cdot 2^{l-1}}(s)\big(F_{3\cdot 2^{l-1}+1}(s)+F_{3\cdot 2^{l-1}-1}(s)\big)\\
&\equiv \big(0~(\mathrm{mod}~2^l)\big)\big(1~(\mathrm{mod}~2^l)+1~(\mathrm{mod}~2^l)\big)\\
&\equiv 0~(\mathrm{mod}~2^{l+1})
\end{align*}
From expression (\ref{Fib:eq variant2}),
\begin{align*}
F_{3\cdot 2^{l}+1}(s)&=F_{3\cdot 2^{l-1}+1}(s)^2+F_{3\cdot 2^{l-1}}(s)^2\\
&\equiv \big(1~(\mathrm{mod}~2^l)\big)^2+\big(0~(\mathrm{mod}~2^l)\big)^2\\
&\equiv 1~(\mathrm{mod}~2^{l+1}).
\end{align*}
As a result, we have $F_0(s)\equiv 0,F_1(s)\equiv 1, F_{3\cdot 2^l}(s)\equiv 0$ and $F_{3\cdot 2^l+1}(s)\equiv 1$ (mod $2^{l+1}$). Therefore, by the recurrence relation (\ref{eq:Fibonacci polynomial_recurrence relation}), the sequence $\{F_m(s)~(\mathrm{mod}~2^{l+1})\}_m$ is periodic of period less than or equal to $3\cdot 2^l$. By the induction hypothesis, the period is a multiple of $3\cdot 2^{l-1}$. By property 1 of Proposition \ref{prop:Fib periodic}, the sequence $\{F_m(s)~(\mathrm{mod}~2^{l+1})\}_m$ is periodic of period $3\cdot 2^l$, which completes the induction step.

3. Assume that $s$ is an even integer, hence $\nu_2(s)\geq 1$. Suppose that $l\leq \nu_2(s)$. Then, $s^j\equiv 0~(\mathrm{mod}~2^l)$ for every positive integer $j$, hence $F_m(s)\equiv F_m(0)~(\mathrm{mod}~2^l)$. Since $F_m(0)=0$ for even numbers $m$ and $F_m(0)=1$ for odd numbers $m$, we obtain the sequence $\{F_m(s)~(\mathrm{mod}~2^{l})\}_m=\{0,1,0,1,\dots\}$ which is periodic of period 2.

Now, suppose that $l>\nu_2(s)$. We use induction on $l\geq \nu_2(s)+1$. First assume that $l=\nu_2(s)+1$. If $m=2k$, then, from the expression (\ref{eq:Fibonacci polynomial_even&odd}), $F_m(s)\equiv\tbinom{k}{1}s\equiv k\cdot 2^{\nu_2(s)}~(\mathrm{mod}~2^{\nu_2(s)+1})$.
Further, when $k$ is even, $F_m(s)\equiv 0~(\mathrm{mod}~2^{\nu_2(s)+1})$ and when $k$ is odd, $F_m(s)\equiv 2^{\nu_2(s)}~(\mathrm{mod}~2^{\nu_2(s)+1})$. If $m=2k+1$, then $F_m(s)\equiv\tbinom{k}{0}\equiv 1~(\mathrm{mod}~2^{\nu_2(s)+1})$. Altogether, we obtain the sequence
\[ \{F_m(s)~(\mathrm{mod}~2^{\nu_2(s)+1})\}_m=\{0,1,2^{\nu_2(s)},1,0,1,2^{\nu_2(s)},1,0,\dots\} \]
is periodic of period $2^2$, where the exponent 2 satisfies the expression $2=l+1-\nu_2(s)$. Therefore, the assertion holds for $l=\nu_2(s)+1$.

For the induction hypothesis, we assume that the assertion holds for some $l\geq \nu_2(s)+1$. Then, we have that
\[ F_{2^{l+1-\nu_2(s)}}(s)\equiv 0~~\mathrm{and}~~F_{2^{l+1-\nu_2(s)}+1}(s)\equiv 1~(\mathrm{mod}~2^l), \]
and by the recurrence relation (\ref{eq:Fibonacci polynomial_recurrence relation}),
\[ F_{2^{l+1-\nu_2(s)}-1}(s)\equiv 1~~\mathrm{and}~~F_{2^{l+1-\nu_2(s)}+2}(s)\equiv s~(\mathrm{mod}~2^l). \]
By expression (\ref{Fib:eq variant1}),
\begin{align*}
F_{2^{l+1-\nu_2(s)+1}}(s)&=F_{2^{l+1-\nu_2(s)}}(s)\big(F_{2^{l+1-\nu_2(s)}+1}(s)+F_{2^{l+1-\nu_2(s)}-1}(s)\big)\\
&\equiv \big(0~(\mathrm{mod}~2^l)\big)\big(1~(\mathrm{mod}~2^l)+1~(\mathrm{mod}~2^l)\big)\\
&\equiv 0~(\mathrm{mod}~2^{l+1}).
\end{align*}
From expression (\ref{Fib:eq variant2}),
\begin{align*}
F_{2^{l+1-\nu_2(s)+1}+1}(s)&=F_{2^{l+1-\nu_2(s)}+1}(s)^2+F_{2^{l+1-\nu_2(s)}}(s)^2\\
&=\big(1~(\mathrm{mod}~2^l)\big)^2+\big(0~(\mathrm{mod}~2^l)\big)^2\\
&\equiv 1~(\mathrm{mod}~2^{l+1}).
\end{align*}
Hence, we obtain $F_0(s)\equiv 0,~F_1(s)\equiv 1,~F_{2^{l+1-\nu_2(s)+1}}(s)\equiv 0$ and $F_{2^{l+1-\nu_2(s)+1}+1}(s)\equiv 1~(\mathrm{mod}~2^{l+1})$. Therefore, by the recurrence relation (\ref{eq:Fibonacci polynomial_recurrence relation}), the sequence
    $\{F_m(s)~(\mathrm{mod}~2^{l+1})\}_m$ is periodic of period
    $2^{l+1-\nu_2(s)+1}$, which completes the induction step.
\end{proof}

\begin{lemma}\label{lem:Fib 2valuation}
Let $F_m$ be the Fibonacci polynomial. For any odd number $s$, $\nu_2\big(F_{3\cdot 2^l}(s)\big)=l+2$ for every integer $l\geq 1$.
\end{lemma}

\begin{proof}
Let $s$ be an odd number. We use induction on $l$. Write $s=1+2t$ for some integer $t$, then we obtain, from  expression (\ref{Fibonacci_6}),
\[ F_{3\cdot 2}(1+2t)=8(1+2t)\big(1+t(t+1)\big)\big(1+2t(t+1)\big). \]
Therefore, we obtain $\nu_2\big(F_{3\cdot 2}(s)\big)=3$, hence the assertion holds true for $l=1$. Now assume that the assertion holds for some $l\geq 1$. By property 1 of Proposition \ref{prop:Fib periodic}, we obtain for an odd number $s$,
\begin{align*}
F_{3\cdot 2^l}(s)\equiv F_{0}(s)\equiv 0~\mathrm{and}~F_{3\cdot 2^l+1}(s)\equiv F_{1}(s)\equiv 1~(\mathrm{mod}~2^{l+1}).
\end{align*}
Further, by the recurrence relation (\ref{eq:Fibonacci polynomial_recurrence relation}), $F_{3\cdot 2^l-1}(s)\equiv 1~(\mathrm{mod}~2^{l+1})$. As a result, we have that $\nu_2\big(F_{3\cdot 2^l+1}(s)+F_{3\cdot 2^l-1}(s)\big)=1$. From expression (\ref{Fib:eq variant1}), we have that
\begin{align*}
\nu_2\big(F_{3\cdot 2^{l+1}}(s)\big)&=\nu_2(F_{3\cdot 2^l}(s))+\nu_2\big(F_{3\cdot 2^l+1}(s)+F_{3\cdot 2^l-1}(s)\big)\\
&=(l+2)+1=l+3.
\end{align*}
Hence, the assertion holds true for $l+1$; therefore, the proof follows by induction.
\end{proof}

\begin{lemma}\label{lem:Fib input i}
Let $F_m$ be the Fibonacci polynomial. The sequence $\{\frac{F_m(i)}{i}\}_m$ is periodic of period 12 where $i=\sqrt{-1}$.
\end{lemma}

\begin{proof}
It is obvious by induction.
\end{proof}
The first 12 entries of the sequence are $\{0,-i,1,0,1,i,0,i,-1,0,-1,-i\}$.

By differentiating the relation (\ref{eq:Fibonacci polynomial_recurrence relation}), the derivative of $F_m(x)$ is given by
\begin{equation}\label{eq:Fibonacci polynomial_recurrence relation_2}
F_m'(x)=F_{m-1}(x)+x F_{m-1}'(x)+ F_{m-2}'(x), \ \text{with}\ F_0'(x)=0 \ \text{and}\  F_1'(x)=0.
\end{equation}

By differentiating the above relation (\ref{eq:Fibonacci polynomial_recurrence relation_2}), the second derivative of $F_m(x)$ is given by
\begin{equation}\label{eq:Fibonacci polynomial_recurrence relation_3}
F_m''(x)=2 F_{m-1}'(x)+x F_{m-1}''(x)+ F_{m-2}''(x), \ \text{with}\ F_0''(x)=0 \ \text{and}\ F_1''(x)=0.
\end{equation}

By differentiating the above relation (\ref{eq:Fibonacci polynomial_recurrence relation_3}), the third derivative of $F_m(x)$ is given by
\begin{equation}\label{eq:Fibonacci polynomial_recurrence relation_4}
F_m'''(x)=3 F_{m-1}''(x)+x F_{m-1}'''(x)+ F_{m-2}'''(x), \ \text{with}\ F_0'''(x)=0 \ \text{and}\ F_1'''(x)=0.
\end{equation}

\begin{lemma}\label{lem:Fib a_n}
Let $F_m$ be the Fibonacci polynomial. For any odd numbers $s,~s_1$ and $s_2$,
\begin{enumerate}
\item $F_{3\cdot 2^{l-1}}'(s)\equiv 0$ and $F_{3\cdot 2^{l-1}+1}'(s)\equiv 2^{l-1}$ (mod $2^l$) for any positive integer $l$,

\item the sequence $\{F_m'(s)~(\mathrm{mod}~2^l)\}_m$ is periodic of period $3\cdot 2^l$ for any positive integer $l$, and

\item the sequence $\{F_m'(s_1)\cdot F_m'(s_2)~(\mathrm{mod}~4)\}_m$ is periodic of period 6.
\end{enumerate}
\end{lemma}

\begin{proof}
1. Let $s$ be an odd number. We use induction on positive integers $l$. Let $l=1$. Then,
\begin{align*}
F_3'(s)=2s\equiv 0~\mathrm{and}~F_4'(s)=2+3s^2\equiv 1~(\mathrm{mod}~2).
\end{align*}

For the induction hypothesis, we assume that $F_{3\cdot 2^{l-1}}'(s)\equiv 0$ and $F_{3\cdot 2^{l-1}+1}'(s)\equiv 2^{l-1}$ (mod $2^l$) for any integer $l\geq 1$. By the recurrence relations (\ref{eq:Fibonacci polynomial_recurrence relation}), (\ref{eq:Fibonacci polynomial_recurrence relation_2}) and Proposition \ref{prop:Fib periodic}, we have $F_{3\cdot 2^{l-1}}(s)\equiv 0$, $F_{3\cdot 2^{l-1}+1}(s)\equiv 1$, $F_{3\cdot 2^{l-1}-1}(s)\equiv 1$ and $F_{3\cdot 2^{l-1}-1}'(s)\equiv 2^{l-1}$ (mod $2^{l}$). Then, by differentiating expression (\ref{Fib:eq variant1}) at $s$, we obtain
\begin{align*}
F_{3\cdot 2^{l}}'(s)&=F_{3\cdot 2^{l-1}}'(s)\big(F_{3\cdot 2^{l-1}+1}(s)+F_{3\cdot 2^{l-1}-1}(s)\big)\\
&\quad +F_{3\cdot 2^{l-1}}(s)\big(F_{3\cdot 2^{l-1}+1}'(s)+F_{3\cdot 2^{l-1}-1}'(s)\big)\\
&\equiv \big(0 (\mathrm{mod}~2^l)\big)\big(1 (\mathrm{mod}~2^l)+1 (\mathrm{mod}~2^l)\big)\\
&\quad +\big(0 (\mathrm{mod}~2^l)\big)\big(2^{l-1} (\mathrm{mod}~2^l)+2^{l-1} (\mathrm{mod}~2^l)\big)\\
&\equiv \big(0 (\mathrm{mod}~2^l)\big)\big(2 (\mathrm{mod}~2^l)\big)+\big(0 (\mathrm{mod}~2^l)\big)\big(0 (\mathrm{mod}~2^l)\big)\\
&\equiv 0~(\mathrm{mod}~2^{l+1}).
\end{align*}
By differentiating expression (\ref{Fib:eq variant2}) at $s$, we obtain
\begin{align*}
F_{3\cdot 2^{l}+1}'(s)&=2 F_{3\cdot 2^{l-1}+1}(s) F_{3\cdot 2^{l-1}+1}'(s)+2 F_{3\cdot 2^{l-1}}(s) F_{3\cdot 2^{l-1}}'(s)\\
&\equiv 2\big(1 (\mathrm{mod}~2^l)\big)\big(2^{l-1} (\mathrm{mod}~2^l)\big)+2\big(0 (\mathrm{mod}~2^l)\big)\big(0 (\mathrm{mod}~2^l)\big)\\
&\equiv 2^l~(\mathrm{mod}~2^{l+1}),
\end{align*}
which completes the induction step.

2. Let $s$ be an odd number. We use induction on positive integers $l$. Let $l=1$. Then, the derivative of the Fibonacci polynomial at $s$ becomes that $F_0'(s)=0$, $F_1'(s)=0$ and
\begin{align*}
F_m'(s)&=F_{m-1}(s)+s F_{m-1}'(s)+F_{m-2}'(s)\\
&\equiv F_{m-1}(s)+F_{m-1}'(s)+F_{m-2}'(s)~(\mathrm{mod}~2)
\end{align*}
for $m\geq 2$. Since the sequence $\{F_m(s)~(\mathrm{mod}~2)\}_m$ is periodic of period 3 by Proposition \ref{prop:Fib periodic}, the sequence
\[ \{F_m'(s)~(\mathrm{mod}~2)\}_m=\{0,0,1,0,1,0,0,0,1,0,1,0,0\dots\} \]
is periodic of period 6.

For the induction hypothesis, we assume that the sequence $\{F_m'(s)~(\mathrm{mod}~2^l)\}_m$ is periodic of period $3\cdot 2^l$ for an integer $l\geq1$. We know that $F_0'(s)=0$, $F_1'(s)=0$. By property 1 of Lemma \ref{lem:Fib a_n}, we have $F_{3\cdot 2^{l+1}}'(s)\equiv 0$ and $F_{3\cdot 2^{l+1}+1}'(s)\equiv 2^{l+1}$~(mod $2^{l+2}$), which implies that $F_{3\cdot 2^{l+1}}'(s)\equiv 0$ and $F_{3\cdot 2^{l+1}+1}'(s)\equiv 0$~(mod $2^{l+1}$). By Proposition \ref{prop:Fib periodic}, we have $F_1(s)\equiv 1$ and $F_{3\cdot 2^{l+1}+1}(s)\equiv 1$ (mod $2^{l+1}$).

Therefore, we have that $F_0'(s)=0$, $F_1'(s)=0$, $F_1(s)=1$, $F_{3\cdot 2^{l+1}}'(s)\equiv 0$, $F_{3\cdot 2^{l+1}+1}'(s)\equiv 0$ and $F_{3\cdot 2^{l+1}+1}(s)\equiv 1$ (mod $2^{l+1}$). Therefore, by the recurrence relation (\ref{eq:Fibonacci polynomial_recurrence relation_2}), the sequence $\{F_m'(s)~(\mathrm{mod}~2^{l+1})\}_m$ is periodic of period less than or equal to $3\cdot 2^{l+1}$. By induction hypothesis, the period is a multiple of $3\cdot 2^{l}$. By property 1 of Lemma \ref{lem:Fib a_n}, the sequence $\{F_m'(s)~(\mathrm{mod}~2^{l+1})\}_m$ is periodic of period $3\cdot 2^{l+1}$, which completes the induction step.

3. From property 2 of  Lemma \ref{lem:Fib a_n}, the sequence
\[ \{F_m'(s)~(\mathrm{mod}~4)\}_m=\{0,0,1,2,1,2,0,2,3,2,3,0,0,\dots\} \]
is periodic of period 12. So, the sequence
\[ \{F_m'(s_1)\cdot F_m'(s_2)~(\mathrm{mod}~4)\}_m=\{0,0,1,0,1,0,0,0,1,0,1,0,0,\dots\} \]
is periodic of period 6.
\end{proof}

\begin{lemma}\label{lem:Fib second derivative}
Let $F_m$ be the Fibonacci polynomial. For any odd number $s$,
\begin{enumerate}
\item $F_{3\cdot 2^{l}}''(s)\equiv 0$ and $F_{3\cdot 2^{l}+1}''(s)\equiv 2^{l}$ (mod $2^{l+1}$) for any positive integer $l$, and

\item the sequence $\{F_m''(s)~(\mathrm{mod}~2^{l+1})\}_m$ is periodic of period $3\cdot 2^{l+1}$ for any positive integer $l$.
\end{enumerate}
\end{lemma}

\begin{proof}
1. Let $s$ be an odd number. We use induction on positive integers $l$. Let $l=1$. Then,
\begin{align*}
F_6''(s)=24s + 20s^3\equiv 0~\mathrm{and}~F_7''(s)=12+60s^2+30s^4\equiv 2s^4\equiv 2~(\mathrm{mod}~4).
\end{align*}

For the induction hypothesis, we assume that $F_{3\cdot 2^{l}}''(s)\equiv 0$ and $F_{3\cdot 2^{l}+1}''(s)\equiv 2^{l}$ (mod $2^{l+1}$) for any integer $l\geq 1$. By the recurrence relation (\ref{eq:Fibonacci polynomial_recurrence relation}) and Proposition \ref{prop:Fib periodic}, we have $F_{3\cdot 2^{l}}(s)\equiv 0$, $F_{3\cdot 2^{l}+1}(s)\equiv 1+2^{l+1}$ and $F_{3\cdot 2^{l}-1}(s)\equiv 1+2^{l+1}$ (mod $2^{l+2}$). By the recurrence relation (\ref{eq:Fibonacci polynomial_recurrence relation_2}) and Lemma \ref{lem:Fib a_n}, we have $F_{3\cdot 2^{l}}'(s)\equiv 0$, $F_{3\cdot 2^{l}+1}'(s)\equiv 2^{l}$ and $F_{3\cdot 2^{l}-1}'(s)\equiv 2^{l}$ (mod $2^{l+1}$). By the recurrence relation (\ref{eq:Fibonacci polynomial_recurrence relation_3}), we have $F_{3\cdot 2^{l}-1}''(s)\equiv 2^{l}$ (mod $2^{l+1}$). Then, by differentiating the expression (\ref{Fib:eq variant1}) twice at $s$, we have
\begin{align*}
F_{3\cdot 2^{l+1}}''(s)&=F_{3\cdot 2^{l}}''(s)\big(F_{3\cdot 2^{l}+1}(s)+F_{3\cdot 2^{l}-1}(s)\big)+F_{3\cdot 2^{l}}'(s)\big(F_{3\cdot 2^{l}+1}'(s)+F_{3\cdot 2^{l}-1}'(s)\big)\\
&\quad +F_{3\cdot 2^{l}}(s)\big(F_{3\cdot 2^{l}+1}''(s)+F_{3\cdot 2^{l}-1}''(s)\big)\\
&\equiv \big(0~(\mathrm{mod}~2^{l+1})\big)\big((1+2^{l+1}~(\mathrm{mod}~2^{l+2})) +(1+2^{l+1}~(\mathrm{mod}~2^{l+2}))\big)\\
&\quad +2\big(0~(\mathrm{mod}~2^{l+1})\big)\big(2^{l}~(\mathrm{mod}~2^{l+1})+2^{l}~(\mathrm{mod}~2^{l+1})\big)\\
&\quad +\big(0~(\mathrm{mod}~2^{l+2})\big)\big(2^{l}~(\mathrm{mod}~2^{l+1})+2^{l}~(\mathrm{mod}~2^{l+1})\big)\\
&\equiv \big(0~(\mathrm{mod}~2^{l+1})\big)\big(2~(\mathrm{mod}~2^{l+2})\big)+\big(0~(\mathrm{mod}~2^{l+2})\big)\big(0~(\mathrm{mod}~2^{l+1})\big)\\
&\quad +\big(0~(\mathrm{mod}~2^{l+2})\big)\big(0~(\mathrm{mod}~2^{l+1})\big)\\
&\equiv 0~(\mathrm{mod}~2^{l+2}).
\end{align*}
By differentiating the expression (\ref{Fib:eq variant2}) twice at $s$, we have
\begin{align*}
F_{3\cdot 2^{l+1}+1}''(s)&=2 F_{3\cdot 2^{l}+1}'(s)^2+2 F_{3\cdot 2^{l}+1}(s) F_{3\cdot 2^{l}+1}''(s)\\
&\quad +2 F_{3\cdot 2^{l}}'(s)^2+2 F_{3\cdot 2^{l}}(s) F_{3\cdot 2^{l}}''(s)\\
&\equiv 2\big(2^l (\mathrm{mod}~2^{l+1})\big)^2+2\big(1+2^{l+1} (\mathrm{mod}~2^{l+2})\big) \big(2^{l} (\mathrm{mod}~2^{l+1})\big)\\
&\quad +2\big(0 (\mathrm{mod}~2^{l+1})\big)^2+2\big(0 (\mathrm{mod}~2^{l+2})\big) \big(0 (\mathrm{mod}~2^{l+1})\big)\\
&\equiv 2^{l+1}~(\mathrm{mod}~2^{l+2}),
\end{align*}
which completes the induction step.

2. Let $s$ be an odd number. We use induction on positive integers $l$. Let $l=1$. Then, the second derivative of the Fibonacci polynomial at $s$ becomes that $F_0''(s)=0$, $F_1''(s)=0$ and
\begin{align*}
F_m''(s)\equiv 2 F_{m-1}'(s)+s\cdot F_{m-1}''(s)+F_{m-2}''(s)~(\mathrm{mod}~4)
\end{align*}
for $m\geq 2$. Since the sequence $\{F_m'(s)~(\mathrm{mod}~4)\}_m$ is periodic of period 12 by Lemma \ref{lem:Fib a_n}, the sequence
\[ \{F_m''(s)~(\mathrm{mod}~4)\}_m=\{0,0,0,2,2,2,0,2,2,2,0,0,0\dots\} \]
is periodic of period 12.

For the induction hypothesis, we assume that the sequence $\{F_m''(s)~(\mathrm{mod}~2^{l+1})\}_m$ is periodic of period $3\cdot 2^{l+1}$ for an integer $l\geq1$. We know that $F_0''(s)=0$, $F_1''(s)=0$. By property 1 of Lemma \ref{lem:Fib second derivative}, we have $F_{3\cdot 2^{l+2}}''(s)\equiv 0$ and $F_{3\cdot 2^{l+2}+1}''(s)\equiv 2^{l+2}$~(mod $2^{l+3}$), which implies that $F_{3\cdot 2^{l+2}}''(s)\equiv 0$ and $F_{3\cdot 2^{l+2}+1}''(s)\equiv 0$~(mod $2^{l+2}$). By Lemma \ref{lem:Fib a_n}, we have $F_1'(s)\equiv 1$ and $F_{3\cdot 2^{l+2}+1}'(s)\equiv 1$ (mod $2^{l+2}$).

So, we have that $F_0''(s)=0$, $F_1''(s)=0$, $F_1'(s)=1$, $F_{3\cdot 2^{l+2}}''(s)\equiv 0$, $F_{3\cdot 2^{l+2}+1}''(s)\equiv 0$ and $F_{3\cdot 2^{l+2}+1}'(s)\equiv 1$ (mod $2^{l+2}$). Therefore, by the recurrence relation (\ref{eq:Fibonacci polynomial_recurrence relation_3}), the sequence $\{F_m''(s)~(\mathrm{mod}~2^{l+2})\}_m$ is periodic of period less than or equal to $3\cdot 2^{l+2}$. By induction hypothesis, the period is a multiple of $3\cdot 2^{l+1}$. By property 1 of Lemma \ref{lem:Fib second derivative}, the sequence $\{F_m''(s)~(\mathrm{mod}~2^{l+2})\}_m$ is periodic of period $3\cdot 2^{l+2}$, which completes the induction step.
\end{proof}

\begin{lemma}\label{lem:Fib third derivative}
Let $F_m$ be the Fibonacci polynomial. For any odd number $s$,
\begin{enumerate}
\item $F_{3\cdot 2^{l}}'''(s)\equiv 0$, $F_{3\cdot 2^{l}+1}'''(s)\equiv 0$ and $F_{3\cdot 2^{l}+2}'''(s)\equiv 2^{l}$ (mod $2^{l+1}$) for any positive integer $l$, and

\item the sequence $\{F_m'''(s)~(\mathrm{mod}~2^{l+1})\}_m$ is periodic of period $3\cdot 2^{l+1}$ for any positive integer $l$.
\end{enumerate}
\end{lemma}

\begin{proof}
1. Let $s$ be an odd number. We use induction on positive integers $l$. Let $l=1$. Then,
\begin{align*}
&F_6'''(s)=24+60s^2\equiv 0,\\
&F_7'''(s)=120s+120s^3\equiv 0~\mathrm{and}\\
&F_8'''(s)=60+360s^2+210s^4\equiv 2s^4\equiv 2~(\mathrm{mod}~4).
\end{align*}

For the induction hypothesis, we assume that $F_{3\cdot 2^{l}}'''(s)\equiv 0$, $F_{3\cdot 2^{l}+1}'''(s)\equiv 0$ and $F_{3\cdot 2^{l}+2}'''(s)\equiv 2^{l}$ (mod $2^{l+1}$) for any integer $l\geq 1$. By the recurrence relation (\ref{eq:Fibonacci polynomial_recurrence relation}) and Proposition \ref{prop:Fib periodic}, we have $F_{3\cdot 2^{l}}(s)\equiv 0$, $F_{3\cdot 2^{l}+1}(s)\equiv 1+2^{l+1}$ and $F_{3\cdot 2^{l}-1}(s)\equiv 1+2^{l+1}$ (mod $2^{l+2}$). By the recurrence relation (\ref{eq:Fibonacci polynomial_recurrence relation_2}) and Lemma \ref{lem:Fib a_n}, we have $F_{3\cdot 2^{l}}'(s)\equiv 0$, $F_{3\cdot 2^{l}+1}'(s)\equiv 2^{l}$ and $F_{3\cdot 2^{l}-1}'(s)\equiv 2^{l}$ (mod $2^{l+1}$). By the recurrence relation (\ref{eq:Fibonacci polynomial_recurrence relation_3}) and Lemma \ref{lem:Fib second derivative}, we have $F_{3\cdot 2^{l}}''(s)\equiv 0$, $F_{3\cdot 2^{l}+1}''(s)\equiv 2^{l}$, $F_{3\cdot 2^{l}-1}''(s)\equiv 2^{l}$ (mod $2^{l+1}$) and $F_{3\cdot 2^{l+1}+1}''(s)\equiv 2^{l+1}$ (mod $2^{l+2}$). By the recurrence relation (\ref{eq:Fibonacci polynomial_recurrence relation_4}), we have $F_{3\cdot 2^{l}-1}'''(s)\equiv 0$ (mod $2^{l+1}$). Then, by differentiating the expression (\ref{Fib:eq variant1}) three times at $s$, we have
\begin{align*}
F_{3\cdot 2^{l+1}}'''(s)&=F_{3\cdot 2^{l}}'''(s)\big(F_{3\cdot 2^{l}+1}(s)+F_{3\cdot 2^{l}-1}(s)\big)+3F_{3\cdot 2^{l}}''(s)\big(F_{3\cdot 2^{l}+1}'(s)+F_{3\cdot 2^{l}-1}'(s)\big)\\
&\quad +3F_{3\cdot 2^{l}}'(s)\big(F_{3\cdot 2^{l}+1}''(s)+F_{3\cdot 2^{l}-1}''(s)\big)+F_{3\cdot 2^{l}}(s)\big(F_{3\cdot 2^{l}+1}'''(s)+F_{3\cdot 2^{l}-1}'''(s)\big)\\
&\equiv \big(0~(\mathrm{mod}~2^{l+1})\big)\big((1+2^{l+1}~(\mathrm{mod}~2^{l+2})) +(1+2^{l+1}~(\mathrm{mod}~2^{l+2}))\big)\\
&\quad +3\big(0~(\mathrm{mod}~2^{l+1})\big)\big(2^{l}~(\mathrm{mod}~2^{l+1})+2^{l}~(\mathrm{mod}~2^{l+1})\big)\\
&\quad +3\big(0~(\mathrm{mod}~2^{l+1})\big)\big(2^{l}~(\mathrm{mod}~2^{l+1})+2^{l}~(\mathrm{mod}~2^{l+1})\big)\\
&\quad +\big(0~(\mathrm{mod}~2^{l+2})\big)\big(0~(\mathrm{mod}~2^{l+1})+0~(\mathrm{mod}~2^{l+1})\big)\\
&\equiv \big(0~(\mathrm{mod}~2^{l+1})\big)\big(2~(\mathrm{mod}~2^{l+2})\big)+3\big(0~(\mathrm{mod}~2^{l+1})\big)\big(0~(\mathrm{mod}~2^{l+1})\big)\\
&\quad +3\big(0~(\mathrm{mod}~2^{l+1})\big)\big(0~(\mathrm{mod}~2^{l+1})\big)+\big(0~(\mathrm{mod}~2^{l+2})\big)\big(0~(\mathrm{mod}~2^{l+1})\big)\\
&\equiv 0~(\mathrm{mod}~2^{l+2}).
\end{align*}
By differentiating the expression (\ref{Fib:eq variant2}) three times at $s$, we have
\begin{align*}
&F_{3\cdot 2^{l+1}+1}'''(s)\\
     &\qquad =6 F_{3\cdot 2^{l}+1}''(s) F_{3\cdot 2^{l}+1}'(s)+2 F_{3\cdot 2^{l}+1}(s) F_{3\cdot 2^{l}+1}'''(s)\\
&\qquad\quad +6 F_{3\cdot 2^{l}}''(s) F_{3\cdot 2^{l}}'(s)+2 F_{3\cdot 2^{l}}(s) F_{3\cdot 2^{l}}'''(s)\\
&\qquad\equiv 6\big(2^l (\mathrm{mod}~2^{l+1})\big) \big(2^{l} (\mathrm{mod}~2^{l+1})\big)+2\big(1+2^{l+1} (\mathrm{mod}~2^{l+2})\big) \big(0 (\mathrm{mod}~2^{l+1})\big)\\
&\qquad\quad +6\big(0 (\mathrm{mod}~2^{l+1})\big) \big(0 (\mathrm{mod}~2^{l+1})\big)+2\big(0 (\mathrm{mod}~2^{l+2})\big) \big(0 (\mathrm{mod}~2^{l+1})\big)\\
&\qquad\equiv 0~(\mathrm{mod}~2^{l+2}).
\end{align*}
By the recurrence relation (\ref{eq:Fibonacci polynomial_recurrence relation_4}), we have
\begin{align*}
F_{3\cdot 2^{l+1}+2}'''(s)&=3 F_{3\cdot 2^{l+1}+1}''(s)+s F_{3\cdot 2^{l+1}+1}'''(s)+F_{3\cdot 2^{l+1}}'''(s)\\
&\equiv 3\big(2^{l+1} (\mathrm{mod}~2^{l+2})\big)+s\big(0 (\mathrm{mod}~2^{l+2})\big)+\big(0 (\mathrm{mod}~2^{l+2})\big)\\
&\equiv 2^{l+1}~(\mathrm{mod}~2^{l+2}),
\end{align*}
which completes the induction step.

2. Let $s$ be an odd number. We use induction on positive integers $l$. Let $l=1$. Then, the third derivative of the Fibonacci polynomial at $s$ becomes that $F_0'''(s)=0$, $F_1'''(s)=0$ and
\begin{align*}
F_m'''(s)\equiv 3 F_{m-1}''(s)+s\cdot F_{m-1}'''(s)+F_{m-2}'''(s)~(\mathrm{mod}~4)
\end{align*}
for $m\geq 2$. Since the sequence $\{F_m''(s)~(\mathrm{mod}~4)\}_m$ is periodic of period 12 by Lemma \ref{lem:Fib second derivative}, the sequence
\[ \{F_m'''(s)~(\mathrm{mod}~4)\}_m=\{0,0,0,2,0,0,0,2,0,0,0,0,0\dots\} \]
is periodic of period 12.

For the induction hypothesis, we assume that the sequence $\{F_m'''(s)~(\mathrm{mod}~2^{l+1})\}_m$ is periodic of period $3\cdot 2^{l+1}$ for an integer $l\geq1$. We know that $F_0'''(s)=0$, $F_1'''(s)=0$. By property 1 of Lemma \ref{lem:Fib third derivative}, we have $F_{3\cdot 2^{l+2}}'''(s)\equiv 0$ and $F_{3\cdot 2^{l+2}+1}'''(s)\equiv 0$~(mod $2^{l+3}$), which implies that $F_{3\cdot 2^{l+2}}''(s)\equiv 0$ and $F_{3\cdot 2^{l+2}+1}''(s)\equiv 0$~(mod $2^{l+2}$). By Lemma \ref{lem:Fib second derivative}, we have $F_1''(s)\equiv 0$ and $F_{3\cdot 2^{l+2}+1}''(s)\equiv 0$ (mod $2^{l+2}$).

So, we have that $F_0'''(s)=0$, $F_1'''(s)=0$, $F_1''(s)=0$, $F_{3\cdot 2^{l+2}}'''(s)\equiv 0$, $F_{3\cdot 2^{l+2}+1}'''(s)\equiv 0$ and $F_{3\cdot 2^{l+2}+1}''(s)\equiv 0$ (mod $2^{l+2}$). Therefore, by the recurrence relation (\ref{eq:Fibonacci polynomial_recurrence relation_4}), the sequence $\{F_m'''(s)~(\mathrm{mod}~2^{l+2})\}_m$ is periodic of period less than or equal to $3\cdot 2^{l+2}$. By induction hypothesis, the period is a multiple of $3\cdot 2^{l+1}$. By property 1 of Lemma \ref{lem:Fib third derivative}, the sequence $\{F_m'''(s)~(\mathrm{mod}~2^{l+2})\}_m$ is periodic of period $3\cdot 2^{l+2}$, which completes the induction step.
\end{proof}

We define three functions from $\N\cup \{0\}$ to itself. For $q=c_0+c_1\cdot 2+c_2\cdot 2^2+\dots+c_n\cdot 2^n$ in $\N\cup \{0\}$, define
\begin{align}
&t(q)=\min\{i\geq 0|c_i=c_{i+1}\},\label{special value t(q)}\\
&u_0(q)=\min\{i\geq 1|c_i=0\},~\mathrm{and}\label{special value u_0(q)}\\
&u_1(q)=\min\{i\geq 1|c_i=1\}.\label{special value u_1(q)}
\end{align}

These functions have simple properties.

\begin{lemma}\label{the relation between q and t}
\begin{enumerate}
\item Assume that $q$ is odd and let $t=t(q)$. Then, $t$ is odd if and only if $c_t=c_{t+1}=0$.

\item Assume that $q$ is even and let $t=t(q)$. Then, $t$ is odd if and only if $c_t=c_{t+1}=1$.
\end{enumerate}
\end{lemma}

\begin{proof}
1. Assume that $q$ is odd, then we write $q=1+c_1\cdot 2+c_2\cdot 2^2+\dots+c_n\cdot 2^n$. If $t$ is odd, then $c_t=c_{t-2}=\dots=c_1=0$, which implies that $c_t=c_{t+1}=0$. If $c_t=c_{t+1}=0$, then $0=c_t=c_{t-2}=\dots=c_1$, which implies that $t$ is odd.

2. Assume that $q$ is even, then we write $q=0+c_1\cdot 2+c_2\cdot 2^2+\dots+c_n\cdot 2^n$. If $t$ is odd, then $c_t=c_{t-2}=\dots=c_1=1$, which implies that $c_t=c_{t+1}=1$. If $c_t=c_{t+1}=1$, then $1=c_t=c_{t-2}=\dots=c_1$, which implies that $t$ is odd.
\end{proof}

These functions link the number of cycles and the starting level of strongly growth mysteriously, see Proposition \ref{prop:m=2+12q,q=1}, Theorem \ref{thm:m=2+12q,q=1}, Proposition \ref{prop:m=2+12q,q=2}, Theorem \ref{thm:m=2+12q,q=2}, Proposition \ref{prop:m=6+12q,q=1}, Theorem \ref{thm:m=6+12q,q=1}, Proposition \ref{prop:m=6+12q,q=2}, Theorem \ref{thm:m=6+12q,q=2}, Proposition \ref{prop:m=10+12q,q=2}, Theorem \ref{thm:m=10+12q,q=2}, Proposition \ref{prop:m=10+12q,q=3} and Theorem \ref{thm:m=10+12q,q=3}.

\section{Minimal decompositions for $F_m(x)$ with $m$ odd}

For odd positive numbers $m$, the minimal decompositions of $\Z_2$ for $F_m(x)$ have the following patterns.
\begin{theorem}
Let $m$ be an odd positive number.
\begin{enumerate}
\item If $m\equiv1~\mathrm{or}~2~(\mathrm{mod}~3)$, then $F_m(x)$ has a fixed point in $\Z_2$, say $a$. So, the minimal decomposition of $\Z_2$ for $F_m(x)$ is
\[ \Z_2=\{a\}\bigsqcup\big(\Z_2-\{a\}\big), \]
where $\Z_2-\{a\}$ is the attracting basin of $\{a\}$.

\item If $m\equiv0~(\mathrm{mod}~3)$, then $F_m(x)$ has a cycle of length 2 in $\Z_2$, say $\{a,b\}$. So, the minimal decomposition of $\Z_2$ for $F_m(x)$ is
\[ \Z_2=\{a,b\}\bigsqcup\big(\Z_2-\{a,b\}\big), \]
where $\{a,b\}$ is a cycle of period 2 and $\Z_2-\{a,b\}$ is the attracting basin of $\{a,b\}$.
\end{enumerate}
\end{theorem}

\begin{proof}
1. Suppose that $m$ is odd with $m\equiv\pm1~(\mathrm{mod}~3)$. Then $F_m(1)\equiv 1$ and $a_1(1)=F'_m(1)\equiv0$ (mod 2) by (\ref{a_l}). So, $\{1\}$ is a cycle of length 1 at level 1 which grows tails by Definition \ref{def:movement}. By property 5 of Proposition \ref{prop:grow tails}, $F_m(x)$ has a fixed point $a$ in the clopen set $1+2\Z_2$ with $1+2\Z_2$ lying its attracting basin. For $2\Z_2$, it is obvious that $F_m(2\Z_2)\subseteq 1+2\Z_2$ because $F_m(x)$ has the constant term 1. Therefore, the minimal decomposition of $\Z_2$ for $F_m(x)$ is $\Z_2=\{a\}\bigsqcup\big(\Z_2-\{a\}\big),$
where $\{a\}$ is a fixed point and $\Z_2-\{a\}$ is the attracting basin of $\{a\}$.

2. Suppose that $m$ is odd with $m\equiv0~(\mathrm{mod}~3)$. Then $F_m(0)=1$ and $F_m(1)\equiv 0~(\mathrm{mod}~2)$. In addition, $F'_m(0)=0$ implies that $a_1(0)=F'_m(0)\cdot F'_m(1)=0$ by (\ref{a_l}). So, $\{0,1\}$ is a cycle of length 2 at level 1 which grows tails. By property 5 of Proposition \ref{prop:grow tails}, $F_m(x)$ has a cycle of length 2, say $\{a,b\}$, in the clopen set $\Z_2$. Therefore, the minimal decomposition is $\Z_2=\{a,b\}\bigsqcup\big(\Z_2-\{a,b\}\big),$ where $\{a,b\}$ is a cycle of period two and $\Z_2-\{a,b\}$ is the attracting basin of $\{a,b\}$.
\end{proof}

\section{Minimal decompositions for $F_m(x)$ with $m\equiv 0$ (mod 12)}
From now on, we consider the cases when $m$ is an even positive number. We divide the cases into $m=0+12q,~m=2+12q,~m=4+12q,~m=6+12q,~m=8+12q~\mathrm{and}~m=10+12q$ with nonnegative integers $q$. We express the minimal decomposition of $\Z_2$ for $F_m(x)$ for each case.

We consider the case $m=0+12q$ with positive integers $q$. For formula (\ref{eq:Fibonacci polynomial_even&odd}), we substitute $m=0+12q$, then we obtain the Fibonacci polynomial
\[ F_m(x)=\sum_{j=0}^{6q-1}\tbinom{6q+j}{1+2j}x^{1+2j}. \]

\begin{theorem}
Let $m=0+12q$ with nonnegative integers $q$. Then the minimal decomposition of $\Z_2$ for $F_m(x)$ is
\[ \Z_2=\{0\}\bigsqcup(\Z_2-\{0\}). \]
Here, $\{0\}$ is the set of a fixed point and $\Z_2-\{0\}$ is the attracting basin of $\{a\}$.
\end{theorem}

\begin{proof}
$F_m(0)=0$ and $a_1(0)=F'_m(0)\equiv0~(\mathrm{mod}~2)$. So, $\{0\}$ is a cycle of length 1 at level 1 which grows tails. By property 5 of Proposition \ref{prop:grow tails}, $F_m(x)$ has a fixed point $a$ in the clopen set $2\Z_2$ with $2\Z_2$ lying its attracting basin. For $t\in\Z_2$,
\begin{align*}
F_m(1+2t)=&\tbinom{6q}{1}(1+2t)+\tbinom{1+6q}{3}(1+2t)^3+\dots\\
=&\tbinom{6q}{1}+\tbinom{1+6q}{3}+\dots+1+2s~\mathrm{for~some}~s\in\Z_2\\
=&F_m(1)+2s,
\end{align*}
where $F_m(1)\equiv 0~(\mathrm{mod}~2)$, so $F_m(1+2t)\subseteq 2\Z_2$ and $F_m(1+2\Z_2)\subseteq 2\Z_2$.
\end{proof}

\section{Minimal decompositions for $F_m(x)$ with $m\equiv 2$ (mod 12)}
For $m=2$, the Fibonacci polynomial is $F_2(x)=x$, so all the elements in $\Z_2$ are fixed. Therefore, the minimal decomposition of $\Z_2$ for $F_2(x)$ can be written as $\Z_2$ trivially.

We consider the case $m=2+12q$ with positive integers $q$. For formula (\ref{eq:Fibonacci polynomial_even&odd}), we substitute $m=2+12q$, then we obtain the Fibonacci polynomial
\[ F_m(x)=\sum_{j=0}^{6q}\tbinom{1+6q+j}{1+2j}x^{1+2j}. \]
We divide nonnegative integers $q$ into cases $q=1+2d$ and $q=2+2d$ for nonnegative integers $d$.

\subsection{Case: $m=2+12q$ with $q=1+2d$}

We consider the case $m=2+12q$ with $q=1+2d$ for nonnegative integers $d$. The Fibonacci polynomial becomes
\[ F_m(x)=\sum_{j=0}^{6+12d}\tbinom{7+12d+j}{1+2j}x^{1+2j}. \]

\begin{proposition}\label{prop:m=2+12q,q=1}
Let $m=2+12q$ with $q=1+2d$ for some nonnegative integer $d$. Let $t=t(q)$, as defined in (\ref{special value t(q)}). The Fibonacci polynomial $F_m(x)$ has two types of cycles:
\begin{enumerate}
\item the cycles $\{1+2k\}$ of length 1 strongly grow at level 3 where $k=0,1,2,3$, and

\item the cycles $\{(1+4k)2^n,(2^{t+2}-1-4k)2^n\}$ of length 2 strongly grow at level $n+t+3$ with $k=0,1,\dots,2^{t+1}-1$ and $n\geq1$.
\end{enumerate}
\end{proposition}

\begin{proof}
1. We compute that $F_{2}(1+2k)\equiv 1+2k$ (mod $2^3$). Since the sequence $\{F_m(s)~(\mathrm{mod}~2^3)\}_m$ is periodic of period $12$ for any odd number $s$ by Proposition \ref{prop:Fib periodic}, for $m=14+48d$,
\[ F_{m}(1+2k)\equiv F_{2}(1+2k)\equiv 1+2k~(\mathrm{mod}~2^3). \]
Therefore, $\{1+2k\}$ is a cycle of length 1 at level 3.

Now we compute the quantity $a_3$ for the above cycles, as defined in (\ref{a_l}). We have that $F'_{2}(1+2k)\equiv 1~(\mathrm{mod}~4)$. Since the sequence $\{F'_m(1+2k)~(\mathrm{mod}~4)\}_m$ is periodic of period 12 by Lemma \ref{lem:Fib a_n}, for $m=14+48d$, we obtain the quentity $a_3$,
\[ a_3(1+2k)=F'_m(1+2k)\equiv F'_{2}(1+2k)\equiv 1~(\mathrm{mod}~4). \]

Now we compute the quantity $b_3$ for the above cycles, as defined in (\ref{b_l}). We have that $F_{14}(1+2k)-(1+2k)\equiv 8~(\mathrm{mod}~2^4)$. Since the sequence $\{F_m(1+2k)-(1+2k)~(\mathrm{mod}~2^4)\}_m$ is periodic of period $3\cdot 2^3$ by Proposition \ref{prop:Fib periodic}, for $m=14+48d$, we obtain
\[ F_m(1+2k)-(1+2k)\equiv 8~(\mathrm{mod}~2^4). \]
So, $F_m(1+2k)-(1+2k)=8+2^4Q$ for some $Q\in\Z$. Finally,
\[ b_3(1+2k)=\frac{F_m(1+2k)-(1+2k)}{2^3}=1+2Q\equiv 1~(\mathrm{mod}~2). \]
Therefore, the cycle $\{1+2k\}$ strongly grows at level 3, which completes the proof.
	
2. For the proof, we divide the cycles $\{(1+4k)2^n,(2^{t+2}-1-4k)2^n\}$ into two types of cycles $\{(1+2k)2^n,(2^{t+2}-1-2k)2^n\}$ and $\{(2^{t+2}+1+2k)2^n,(2^{t+3}-1-2k)2^n\}$ where $k=0,1,\dots,2^{t}-1$ and $n\geq1$.

First, assume that $t$ is odd. By Lemma \ref{the relation between q and t}, $c_t=c_{t+1}=0$. Then,
\begin{align*}
q&=1+2^2+2^4+\dots+2^{t-1}+0\cdot 2^t+0\cdot 2^{t+1}+2^{t+2}r\\
&=\frac{2\cdot 2^t-1}{3}+2^{t+2} r
\end{align*}
for some $r\in\Z$. The Fibonacci polynomial becomes\\
$F_m(x)=\sum_{j=0}^{4\cdot 2^t+24\cdot 2^t r-2}\tbinom{4\cdot 2^t+24\cdot 2^t r-1+j}{1+2j}x^{1+2j}$. Consider the form
\[ F_m\big((1+2k)2^n\big)=\sum_{j=0}^{4\cdot 2^t+24\cdot 2^t r-2}\tbinom{4\cdot 2^t+24\cdot 2^t r-1+j}{1+2j}\big((1+2k)2^n\big)^{1+2j}. \]
For $j=1$, the summand has the valuation
\begin{align*}
&\nu_2\Big(\tbinom{4\cdot 2^t+24\cdot 2^t r}{3}\big((1+2k)2^n\big)^{3}\Big)\\
&=\nu_2\big(2^{3n+t+2} \cdot\frac{1}{3}(1+2k)^3 (1+6r)(-1+2^{t+1}+3\cdot 2^{t+2} r)\\
&\quad \cdot(-1+2^{t+2}+3\cdot 2^{t+3} r)\big)\\
&\geq n+t+3.
\end{align*}
For $j\geq 2$,
\begin{multline*}
\tbinom{4\cdot 2^t+24\cdot 2^t r-1+j}{1+2j}\big((1+2k)2^n\big)^{1+2j}\\
=\frac{\big(2^{t+2}(1+6r)-1+j\big)\dots\big(2^{t+2}(1+6r)-1-j\big)}{(1+2j)!}\big((1+2k)2^n\big)^{1+2j}.
\end{multline*}
In the above expression, the numerator has the factor $2^{t+2}(1+6r)$ and the number of even factors in the numerator, excluding the term $2^{t+2}(1+6r)$, is $j-1$. It is known that for $x=\sum_{i=0}^{s}x_i 2^i$ with $x_i=0$ or 1, $\nu_2(x!)=x-\sum_{i=0}^{s}x_i$. Hence, $\nu_2\big((1+2j)!\big)\leq 2j$. Therefore,
\begin{align}
\nu_2\Big(\tbinom{4\cdot 2^t+24\cdot 2^t r-1+j}{1+2j}&\big((1+2k)2^n\big)^{1+2j}\Big)\label{m=14+12q,q=4d coefficient}\\
&\geq(t+2)+(j-1)+(1+2j)n-2j\notag\\
&=1+t+n+j(2n-1)\notag\\
&\geq n+t+3\notag.
\end{align}
So,
\begin{align*}
F_m\big((1+2k)2^n\big)&\equiv\tbinom{4\cdot 2^t+24\cdot 2^t r-1}{1}(1+2k)2^n\\
&\equiv(2^{t+2}-1-2k)2^n~(\mathrm{mod}~2^{n+t+3}),
\end{align*}
\begin{align*}
F_m\big((2^{t+2}-1-2k)2^n\big)&\equiv\tbinom{4\cdot 2^t+24\cdot 2^t r-1}{1}(2^{t+2}-1-2k)2^n\\
&\equiv(1+2k)2^n~(\mathrm{mod}~2^{n+t+3}),
\end{align*}
\begin{align*}
F_m\big((2^{t+2}+1+2k)2^n\big)&\equiv\tbinom{4\cdot 2^t+24\cdot 2^t r-1}{1}(2^{t+2}+1+2k)2^n\\
&\equiv(2^{t+3}-1-2k)2^n~(\mathrm{mod}~2^{n+t+3})~\mathrm{and}
\end{align*}
\begin{align*}
F_m\big((2^{t+3}-1-2k)2^n\big)&\equiv\tbinom{4\cdot 2^t+24\cdot 2^t r-1}{1}(2^{t+3}-1-2k)2^n\\
&\equiv(2^{t+2}+1+2k)2^n~(\mathrm{mod}~2^{n+t+3}).
\end{align*}
Therefore, $\{(1+2k)2^n,(2^{t+2}-1-2k)2^n\}$ and $\{(2^{t+2}+1+2k)2^n,(2^{t+3}-1-2k)2^n\}$ are cycles of length 2 at level $n+t+3$.

Now we compute the quantity $a_{n+t+3}$ for the above cycles, as defined in (\ref{a_l}).
\begin{align*}
a_{n+t+3}\big((1+2k)2^n\big)&=F_m'\big((1+2k)2^n\big)\cdot F_m'\big((2^{t+2}-1-2k)2^n\big)\\
&\equiv \tbinom{4\cdot 2^t+24\cdot 2^t r-1}{1}\cdot \tbinom{4\cdot 2^t+24\cdot 2^t r-1}{1}\\
&\equiv 1~(\mathrm{mod}~4)~\mathrm{and}
\end{align*}
\begin{align*}
a_{n+t+3}\big((2^{t+2}+1+&2k)2^n\big)\\
&=F_m'\big((2^{t+2}+1+2k)2^n\big)\cdot F_m'\big((2^{t+3}-1-2k)2^n\big)\\
&\equiv \tbinom{4\cdot 2^t+24\cdot 2^t r-1}{1}\cdot \tbinom{4\cdot 2^t+24\cdot 2^t r-1}{1}\\
&\equiv 1~(\mathrm{mod}~4).
\end{align*}

Now we compute the quantity $b_{n+t+3}$ for the above cycles, as defined in (\ref{b_l}).
\begin{align*} b_{n+t+3}\big((1+2k)2^n\big)=\frac{F_m^2\big((1+2k)2^n\big)-(1+2k)2^n}{2^{n+t+3}}.
\end{align*}
We show that $\nu_2\Big(F_m^2\big((1+2k)2^n\big)-(1+2k)2^n\Big)=n+t+3$.
\[ F_m\big((1+2k)2^n\big)=\sum_{j=0}^{4\cdot 2^t+24\cdot 2^t r-2}\tbinom{4\cdot 2^t+24\cdot 2^t r-1+j}{1+2j}\big((1+2k)2^n\big)^{1+2j}. \]
For $j=1$, the summand has the valuation
\begin{align*}
\nu_2\Big(\tbinom{4\cdot 2^t+24\cdot 2^t r}{3}&\big((1+2k)2^n\big)^{3}\Big)\\
&=\nu_2\big(2^{3n+t+2} \cdot\frac{1}{3}(1+2k)^3 (1+6r)(-1+2^{t+1}+3\cdot 2^{t+2} r)\\
&\quad \cdot(-1+2^{t+2}+3\cdot 2^{t+3} r)\big)\\
&\geq n+t+4.
\end{align*}
For $j=2$, the summand has the valuation
\begin{align*}
\nu_2&\Big(\tbinom{4\cdot 2^t+24\cdot 2^t r+1}{5}\big((1+2k)2^n\big)^{5}\Big)\\
&=\nu_2\big(2^{5n+t} \cdot\frac{1}{15}(1+2k)^5 (1+6r)(-1+2^{t+1} +3\cdot 2^{t+2} r)\\
&\quad \cdot(-3+2^{t+2}+3\cdot 2^{t+3} r)(-1+2^{t+2}+3\cdot 2^{t+3} r)(1+2^{t+2}+3\cdot 2^{t+3} r)\big)\\
&\geq n+t+4.
\end{align*}
For $j\geq 3$, by expression (\ref{m=14+12q,q=4d coefficient}), $\nu_2\Big(\tbinom{4\cdot 2^t+24\cdot 2^t r-1+j}{1+2j}\big((1+2k)2^n\big)^{1+2j}\Big)
\geq 1+t+n+j(2n-1) \geq n+t+4.$
So, we obtain
\begin{align*}
F_m\big((1+2k)2^n\big)&\equiv \tbinom{4\cdot 2^t+24\cdot 2^t r-1}{1}(1+2k)2^n\\
&\equiv (1+2k)(-1+2^{t+2}+3\cdot 2^{t+3} r)2^n~(\mathrm{mod}~2^{n+t+4}).
\end{align*}
Now we compute the following.
\begin{align*}
F_m^2&\big((1+2k)2^n\big)-(1+2k)2^n\\
&\equiv F_m\big((1+2k)(-1+2^{t+2}+3\cdot 2^{t+3} r)2^n\big)-(1+2k)2^n\\
&\equiv \sum_{j=0}^{4\cdot 2^t+24\cdot 2^t r-2}\tbinom{4\cdot 2^t+24\cdot 2^t r-1+j}{1+2j}\big((1+2k)(-1+2^{t+2}+3\cdot 2^{t+3} r)2^n\big)^{1+2j}\\
&\quad-(1+2k)2^n~(\mathrm{mod}~2^{n+t+4}).
\end{align*}
For $j=0$, with the term $(1+2k)2^n$, the summand has the valuation
\begin{align*}
\nu_2&\big(\tbinom{4\cdot 2^t+24\cdot 2^t r-1}{1}(1+2k)(-1+2^{t+2}+3\cdot 2^{t+3} r)2^n-(1+2k)2^n\big)\\
&=\nu_2\big(2^{n+t+3} (1+2k)(1+6r)(-1+2^{t+1}+3\cdot 2^{t+2} r)\big)\\
&=n+t+3.
\end{align*}
For $j=1$, the summand has the valuation
\begin{align*}
\nu_2\Big(&\tbinom{4\cdot 2^t+24\cdot 2^t r}{3}\big((1+2k)(-1+2^{t+2}+3\cdot 2^{t+3} r) 2^n\big)^3\Big)\\
&=\nu_2\big(2^{3n+t+2} \cdot\frac{1}{3}(1+2k)^3 (1+6r)(-1+2^{t+1}+3\cdot 2^{t+2} r)\\
&\quad \cdot(-1+2^{t+2}+3\cdot 2^{t+3} r)^4\big)\\
&\geq n+t+4.
\end{align*}
For $j=2$, the summand has the valuation
\begin{align*}
\nu_2\Big(&\tbinom{4\cdot 2^t+24\cdot 2^t r+1}{5}\big((1+2k)(-1+2^{t+2}+3\cdot 2^{t+3} r)2^n\big)^{5}\Big)\\
&=\nu_2\big(2^{5n+t} \frac{1}{15}(1+2k)^5 (1+6r)(-1+2^{t+1}+3\cdot 2^{t+2} r)\\
&\quad \cdot(-3+2^{t+2}+3\cdot 2^{t+3} r)(-1+2^{t+2}+3\cdot 2^{t+3} r)^6\\
&\quad \cdot(1+2^{t+2}+3\cdot 2^{t+3} r)\big)\\
&\geq n+t+4.
\end{align*}
For $j\geq 3$, by expression (\ref{m=14+12q,q=4d coefficient}),
\begin{align*}
\nu_2\Big(\tbinom{4\cdot 2^t+24\cdot 2^t r-1+j}{1+2j}&\big((1+2k)(-1+2^{t+2}+3\cdot 2^{t+3} r)2^n\big)^{1+2j}\Big)\\
&\geq 1+t+n+j(2n-1)\\
&\geq n+t+4.
\end{align*}	
Combining these, we obtain
\[ \nu_2\Big(F_m^2\big((1+2k)2^n\big)-(1+2k)2^n\Big)=n+t+3. \]
Hence, $b_{n+t+3}\big((1+2k)2^n\big)\equiv 1~(\mathrm{mod}~2)$.

Because $F_m^2(x)$ is an odd function and $(2^{t+3}-1-2k)2^n\equiv -(1+2k)2^n$ $(\mathrm{mod}~2^{n+t+3})$,
\begin{align*}
b_{n+t+3}\big((2^{t+3}-1-2k)2^n\big)&=b_{n+t+3}\big(-(1+2k)2^n\big)\\
&=\frac{F_m^2\big(-(1+2k)2^n\big)-\big(-(1+2k)2^n\big)}{2^{n+t+3}}\\
&=-\frac{F_m^2\big((1+2k)2^n\big)-\big((1+2k)2^n\big)}{2^{n+t+3}}\\
&=-b_{n+t+3}\big((1+2k)2^n\big)\\
&\equiv 1~(\mathrm{mod}~2).
\end{align*}

Second, assume that $t$ is even. By Lemma \ref{the relation between q and t}, $c_t=c_{t+1}=1$. Then,
\begin{align*}
q&=1+2^2+2^4+2^6+\dots+2^t+2^{t+1}+2^{t+2}r\\
&=\frac{10\cdot 2^t-1}{3}+2^{t+2} r
\end{align*}
for some $r\in\Z$. The Fibonacci polynomial becomes\\
$F_m(x)=\sum_{j=0}^{20\cdot 2^t+24\cdot 2^t r-2}\tbinom{20\cdot 2^t+24\cdot 2^t r-1+j}{1+2j}x^{1+2j}$. Consider the form
\[ F_m\big((1+2k)2^n\big)=\sum_{j=0}^{20\cdot 2^t+24\cdot 2^t r-2}\tbinom{20\cdot 2^t+24\cdot 2^t r-1+j}{1+2j}\big((1+2k)2^n\big)^{1+2j}. \]
For $j=1$, the summand has the valuation
\begin{align*}
\nu_2\Big(\tbinom{20\cdot 2^t+24\cdot 2^t r}{3}&\big((1+2k)2^n\big)^{3}\Big)\\
&=\nu_2\big(2^{3n+t+2} \cdot\frac{1}{3}(1+2k)^3 (5+6r)\\
&\quad \cdot(-1+5\cdot 2^{t+1}+3\cdot 2^{t+2} r)(-1+5\cdot 2^{t+2}+3\cdot 2^{t+3} r)\big)\\
&\geq n+t+3.
\end{align*}
For $j\geq 2$,
\begin{multline*}
\tbinom{20\cdot 2^t+24\cdot 2^t r-1+j}{1+2j}\big((1+2k)2^n\big)^{1+2j}\\
=\frac{\big(2^{t+2}(5+6r)-1+j\big)\dots\big(2^{t+2}(5+6r)-1-j\big)}{(1+2j)!}\big((1+2k)2^n\big)^{1+2j}.
\end{multline*}
In the above expression, the numerator has the factor $2^{t+2}(5+6r)$ and the number of even factors in the numerator, excluding the term $2^{t+2}(5+6r)$, is $j-1$. It is known that for $x=\sum_{i=0}^{s}x_i 2^i$ with $x_i=0$ or 1, $\nu_2(x!)=x-\sum_{i=0}^{s}x_i$. Hence, $\nu_2\big((1+2j)!\big)\leq 2j$. Therefore,
\begin{align}
\nu_2\Big(\tbinom{20\cdot 2^t+24\cdot 2^t r-1+j}{1+2j}&\big((1+2k)2^n\big)^{1+2j}\Big)\label{m=14+12q,q=4d coefficient_2}\\
&\geq(t+2)+(j-1)+(1+2j)n-2j\notag\\
&=1+t+n+j(2n-1)\notag\\
&\geq n+t+3\notag.
\end{align}
So,
\begin{align*}
F_m\big((1+2k)2^n\big)\equiv&\tbinom{20\cdot 2^t+24\cdot 2^t r-1}{1}(1+2k)2^n\\
\equiv&(2^{t+2}-1-2k)2^n~(\mathrm{mod}~2^{n+t+3}),
\end{align*}
\begin{align*}
F_m\big((2^{t+2}-1-2k)2^n\big)\equiv&\tbinom{20\cdot 2^t+24\cdot 2^t r-1}{1}(2^{t+2}-1-2k)2^n\\
\equiv&(1+2k)2^n~(\mathrm{mod}~2^{n+t+3}),
\end{align*}
\begin{align*}
F_m\big((2^{t+2}+1+2k)2^n\big)\equiv&\tbinom{20\cdot 2^t+24\cdot 2^t r-1}{1}(2^{t+2}+1+2k)2^n\\
\equiv&(2^{t+3}-1-2k)2^n~(\mathrm{mod}~2^{n+t+3})~\mathrm{and}
\end{align*}
\begin{align*}
F_m\big((2^{t+3}-1-2k)2^n\big)\equiv&\tbinom{20\cdot 2^t+24\cdot 2^t r-1}{1}(2^{t+3}-1-2k)2^n\\
\equiv&(2^{t+2}+1+2k)2^n~(\mathrm{mod}~2^{n+t+3}).
\end{align*}
Therefore, $\{(1+2k)2^n,(2^{t+2}-1-2k)2^n\}$ and $\{(2^{t+2}+1+2k)2^n,(2^{t+3}-1-2k)2^n\}$ are cycles of length 2 at level $n+t+3$.

Now we compute the quantity $a_{n+t+3}$ for the above cycles, as defined in (\ref{a_l}).
\begin{align*}
a_{n+t+3}\big((1+2k)2^n\big)&=F_m'\big((1+2k)2^n\big)\cdot F_m'\big((2^{t+2}-1-2k)2^n\big)\\
&\equiv \tbinom{20\cdot 2^t+24\cdot 2^t r-1}{1}\cdot\tbinom{20\cdot 2^t+24\cdot 2^t r-1}{1}\\
&\equiv 1~(\mathrm{mod}~4)~\mathrm{and}
\end{align*}
\begin{align*}
a_{n+t+3}\big((2^{t+2}+1+2k)2^n\big)
&=F_m'\big((2^{t+2}+1+2k)2^n\big)\cdot F_m'\big((2^{t+3}-1-2k)2^n\big)\\
&\equiv \tbinom{20\cdot 2^t+24\cdot 2^t r-1}{1}\cdot\tbinom{20\cdot 2^t+24\cdot 2^t r-1}{1}\\
&\equiv 1~(\mathrm{mod}~4).
\end{align*}

Now we compute the quantity $b_{n+t+3}$ for the above cycles, which is defined in (\ref{b_l}).
\begin{align*} b_{n+t+3}\big((1+2k)2^n\big)=\frac{F_m^2\big((1+2k)2^n\big)-(1+2k)2^n}{2^{n+t+3}}.
\end{align*}
We show that $\nu_2\Big(F_m^2\big((1+2k)2^n\big)-(1+2k)2^n\Big)=n+t+3$.
\[ F_m\big((1+2k)2^n\big)=\sum_{j=0}^{20\cdot 2^t+24\cdot 2^t r-2}\tbinom{20\cdot 2^t+24\cdot 2^t r-1+j}{1+2j}\big((1+2k)2^n\big)^{1+2j}. \]
For $j=1$, the summand has the valuation
\begin{align*}
\nu_2\Big(\tbinom{20\cdot 2^t+24\cdot 2^t r}{3}&\big((1+2k)2^n\big)^{3}\Big)\\
&=\nu_2\big(2^{3n+t+2} \cdot\frac{1}{3}(1+2k)^3 (5+6r)\\
&\quad \cdot(-1+5\cdot 2^{t+1}+3\cdot 2^{t+2} r)(-1+5\cdot 2^{t+2}+3\cdot 2^{t+3} r)\big)\\
&\geq n+t+4.
\end{align*}
For $j=2$, the summand has the valuation
\begin{align*}
\nu_2&\Big(\tbinom{20\cdot 2^t+24\cdot 2^t r+1}{5}\big((1+2k)2^n\big)^{5}\Big)\\
&=\nu_2\big(2^{5n+t} \cdot\frac{1}{15}(1+2k)^5 (5+6r)(-1+5\cdot 2^{t+1}+3\cdot 2^{t+2} r)\\
&\quad \cdot(-3+5\cdot 2^{t+2}+3\cdot 2^{t+3} r)(-1+5\cdot 2^{t+2}+3\cdot 2^{t+3} r)\\
&\quad \cdot(1+5\cdot 2^{t+2}+3\cdot 2^{t+3} r)\big)\\
&\geq n+t+4.
\end{align*}
For $j\geq 3$, by expression (\ref{m=14+12q,q=4d coefficient_2}), $\nu_2\Big(\tbinom{20\cdot 2^t+24\cdot 2^t r-1+j}{1+2j}\big((1+2k)2^n\big)^{1+2j}\Big)
\geq 1+t+n+j(2n-1) \geq n+t+4.$
Therefore, we obtain
\begin{align*}
F_m\big((1+2k)2^n\big)&\equiv \tbinom{20\cdot 2^t+24\cdot 2^t r-1}{1}(1+2k)2^n\\
&\equiv (1+2k)(-1+5\cdot 2^{t+2}+3\cdot 2^{t+3} r)2^n~(\mathrm{mod}~2^{n+t+4}).
\end{align*}
Now we compute the following.
\begin{align*}
F_m^2&\big((1+2k)2^n\big)-(1+2k)2^n\\
&\equiv F_m\big((1+2k)(-1+5\cdot 2^{t+2}+3\cdot 2^{t+3} r)2^n\big)-(1+2k)2^n\\
&\equiv \sum_{j=0}^{20\cdot 2^t+24\cdot 2^t r-2}\tbinom{20\cdot 2^t+24\cdot 2^t r-1+j}{1+2j}\big((1+2k)(-1+5\cdot 2^{t+2}+3\cdot 2^{t+3} r)2^n\big)^{1+2j}\\
&\quad-(1+2k)2^n~(\mathrm{mod}~2^{n+t+4}).
\end{align*}
For $j=0$, with the term $(1+2k)2^n$, the summand has the valuation
\begin{align*}
\nu_2&\big(\tbinom{20\cdot 2^t+24\cdot 2^t r-1}{1}(1+2k)(-1+5\cdot 2^{t+2}+3\cdot 2^{t+3} r)2^n-(1+2k)2^n\big)\\
&=\nu_2\big(2^{n+t+3} (1+2k)(5+6r)(-1+5\cdot 2^{t+1}+3\cdot 2^{t+2} r)\big)\\
&=n+t+3.
\end{align*}
For $j=1$, the summand has the valuation
\begin{align*}
\nu_2\Big(&\tbinom{20\cdot 2^t+24\cdot 2^t r}{3}\big((1+2k)(-1+5\cdot 2^{t+2}+3\cdot 2^{t+3} r)2^n\big)^3\Big)\\
&=\nu_2\big(2^{3n+t+2} \cdot\frac{1}{3}(1+2k)^3 (5+6r)(-1+5\cdot 2^{t+1}+3\cdot 2^{t+2} r)\\
&\quad \cdot(-1+5\cdot 2^{t+2}+3\cdot 2^{t+3} r)^4\big)\\
&\geq n+t+4.
\end{align*}
For $j=2$, the summand has the valuation
\begin{align*}
\nu_2\Big(&\tbinom{20\cdot 2^t+24\cdot 2^t r+1}{5}\big((1+2k)(-1+5\cdot 2^{t+2}+3\cdot 2^{t+3} r)2^n\big)^{5}\Big)\\
&=\nu_2\big(2^{5n+t} \cdot\frac{1}{15}(1+2k)^5 (5+6r)(-1+5\cdot 2^{t+1}+3\cdot 2^{t+2} r)\\
&\quad \cdot(-3+5\cdot 2^{t+2}+3\cdot 2^{t+3} r)(-1+5\cdot 2^{t+2}+3\cdot 2^{t+3} r)^6\\
&\quad \cdot(1+5\cdot 2^{t+2}+3\cdot 2^{t+3} r)\big)\\
&\geq n+t+4.
\end{align*}
For $j\geq 3$, by expression (\ref{m=14+12q,q=4d coefficient_2}),
\begin{align*}
\nu_2\Big(\tbinom{20\cdot 2^t+24\cdot 2^t r-1+j}{1+2j}&\big((1+2k)(-1+5\cdot 2^{t+2}+3\cdot 2^{t+3} r)2^n\big)^{1+2j}\Big)\\
&\geq 1+t+n+j(2n-1)\\
&\geq n+t+4.
\end{align*}	
Combining these, we obtain
\[ \nu_2\Big(F_m^2\big((1+2k)2^n\big)-(1+2k)2^n\Big)=n+t+3. \]
Hence, $b_{n+t+3}\big((1+2k)2^n\big)\equiv 1~(\mathrm{mod}~2)$.

Because $F_m^2(x)$ is an odd function and $(2^{t+3}-1-2k)2^n\equiv -(1+2k)2^n$ $(\mathrm{mod}~2^{n+t+3})$,
\begin{align*}
b_{n+t+3}\big((2^{t+3}-1-2k)2^n\big)&=b_{n+t+3}\big(-(1+2k)2^n\big)\\
&=\frac{F_m^2\big(-(1+2k)2^n\big)-\big(-(1+2k)2^n\big)}{2^{n+t+3}}\\
&=-\frac{F_m^2\big((1+2k)2^n\big)-\big((1+2k)2^n\big)}{2^{n+t+3}}\\
&=-b_{n+t+3}\big((1+2k)2^n\big)\\
&\equiv 1~(\mathrm{mod}~2).
\end{align*}
Therefore, the cycles $\{(1+2k)2^n,(2^{t+2}-1-2k)2^n\}$ and $\{(2^{t+2}+1+2k)2^n,(2^{t+3}-1-2k)2^n\}$ strongly grow at level $n+t+3$, which completes the proof.
\end{proof}

By Proposition \ref{prop:m=2+12q,q=1}, we conclude that the following is true.
\begin{theorem}\label{thm:m=2+12q,q=1}
The minimal decomposition of $\Z_2$ for $F_m(x)$ with $m=2+12q$ and $q=1+2d$ for nonnegative integers $d$ is
\[ \Z_2=\{0\}\bigsqcup\Bigl(\big(\bigcup_{k=0}^3 A_k\big)\cup\big(\bigcup_{n\geq 1}\bigcup_{k=0}^{2^{t+1}-1} M_{n,k}\big)\Bigr), \]
where 
\begin{align*}
	& A_k=1+2k+2^3\Z_2 \ \text{and} \\
	& M_{n,k}=\{(1+4k)2^n+2^{n+t+3}\Z_2\}\cup \{ (2^{t+2}-1-4k)2^n+2^{n+t+3}\Z_2\}
\end{align*} 
with $t=t(q)$. Here, $\{0\}$ is the set of a fixed point and $A_k$'s and $M_{n,k}$'s are the minimal components.
\end{theorem}

\subsection{Case: $m=2+12q$ with $q=2+2d$}

We consider the case $m=2+12q$ with $q=2+2d$ for nonnegative integers $d$. Let $u=u_1(q)$, as defined in (\ref{special value u_1(q)}). Then, $q=2^{u}+2^{u+1} r$ for some $r\in\Z$. Then, $m=12\cdot 2^{u}+2+24\cdot 2^{u} r$ and the Fibonacci polynomial becomes
\[ F_m(x)=\sum_{j=0}^{6\cdot 2^{u}+12\cdot 2^{u} r}\tbinom{6\cdot 2^{u}+12\cdot 2^{u} r+1+j}{1+2j}x^{1+2j}. \]

\begin{proposition}\label{prop:m=2+12q,q=2}
Let $m=2+12q$ with $q=2+2d$ for some nonnegative integer $d$. Let $u=u_1(q)$, as defined in (\ref{special value u_1(q)}). The Fibonacci polynomial $F_m(x)$ has two types of cycles:
\begin{enumerate}
\item the cycles $\{1+2k\}$ of length 1 strongly grow at level $u+3$ where $k=0,1,\dots,2^{u+2}-1$, and

\item the cycles $\{(1+2k)2^n\}$ of length 1 strongly grow at level $n+u+1$ with $k=0,1,\dots,2^{u}-1$ and $n\geq1$.
\end{enumerate}
\end{proposition}

\begin{proof}
1. We compute that $F_{2}(1+2k)=1+2k$. Since the sequence $\{F_m(1+2k)~(\mathrm{mod}~2^{u+3})\}_m$ is periodic of period $12\cdot 2^{u}$ by Proposition \ref{prop:Fib periodic}, for $m=12\cdot 2^{u}+2+24\cdot 2^{u} r$,
\begin{align*}
F_m(1+2k)&=F_{12\cdot 2^{u}+24\cdot 2^{u} r+2}(1+2k)\\
&\equiv F_{2}(1+2k)\equiv 1+2k~(\mathrm{mod}~2^{u+3}).
\end{align*}
Therefore, $\{1+2k\}$ is a cycle of length 1 at level $u+3$.

Now we compute the quantity $a_{u+3}$ for the above cycles, as defined in (\ref{a_l}). We have that $F_{2}'(1+2k)\equiv 1~(\mathrm{mod}~4)$. Since the sequence $\{F_m'(1+2k)~(\mathrm{mod}~4)\}_m$ is periodic of period 12 by Lemma \ref{lem:Fib a_n}, for $m=12\cdot 2^{u}+2+24\cdot 2^{u} r$, we obtain the quantity $a_{u+3}$,
\[ a_{u+3}(1+2k)=F_m'(1+2k)\equiv F_{2}'(1+2k)\equiv 1~(\mathrm{mod}~4). \]

Now we compute the quantity $b_{u+3}$ for the above cycles, as defined in (\ref{b_l}). By Proposition \ref{prop:Fib periodic} and the recurrence relation (\ref{eq:Fibonacci polynomial_recurrence relation}), we have $F_{12\cdot 2^{u}+2}(1+2k)\equiv 1+2k+2^{u+3}~(\mathrm{mod}~2^{u+4})$. Since $\{F_m(1+2k)~(\mathrm{mod}~2^{u+4})\}_m$ is periodic of period $24\cdot 2^{u}$, for $m=12\cdot 2^{u}+2+24\cdot 2^{u} r$, we obtain
\[ F_m(1+2k)\equiv F_{12\cdot 2^{u}+2}(1+2k)\equiv 1+2k+2^{u+3}~(\mathrm{mod}~2^{u+4}). \]
So, $F_m(1+2k)=1+2k+2^{u+3}+2^{u+4}Q$ for some $Q\in\Z$. Finally,
\[ b_{u+3}=\frac{F_m(1+2k)-(1+2k)}{2^{u+3}}=1+2Q\equiv 1~(\mathrm{mod}~2). \]
Therefore, the cycle $\{1+2k\}$ strongly grows at level $u+3$, which completes the proof.

2. Consider the form
\begin{align*}
F_m\big((1+2k)2^n\big)=\sum_{j=0}^{6\cdot 2^{u}+12\cdot 2^{u} r}\tbinom{6\cdot 2^{u}+12\cdot 2^{u} r+1+j}{1+2j}\big((1+2k)2^n\big)^{1+2j}.
\end{align*}
For $j\geq 1$,
\begin{multline*}
\tbinom{6\cdot 2^{u}+12\cdot 2^{u} r+1+j}{1+2j}\big((1+2k)2^n\big)^{1+2j}\\
=\frac{(2^{u+1}(3+6r)+1+j)\dots(2^{u+1}(3+6r)+1-j)}{(1+2j)!}\big((1+2k)2^n\big)^{1+2j}.
\end{multline*}
In the above expression, the numerator has the factor $2^{u+1}(3+6r)$ and the number of even factors in the numerator, excluding the term $2^{u+1}(3+6r)$, is $j-1$. It is known that for $x=\sum_{i=0}^{s}x_i 2^i$ with $x_i=0$ or 1, $\nu_2(x!)=x-\sum_{i=0}^{s}x_i$. Hence, $\nu_2\big((1+2j)!\big)\leq 2j$. Therefore,
\begin{align}
\nu_2\Big(\tbinom{6\cdot 2^{u}+12\cdot 2^{u} r+1+j}{1+2j}&\big((1+2k)2^n\big)^{1+2j}\Big)\label{m=2+12q,q=2+2d coefficient}\\
&\geq(u+1)+(j-1)+(1+2j)n-2j\notag\\
&=n+u+j(2n-1)\notag\\
&\geq n+u+1.\notag
\end{align}
So,
\begin{align*}
F_m\big((1+2k)2^n\big)\equiv&\tbinom{6\cdot 2^{u}+12\cdot 2^{u} r+1}{1}(1+2k)2^n\\
\equiv&(1+2k)2^n~(\mathrm{mod}~2^{n+u+1}).
\end{align*}
Therefore, $\{(1+2k)2^n\}$ is a cycle of length 1 at level $n+u+1$.

Now we compute the quantity $a_{n+u+1}$ for the above cycles, as defined in (\ref{a_l}).
\begin{align*}
a_{n+u+1}\big((1+2k) 2^n\big)&=F'_m\big((1+2k) 2^n\big)\equiv \tbinom{6\cdot 2^{u}+12\cdot 2^{u} r+1}{1}\equiv 1~(\mathrm{mod}~4).
\end{align*}
Now we compute the quantity $b_{n+u+1}$ for the above cycles, as defined in (\ref{b_l}).
\[ b_{n+u+1}\big((1+2k) 2^n\big)=\frac{F_m\big((1+2k)2^n\big)-(1+2k)2^n}{2^{n+u+1}}. \]
We show that $\nu_2\Big(F_m\big((1+2k)2^n\big)-(1+2k)2^n\Big)=n+u+1$.
\begin{multline*}
F_m\big((1+2k)2^n\big)-(1+2k)2^n\\
=\sum_{j=0}^{6\cdot 2^{u}+12\cdot 2^{u} r}\tbinom{6\cdot 2^{u}+12\cdot 2^{u} r+1+j}{1+2j}\big((1+2k)2^n\big)^{1+2j}-(1+2k)2^n.
\end{multline*}
For $j=0$, with the term $(1+2k)2^n$, the summand has the valuation
\begin{align*}
\nu_2\big(\tbinom{6\cdot 2^{u}+12\cdot 2^{u} r+1}{1}&(1+2k)2^n-(1+2k)2^n\big)\\
&=\nu_2\big(2^{n+u+1} 3(1+2k)(1+2r)\big)\\
&=n+u+1.
\end{align*}
For $j=1$, the summand has the valuation
\begin{align*}
\nu_2\Big(\tbinom{6\cdot 2^{u}+12\cdot 2^{u} r+2}{3}&\big((1+2k)2^n\big)^3\Big)\\
&=\nu_2\big(2^{3n+u+1} (1+2k)^3 (1+2r)(1+3\cdot 2^{u}+6\cdot 2^{u} r)\\
&\quad \cdot(1+6\cdot 2^{u}+12\cdot 2^{u} r)\big)\\
&\geq n+u+2.
\end{align*}
For $j\geq2$, by expression (\ref{m=2+12q,q=2+2d coefficient}),
\begin{align*}
\nu_2\Big(\tbinom{6\cdot 2^{u}+12\cdot 2^{u} r+1+j}{1+2j}\big((1+2k)2^n\big)^{1+2j}\Big)&\geq n+u+j(2n-1)\\
&\geq n+u+2.
\end{align*}
Combining these, we obtain
\[ \nu_2\Big(F_m\big((1+2k)2^n\big)-(1+2k)2^n\Big)=n+u+1. \]
Hence, $b_{n+u+1}\big((1+2k)2^n\big)\equiv 1~(\mathrm{mod}~2)$. Therefore, the cycle $\{(1+2k)2^n\}$ strongly grows at level $n+u+1$, which completes the proof.
\end{proof}

By Proposition \ref{prop:m=2+12q,q=2}, we conclude that the following is true.
\begin{theorem}\label{thm:m=2+12q,q=2}
The minimal decomposition of $\Z_2$ for $F_m(x)$ with $m=2+12q$ and $q=2+2d$ for nonnegative integers $d$ is
\[ \Z_2=\{0\}\bigsqcup\Bigl(\big(\bigcup_{k=0}^{2^{u+2}-1} A_k\big)\cup\big(\bigcup_{n\geq 1}\bigcup_{k=0}^{2^{u}-1} M_{n,k}\big)\Bigr), \]
where 
\[ A_k=1+2k+2^{u+3}\Z_2 \ \text{and} \ M_{n,k}=(1+2k)2^n+2^{n+u+1}\Z_2 \]
with $u=u_1(q)$. Here, $\{0\}$ is the set of a fixed point and $A_k$'s and $M_{n,k}$'s are the minimal components.
\end{theorem}

\section{Minimal decompositions for $F_m(x)$ with $m\equiv 4$ (mod 12)}

We consider the case $m=4+12q$ with nonnegative integers $q$. The Fibonacci polynomial becomes
\[ F_m(x)=\sum_{j=0}^{1+6q}\tbinom{2+6q+j}{1+2j}x^{1+2j}. \]

\begin{proposition}\label{m=4+12q even numbers}
Let $m=4+12q$ for some nonnegative integer $q$. The Fibonacci polynomial $F_m(x)$ has a fixed point 0 in the clopen set $2\Z_2$ with $2\Z_2$ lying in its attracting basin.
\end{proposition}

\begin{proof}
Since $F_m(0)=0$, $\{0\}$ is a cycle of length 1, i.e., a fixed point at any level. We check that $a_1(0)=F_m'(0)\equiv 0~(\mathrm{mod}~2)$ by (\ref{a_l}), so $\{0\}$ grows tails at level 1 by Definition \ref{def:movement}. By property 5 of Proposition \ref{prop:grow tails}, $F_m(x)$ has a fixed point 0 in the clopen set $2\Z_2$ with $2\Z_2$ lying its attracting basin.
\end{proof}

Now, we consider elements in the set $1+2\Z_2$, which is the complement of $2\Z_2$ in $\Z_2$. For this purpose, we divide nonnegative integers $q$ into the cases that $q$ is $1+2d,\ 4d,\ 6+8d,\ 10+16d,\ 18+32d,\ 34+64d $, and $2+64d $ for nonnegative integers $d$.

\subsection{Case: $m=4+12q$ with $q=1+2d$}

We consider the case $m=4+12q$ with $q=1+2d$ for nonnegative integers $d$. The Fibonacci polynomial becomes
\[ F_m(x)=\sum_{j=0}^{7+12d}\tbinom{8+12d+j}{1+2j}x^{1+2j}. \]

\begin{proposition}\label{prop:m=4+12q,q=1}
Let $m=4+12q$ with $q=1+2d$ for some nonnegative integer $d$. The Fibonacci polynomial $F_m(x)$ has cycles $\{1+4k,11+4k\}$ of length 2 which strongly grow at level 4 where $k=0,1,2$ and 3.
\end{proposition}

\begin{proof}
We compute that $F_{16}(1+4k)\equiv 11+4k$ and $F_{16}(11+4k)\equiv 1+4k~(\mathrm{mod}~2^4)$. Since the sequence $\{F_m(s)~(\mathrm{mod}~2^4)\}_m$ is periodic of period 24 for any odd number $s$ by Proposition \ref{prop:Fib periodic}, for $m=16+24d$,
\begin{align}
&F_{m}(1+4k)\equiv F_{16}(1+4k)\equiv 11+4k\ (\mathrm{mod}~2^4)~~\mathrm{and}\label{thm:m=4+12q,q=1_1}\\
&F_{m}(11+4k)\equiv F_{16}(11+4k)\equiv 1+4k~(\mathrm{mod}~2^4).\label{thm:m=4+12q,q=1_2}
\end{align}
Therefore, $\{1+4k,11+4k\}$ is a cycle of length 2 at level 4.

Now we compute the quantity $a_4$ for the above cycles, as defined in (\ref{a_l}). We have that $F'_{4}(1+4k)\cdot F'_{4}(11+4k)\equiv 1~(\mathrm{mod}~4)$. Since the sequence $\{F'_m(1+4k)\cdot F'_m(11+4k)~(\mathrm{mod}~4)\}_m$ is periodic of period 6 by Lemma \ref{lem:Fib a_n}, for $m=16+24d$, we obtain the quantity $a_4$,
\begin{align*}
a_4(1+4k)&=F'_m(1+4k)\cdot F'_m(11+4k)\equiv F'_{4}(1+4k)\cdot F'_{4}(11+4k)\\
&\equiv 1~(\mathrm{mod}~4).
\end{align*}
	
Now we compute the quantity $b_4$ for the above cycles, as defined in (\ref{b_l}). From the expressions (\ref{thm:m=4+12q,q=1_1}) and (\ref{thm:m=4+12q,q=1_2}), we write
\begin{align}\label{m=4+12q and q=1+2d A,B coeff}
F_m(1+4k)= 11+4k+2^4A~\mathrm{and}~F_m(11+4k)= 1+4k+2^4B
\end{align}
for some $A,B\in\Z$. To obtain the value $b_4$, the following expression is useful. For integers $j\geq 0$, the following holds
\begin{equation}\label{b_l with m=4+12q and q=1+2d}
(11+4k)^{1+2j}\equiv (1+4k)^{1+2j}+2+4(1+2j)(2j)+8(-1)^j~(\mathrm{mod}~2^5).
\end{equation}
We check that the three sequences $\{(1+4k)^{1+2j}\}_j$, $\{(-1)^j\}_j$ and $\{4(1+2j)(2j)\}_j$ modulo $2^5$ are periodic of period 4, 2 and 4, respectively, and a direct computation shows that the expression holds for $j=0,1,2$ and 3. Therefore, the expression holds for every $j\geq 0$. Using the expression (\ref{b_l with m=4+12q and q=1+2d}) and the equalities $F_m(1)=\sum_{j=0}^{7+12d}\tbinom{8+12d+j}{1+2j}$, $\frac{F_m(i)}{i}=\sum_{j=0}^{7+12d}\tbinom{8+12d+j}{1+2j}(-1)^j$ and $F_m''(1)=\sum_{j=0}^{7+12d}\tbinom{8+12d+j}{1+2j}(1+2j)(2j)$, we obtain
\begin{align*}
F_m(11+4k)&=\sum_{j=0}^{7+12d}\tbinom{8+12d+j}{1+2j}(11+4k)^{1+2j}\\
&\equiv\sum_{j=0}^{7+12d}\tbinom{8+12d+j}{1+2j}\big((1+4k)^{1+2j}+2+4(1+2j)(2j)+8(-1)^j\big)\\
&\equiv F_m(1+4k)+2F_m(1)+4 F_m''(1)+8 \frac{F_m(i)}{i}~(\mathrm{mod}~2^5).
\end{align*}

We know that the sequence $\{F_m(1)~(\mathrm{mod}~2^4)\}_m$ is periodic of period 24 by Proposition \ref{prop:Fib periodic}. Since $m\equiv 16$~(mod 24), $F_{16}(1)\equiv 11~(\mathrm{mod}~2^4)$ implies $F_m(1)\equiv 11~(\mathrm{mod}~2^4)$. Write $F_m(1)=11+16h$ for some $h\in\Z_2$.

We know that the sequence $\{\frac{F_m(i)}{i}\}_m$ is periodic of period 12 by Lemma \ref{lem:Fib input i}. Since $m\equiv 4$~(mod 12), $\frac{F_4(i)}{i}=1$ implies $\frac{F_m(i)}{i}=1$.

We know that the sequence $\{F_m''(1)~(\mathrm{mod}~8)\}_m$ is periodic of period 24 by Lemma \ref{lem:Fib second derivative}. Since $m\equiv 16$~(mod 24), $F_{16}''(1)\equiv 2~(\mathrm{mod}~8)$ implies $F_m(1)\equiv2~(\mathrm{mod}~8)$. Write $F_m''(1)=2+8g$ for some $g\in\Z_2$.

Therefore, we obtain
\begin{align*}
F_m(11+4k)&\equiv (11+4k+2^4A)+2(11+16h)+4(2+8g)+8\\
&\equiv 17+4k+2^4A~(\mathrm{mod}~2^5).
\end{align*}
From expression (\ref{m=4+12q and q=1+2d A,B coeff}), we have that
\[ 1+4k+2^4B\equiv 17+4k+2^4A~(\mathrm{mod}~2^5), \]
so $A+B\equiv 1$ (mod 2).

We compute that
\begin{align*}
(11+4k+2^4A)^{1+2j}&\equiv (11+4k)^{1+2j}+\tbinom{1+2j}{1}(11+4k)^{2j}2^4A\\
&\equiv (11+4k)^{1+2j}+2^4A~(\mathrm{mod}~2^5).
\end{align*}
Therefore,
\begin{align*}
F_m^2(1+4k)&=F_m(11+4k+2^4A)\\
&=\sum_{j=0}^{7+12d}\tbinom{8+12d+j}{1+2j}(11+4k+2^4A)^{1+2j}\\
&\equiv\sum_{j=0}^{7+12d}\tbinom{8+12d+j}{1+2j}\big((11+4k)^{1+2j}+2^4A\big)\\
&\equiv F_m(11+4k)+2^4A\cdot F_m(1)\\
&\equiv 1+4k+2^4B+2^4A(11+16h)\\
&\equiv1+4k+2^4(A+B)\\
&\equiv1+4k+2^4~(\mathrm{mod}~2^5).
\end{align*}
So, $F_m^2(1+4k)=1+4k+2^4+2^5Q$ for some $Q\in\Z$. Finally,
\begin{equation*}
b_4(1+4k)=\frac{F_m^2(1+4k)-(1+4k)}{2^4}=1+2Q\equiv 1~(\mathrm{mod}~2).
\end{equation*}
Therefore, the cycle $\{1+4k,11+4k\}$ strongly grows at level 4, which completes the proof.
\end{proof}

By Proposition \ref{m=4+12q even numbers} and \ref{prop:m=4+12q,q=1}, we conclude that the following is true.
\begin{theorem}
The minimal decomposition of $\Z_2$ for $F_m(x)$ with $m=4+12q$ and $q=1+2d$ for nonnegative integers $d$ is
\[ \Z_2=\{0\}\bigsqcup\Big(\bigcup_{k=0}^3 M_k\Big)\bigsqcup(2\Z_2-\{0\}), \]
where 
\[ M_k=(1+4k+2^4\Z_2)\cup(11+4k+2^4\Z_2). \]
Here, $\{0\}$ is the set of a fixed point and $M_k$'s are the minimal components. The set $2\Z_2-\{0\}$ is the attracting basin of the fixed point $0$.
\end{theorem}

\subsection{Case: $m=4+12q$ with $q=4d$}

We consider the case $m=4+12q$ with $q=4d$ for nonnegative integers $d$. The Fibonacci polynomial becomes
\[ F_m(x)=\sum_{j=0}^{1+24d}\tbinom{2+24d+j}{1+2j}x^{1+2j}. \]

\begin{proposition}\label{prop:m=4+12q,q=0}
Let $m=4+12q$ with $q=4d$ for some nonnegative integer $d$. The Fibonacci polynomial $F_m(x)$ has cycles $\{1+4k,3+4k\}$ of length 2 which strongly grow at level 5 where $k=0,1,\dots,2^3-1$.
\end{proposition}

We omit the proof of Proposition \ref{prop:m=4+12q,q=0} because it is similar to that of Proposition \ref{prop:m=4+12q,q=1}. By Proposition \ref{m=4+12q even numbers} and \ref{prop:m=4+12q,q=0}, we conclude that the following is true.
\begin{theorem}
The minimal decomposition of $\Z_2$ for Fibonacci polynomial $F_m(x)$ with $m=4+12q$ and $q=4d$ for nonnegative integers $d$ is
\[ \Z_2=\{0\}\bigsqcup\Bigl(\bigcup_{k=0}^{2^3-1} M_k\Bigr)\bigsqcup(2\Z_2-\{0\}), \]
where 
\[ M_k=(1+4k+2^5\Z_2)\cup(3+4k+2^5\Z_2). \] 
Here, $\{0\}$ is the set of a fixed point and $M_k$'s are the minimal components. The set $2\Z_2-\{0\}$ is the attracting basin of the fixed point $0$.
\end{theorem}

\subsection{Case: $m=4+12q$ with $q=6+8d$}

We consider the case $m=4+12q$ with $q=6+8d$ for nonnegative integers $d$. The Fibonacci polynomial becomes
\[ F_m(x)=\sum_{j=0}^{37+48d}\tbinom{38+48d+j}{1+2j}x^{1+2j}. \]

\begin{proposition}\label{prop:m=4+12q,q=6}
Let $m=4+12q$ with $q=6+8d$ for some nonnegative integer $d$. The Fibonacci polynomial $F_m(x)$ has cycles $\{1+4k,19+20k+16k^2\}$ of length 2 which strongly grows at level 6 where $k=0,\dots,2^4-1$.
\end{proposition}

\begin{proof}
For the proof, we divide the cycles $\{1+4k,19+20k+16k^2\}$ into two types of cycles $\{1+8k,19+40k\}$ and $\{63+56k,45+24k\}$ where $k=0,\dots,2^3-1$.

We omit the further proof because it is similar to that of Proposition \ref{prop:m=4+12q,q=1}.
\end{proof}

By Proposition \ref{m=4+12q even numbers} and \ref{prop:m=4+12q,q=6}, we conclude that the following is true.
\begin{theorem}
The minimal decomposition of $\Z_2$ for $F_m(x)$ with $m=4+12q$ and $q=6+8d$ for nonnegative integers $d$ is
\[ \Z_2=\{0\}\bigsqcup\Big(\bigcup_{k=0}^{2^4-1} M_{k}\Big)\bigsqcup(2\Z_2-\{0\}), \]
where 
\[ M_{k}=(1+4k+2^6\Z_2)\cup(19+20k+16k^2+2^6\Z_2). \]
Here, $\{0\}$ is the set of a fixed point and $M_k$'s are the minimal components. The set $2\Z_2-\{0\}$ is the attracting basin of the fixed point $0$.
\end{theorem}

\subsection{Case: $m=4+12q$ with $q=10+16d$}

We consider the case $m=4+12q$ with $q=10+16d$ for nonnegative integers $d$. The Fibonacci polynomial becomes
\[ F_m(x)=\sum_{j=0}^{61+96d}\tbinom{62+96d+j}{1+2j}x^{1+2j}. \]

\begin{proposition}\label{prop:m=4+12q,q=10}
Let $m=4+12q$ with $q=10+16d$ for some nonnegative integer $d$. The Fibonacci polynomial $F_m(x)$ has cycles $\{1+4k,51+116k+48k^2\}$ of length 2 which strongly grow at level 7 where $k=0,\dots,2^5-1$.
\end{proposition}

\begin{proof}
For the proof, we divide the cycles $\{1+4k,51+116k+48k^2\}$ into two types of cycles $\{1+8k,51+40k\}$ and $\{127+120k,77+88k\}$ where $k=0,\dots,2^4-1$.

We omit the further proof because it is similar to that of Proposition \ref{prop:m=4+12q,q=1}.
\end{proof}

By Proposition \ref{m=4+12q even numbers} and \ref{prop:m=4+12q,q=10}, we conclude that the following is true.
\begin{theorem}
The minimal decomposition of $\Z_2$ for $F_m(x)$ with $m=4+12q$ and $q=10+16d$ for nonnegative integers $d$ is
\[ \Z_2=\{0\}\bigsqcup\Big(\bigcup_{k=0}^{2^5-1} M_{k}\Big)\bigsqcup(2\Z_2-\{0\}), \]
where 
\[ M_{k}=(1+4k+2^7\Z_2)\cup\big(51+116k+48k^2+2^7\Z_2\big). \] 
Here, $\{0\}$ is the set of a fixed point and $M_k$'s are the minimal components. The set $2\Z_2-\{0\}$ is the attracting basin of the fixed point $0$.
\end{theorem}

\subsection{Case: $m=4+12q$ with $q=18+32d$}

We consider the case $m=4+12q$ with $q=18+32d$ for nonnegative integers $d$. The Fibonacci polynomial becomes
\[ F_m(x)=\sum_{j=0}^{109+192d}\tbinom{110+192d+j}{1+2j}x^{1+2j}. \]

\begin{proposition}\label{prop:m=4+12q,q=18}
Let $m=4+12q$ with $q=18+32d$ for some nonnegative integer $d$. The Fibonacci polynomial $F_m(x)$ has cycles $\{1+4k,243+116k+112k^2+64k^3\}$ of length 2 which strongly grow at level 8 where $k=0,\dots,2^6-1$.
\end{proposition}

\begin{proof}
For the proof, we divide the cycles $\{1+4k,243+116k+112k^2+64k^3\}$ into four types of cycles $\{1+16k,243-48k\}$, $\{5+16k,23+80k\}$, $\{251-16k,233-80k\}$ and $\{255-16k,13+48k\}$ where $k=0,\dots,2^4-1$.

We omit the further proof because it is similar to that of Proposition \ref{prop:m=4+12q,q=1}.
\end{proof}

By Proposition \ref{m=4+12q even numbers} and \ref{prop:m=4+12q,q=18}, we conclude that the following is true.
\begin{theorem}
The minimal decomposition of $\Z_2$ for $F_m(x)$ with $m=4+12q$ and $q=18+32d$ for nonnegative integers $d$ is
\[ \Z_2=\{0\}\bigsqcup\Big(\bigcup_{k=0}^{2^6-1} M_{k}\Big)\bigsqcup(2\Z_2-\{0\}), \]
where 
\[ M_{k}=(1+4k+2^8\Z_2)\cup(243+116k+112k^2+64k^3+2^8\Z_2). \]
Here, $\{0\}$ is the set of a fixed point and $M_k$'s are the minimal components. The set $2\Z_2-\{0\}$ is the attracting basin of the fixed point $0$.
\end{theorem}

\subsection{Case: $m=4+12q$ with $q=34+64d$}

We consider the case $m=4+12q$ with $q=34+64d$ for nonnegative integers $d$. The Fibonacci polynomial becomes
\[ F_m(x)=\sum_{j=0}^{205+384d}\tbinom{206+384d+j}{1+2j}x^{1+2j}. \]

\begin{proposition}\label{prop:m=4+12q,q=34}
Let $m=4+12q$ with $q=34+64d$ for some nonnegative integer $d$. The Fibonacci polynomial $F_m(x)$ has cycles $\{1+4k,115+116k+112k^2+64k^3\}$ of length 2 which strongly grow at level 9 where $k=0,\dots,2^7-1$.
\end{proposition}

\begin{proof}
For the proof, we divide the cycles $\{1+4k,115+116k+112k^2+64k^3\}$ into four types of cycles $\{1+16k,115+208k\}$, $\{5+16k,407-176k\}$, $\{511-16k,397-208k\}$ and $\{507-16k,105+176k\}$ where $k=0,\dots,2^5-1$.

We omit the further proof because it is similar to that of Proposition \ref{prop:m=4+12q,q=1}.
\end{proof}

By Proposition \ref{m=4+12q even numbers} and \ref{prop:m=4+12q,q=34}, we conclude that the following is true.
\begin{theorem}
The minimal decomposition of $\Z_2$ for $F_m(x)$ with $m=4+12q$ and $q=34+64d$ for nonnegative integers $d$ is
\[ \Z_2=\{0\}\bigsqcup\Big(\bigcup_{k=0}^{2^7-1} M_{k}\Big)\bigsqcup(2\Z_2-\{0\}), \]
where 
\[ M_{k}=(1+4k+2^9\Z_2)\cup\big(115+116k+112k^2+64k^3+2^9\Z_2\big). \] 
Here, $\{0\}$ is the set of a fixed point and $M_k$'s are the minimal components. The set $2\Z_2-\{0\}$ is the attracting basin of the fixed point $0$.
\end{theorem}

\subsection{Case: $m=4+12q$ with $q=2+64d$}

We consider the case $m=4+12q$ with $q=2+64d$ for nonnegative integers $d$. The Fibonacci polynomial becomes
\[ F_m(x)=\sum_{j=0}^{13+384d}\tbinom{14+384d+j}{1+2j}x^{1+2j}. \]

\begin{proposition}\label{prop:m=4+12q, q=2+64d sg at lev10}
Let $m=4+12q$ with $q=2+64d$ for some nonnegative integer $d$.
\begin{enumerate}
\item When $d$ is even, the Fibonacci polynomial $F_m(x)$ has strongly growing cycles $\{1+16k,371+464k+768k^2\}$ and $\{-1-16k,653+560k+256k^2\}$ and strongly splitting cycles $\{5+16k,663+80k+768k^2\}$ and $\{-5-16k,361+944k+256k^2\}$ with $k=0,1,\dots,2^6-1$ at level 10.

\item When $d$ is odd, the Fibonacci polynomial $F_m(x)$ has strongly growing cycles $\{5+16k,151+80k+768k^2\}$ and $\{-5-16k,873+944k+256k^2\}$ and strongly splitting cycles $\{1+16k,883+464k+768k^2\}$ and $\{-1-16k,141+560k+256k^2\}$ with $k=0,1,\dots,2^6-1$ at level 10.
\end{enumerate}

\end{proposition}

\begin{proof}
1. Assume that $d\equiv 0$ (mod 2). Then, $m\equiv 28~(\mathrm{mod}~3\cdot 2^9)$. We compute that $F_{28}(1+16k)\equiv 371+464k+768k^2$, $F_{28}(371+464k+768k^2)\equiv 1+16k$, $F_{28}(-1-16k)\equiv 653+560k+256k^2$, $F_{28}(653+560k+256k^2)\equiv -1-16k$, $F_{28}(5+16k)\equiv 663+80k+768k^2$, $F_{28}(663+80k+768k^2)\equiv 5+16k$, $F_{28}(-5-16k)\equiv 361+944k+256k^2$, $F_{28}(361+944k+256k^2)\equiv -5-16k~(\mathrm{mod}~2^{10})$. Since the sequence $\{F_m(s)~(\mathrm{mod}~2^{10})\}_m$ is periodic of period 1536 for any odd number $s$ by Proposition \ref{prop:Fib periodic}, for $m=28+768d$,
\begin{align}
&F_{m}(1+16k)\equiv F_{28}(1+16k)\equiv 371+464k+768k^2~ (\mathrm{mod}~2^{10}),\label{thm:m=4+12q,q=2_1}\\
&F_{m}(371+464k+768k^2)\equiv F_{28}(371+464k+768k^2)\equiv 1+16k~ (\mathrm{mod}~2^{10}),\label{thm:m=4+12q,q=2_2}\\
&F_{m}(-1-16k)\equiv F_{28}(-1-16k)\equiv 653+560k+256k^2~ (\mathrm{mod}~2^{10})~\mathrm{and}\notag\\
&F_{m}(653+560k+256k^2)\equiv F_{28}(653+560k+256k^2)\equiv -1-16k~(\mathrm{mod}~2^{10}).\notag\\
&F_{m}(5+16k)\equiv F_{28}(5+16k)\equiv 663+80k+768k^2~ (\mathrm{mod}~2^{10}),\label{thm:m=4+12q,q=2_3}\\
&F_{m}(663+80k+768k^2)\equiv F_{28}(663+80k+768k^2)\equiv 5+16k~ (\mathrm{mod}~2^{10}),\label{thm:m=4+12q,q=2_4}\\
&F_{m}(-5-16k)\equiv F_{28}(-5-16k)\equiv 361+944k+256k^2~ (\mathrm{mod}~2^{10})~\mathrm{and}\notag\\
&F_{m}(361+944k+256k^2)\equiv F_{28}(361+944k+256k^2)\equiv -5-16k~(\mathrm{mod}~2^{10}).\notag
\end{align}
Therefore, $\{1+16k,371+464k+768k^2\}$, $\{-1-16k,653+560k+256k^2\}$, $\{5+16k,663+80k+768k^2\}$ and $\{-5-16k,361+944k+256k^2\}$ are cycles of length 2 at level 10.

Now we compute the quantity $a_{10}$ for the above cycles, as defined in (\ref{a_l}). We have that $F'_{4}(1+16k)\cdot F'_{4}(371+464k+768k^2)\equiv 1$, $F'_{4}(-1-16k)\cdot F'_{4}(653+560k+256k^2)\equiv 1$, $F'_{4}(5+16k)\cdot F'_{4}(663+80k+768k^2)\equiv 1$, $F'_{4}(-5-16k)\cdot F'_{4}(361+944k+256k^2)\equiv 1~(\mathrm{mod}~4)$. Since the sequences $\{F'_m(1+16k)\cdot F'_m(371+464k+768k^2)~(\mathrm{mod}~4)\}_m$, $\{F'_m(-1-16k)\cdot F'_m(653+560k+256k^2)~(\mathrm{mod}~4)\}_m$, $\{F'_m(5+16k)\cdot F'_m(663+80k+768k^2)~(\mathrm{mod}~4)\}_m$ and $\{F'_m(-5-16k)\cdot F'_m(361+944k+256k^2)~(\mathrm{mod}~4)\}_m$ are periodic of period 6 by Lemma \ref{lem:Fib a_n}, for $m=28+768d$, we obtain the quantity $a_{10}$,
\begin{align*}
a_{10}(1+16k)&=F'_m(1+16k)\cdot F'_m(371+464k+768k^2)\\
&\equiv F'_{4}(1+16k)\cdot F'_{4}(371+464k+768k^2)\equiv 1~(\mathrm{mod}~4),\\
a_{10}(-1-16k)&=F'_m(-1-16k)\cdot F'_m(653+560k+256k^2)\\
&\equiv F'_{4}(-1-16k)\cdot F'_{4}(653+560k+256k^2)\equiv 1~(\mathrm{mod}~4),\\
a_{10}(5+16k)&=F'_m(5+16k)\cdot F'_m(663+80k+768k^2)\\
&\equiv F'_{4}(5+16k)\cdot F'_{4}(663+80k+768k^2)\equiv 1~(\mathrm{mod}~4)~\mathrm{and}\\
a_{10}(-5-16k)&=F'_m(-5-16k)\cdot F'_m(361+944k+256k^2)\\
&\equiv F'_{4}(-5-16k)\cdot F'_{4}(361+944k+256k^2)\equiv 1~(\mathrm{mod}~4).\\
\end{align*}

Now we compute the quantity $b_{10}$ for the above cycles, as defined in (\ref{b_l}). From the expressions (\ref{thm:m=4+12q,q=2_1}) and (\ref{thm:m=4+12q,q=2_2}), we write
\begin{align}\label{m=4+12q and q=2+64d A,B coeff}
&F_m(1+16k)= 371+464k+768k^2+2^{10}A~\mathrm{and}\\
&F_m(371+464k+768k^2)= 1+16k+2^{10}B\notag
\end{align}
for some $A,B\in\Z$. To obtain the value $b_{10}$, the following expression is useful. For integers $j\geq 0$, the following holds
\begin{align}\label{b_l with m=4+12q and q=2+64d}
(371+464k+768k^2&)^{1+2j}\\
&\equiv (1+16k)^{1+2j}+(94+768k^2)+(20+192k)(1+2j)\notag\\
&\quad +(152+64k)(1+2j)(2j)+544(1+2j)(2j)(2j-1)\notag\\
&\quad +(256+256k)(-1)^j~(\mathrm{mod}~2^{11}).\notag
\end{align}
We check that the four sequences $\{(1+16k)^{1+2j}\}_j$, $\{(20+192k)(1+2j)\}_j$, $\{(152+64k)(1+2j)(2j)\}_j$, $\{544(1+2j)(2j)(2j-1)\}_j$ and $\{(-1)^j\}_j$ modulo $2^{11}$ are periodic of period 64, 256, 128, 32 and 2, respectively, and a direct computation shows that the expression holds for $j=0,\dots,2^8-1$. Therefore, the expression holds for every $j\geq 0$. Using the expression (\ref{b_l with m=4+12q and q=2+64d}) and the equalities $F_m(1)=\sum_{j=0}^{13+384d}\tbinom{14+384d+j}{1+2j}$, $F_m'(1)=\sum_{j=0}^{13+384d}\tbinom{14+384d+j}{1+2j}(1+2j)$, $F_m''(1)=\sum_{j=0}^{13+384d}\tbinom{14+384d+j}{1+2j}(1+2j)(2j)$,
$F_m'''(1)=\sum_{j=0}^{13+384d}\tbinom{14+384d+j}{1+2j}(1+2j)(2j)(2j-1)$ and $\frac{F_m(i)}{i}=\sum_{j=0}^{13+384d}\tbinom{14+384d+j}{1+2j}(-1)^j$, we obtain
\begin{align*}
F_m(371&+464k+768k^2)\\
&=\sum_{j=0}^{13+384d}\tbinom{14+384d+j}{1+2j}(371+464k+768k^2)^{1+2j}\\
&\equiv\sum_{j=0}^{13+384d}\tbinom{14+384d+j}{1+2j} \big((1+16k)^{1+2j}+(94+768k^2)\\
&\quad +(20+192k)(1+2j) +(152+64k)(1+2j)(2j)\\
&\quad +544(1+2j)(2j)(2j-1)+(256+256k)(-1)^j\big)\\
&\equiv F_m(1+16k)+(94+768k^2)F_m(1)+(20+192k)F_m'(1)\\
&\quad +(152+64k)F_m''(1)+544F_m'''(1)+(256+256k)\frac{F_m(i)}{i}~(\mathrm{mod}~2^{11}).
\end{align*}

We know that the sequence $\{F_m(1)~(\mathrm{mod}~2^{10})\}_m$ is periodic of period 1536 by Proposition \ref{prop:Fib periodic}. Since $m\equiv 28~$ (mod 1536), $F_{28}(1)\equiv 371~(\mathrm{mod}~2^{10})$ implies $F_m(1)\equiv 371~(\mathrm{mod}~2^{10})$. Write $F_m(1)=371+1024h$ for some $h\in\Z_2$.

We know that the sequence $\{F_m'(1)~(\mathrm{mod}~2^9)\}_m$ is periodic of period 1536 by Lemma \ref{lem:Fib a_n}. Since $m\equiv 28$~(mod 1536), $F_{28}'(1)\equiv 285~(\mathrm{mod}~2^9)$ implies $F_m'(1)\equiv 285~(\mathrm{mod}~2^9)$. Write $F_m'(1)=285+512g$ for some $g\in\Z_2$.

We know that the sequence $\{F_m''(1)~(\mathrm{mod}~2^8)\}_m$ is periodic of period 768 by Lemma \ref{lem:Fib second derivative}. Since $m\equiv 28$~(mod 768), $F_{28}''(1)\equiv 174~(\mathrm{mod}~2^8)$ implies $F_m''(1)\equiv 174~(\mathrm{mod}~2^8)$. Write $F_m''(1)=174+256e$ for some $e\in\Z_2$.

We know that the sequence $\{F_m'''(1)~(\mathrm{mod}~2^6)\}_m$ is periodic of period 192 by Lemma \ref{lem:Fib third derivative}. Since $m\equiv 28$~(mod 192), $F_{28}'''(1)\equiv 62~(\mathrm{mod}~2^6)$ implies $F_m'''(1)\equiv 62~(\mathrm{mod}~2^6)$. Write $F_m'''(1)=62+64r$ for some $r\in\Z_2$.

We know that the sequence $\{\frac{F_m(i)}{i}\}_m$ is periodic of period 12 by Lemma \ref{lem:Fib input i}. Since $m\equiv 4$~(mod 12), $\frac{F_4(i)}{i}=1$ implies $\frac{F_m(i)}{i}=1$.

Therefore, we obtain
\begin{align*}
F_m(371+&464k+768k^2)\\
&\equiv (371+464k+768k^2+2^{10}A)+(94+768k^2)(371+1024h)\\
&\quad +(20+192k)(285+512g)+(152+64k)(174+256e)\\
&\quad +544(62+64r)+(256+256k)\\
&\equiv 1025+16k+2^{10}A~(\mathrm{mod}~2^{11}).
\end{align*}
From expression (\ref{m=4+12q and q=2+64d A,B coeff}), we have that
\[ 1+16k+2^{10}B\equiv 1025+16k+2^{10}A~(\mathrm{mod}~2^{11}), \]
so $A+B\equiv 1$ (mod 2).

We compute that
\begin{align*}
(371+&464k+768k^2+2^{10}A)^{1+2j}\\
&\equiv (371+464k+768k^2)^{1+2j}+\tbinom{1+2j}{1}(371+464k+768k^2)^{2j}2^{10}A\\
&\equiv (371+464k+768k^2)^{1+2j}+2^{10}A~(\mathrm{mod}~2^{11}).
\end{align*}
Therefore,
\begin{align*}
F_m^2(1+16k)&=F_m(371+464k+768k^2+2^{10}A)\\
&=\sum_{j=0}^{13+384d}\tbinom{14+384d+j}{1+2j}(371+464k+768k^2+2^{10}A)^{1+2j}\\
&\equiv\sum_{j=0}^{13+384d}\tbinom{14+384d+j}{1+2j}\big((371+464k+768k^2)^{1+2j}+2^{10}A\big)\\
&\equiv F_m(371+464k+768k^2)+2^{10}A\cdot F_m(1)\\
&\equiv 1+16k+2^{10}B+2^{10}A(371+1024h)\\
&\equiv 1+16k+2^{10}(A+B)\\
&\equiv 1+16k+2^{10}~(\mathrm{mod}~2^{11}).
\end{align*}
So, $F_m^2(1+16k)=1+16k+2^{10}+2^{11}Q$ for some $Q\in\Z$. Finally,
\begin{equation*}
b_{10}(1+16k)=\frac{F_m^2(1+16k)-(1+16k)}{2^{10}}=1+2Q\equiv 1~(\mathrm{mod}~2).
\end{equation*}

Because $F_m^2(x)$ is an odd function and the cycle  $\{-1-16k,653+560k+256k^2\}$ is the same as $\{-(1+16k),-(371+464k+768k^2)\}$ modulo $2^{10}$,
\begin{align*}
b_{10}(-1-16k)&=b_{10}\big(-(1+16k)\big)\\
&=\frac{F_m^2\big(-(1+16k)\big)-\big(-(1+16k)\big)}{2^{10}}\\
&=-\frac{F_m^2(1+16k)-(1+16k)}{2^{10}}\\
&=-b_{10}(1+16k)\\
&\equiv 1~(\mathrm{mod}~2).
\end{align*}

From the expressions (\ref{thm:m=4+12q,q=2_3}) and (\ref{thm:m=4+12q,q=2_4}), we write
\begin{align}\label{m=4+12q and q=2+64d A,B coeff_2}
&F_m(5+16k)= 663+80k+768k^2+2^{10}A~\mathrm{and}\\
&F_m(663+80k+768k^2)= 5+16k+2^{10}B\notag
\end{align}
for some $A,B\in\Z$. To obtain the value $b_{10}$, the following expression is useful. For integers $j\geq 0$, the following holds
\begin{align}\label{b_l with m=4+12q and q=2+64d_2}
(663+80k+768k^2&)^{1+2j}\\
&\equiv (5+16k)^{1+2j}+(94+768k^2)+(52+320k)(1+2j)\notag\\
&\quad +(56+448k)(1+2j)(2j)+608(1+2j)(2j)(2j-1)\notag\\
&\quad +(512+1792k)(-1)^j~(\mathrm{mod}~2^{11}).\notag
\end{align}
We check that the four sequences $\{(5+16k)^{1+2j}\}_j$, $\{(52+320k)(1+2j)\}_j$, $\{(56+448k)(1+2j)(2j)\}_j$, $\{608(1+2j)(2j)(2j-1)\}_j$ and $\{(-1)^j\}_j$ modulo $2^{11}$ are periodic of period 256, 256, 128, 32 and 2, respectively, and a direct computation shows that the expression holds for $j=0,\dots,2^8-1$. Therefore, the expression holds for every $j\geq 0$. Using the expression (\ref{b_l with m=4+12q and q=2+64d_2}) and the equalities $F_m(1)=\sum_{j=0}^{13+384d}\tbinom{14+384d+j}{1+2j}$, $F_m'(1)=\sum_{j=0}^{13+384d}\tbinom{14+384d+j}{1+2j}(1+2j)$, $F_m''(1)=\sum_{j=0}^{13+384d}\tbinom{14+384d+j}{1+2j}(1+2j)(2j)$,
$F_m'''(1)=\sum_{j=0}^{13+384d}\tbinom{14+384d+j}{1+2j}(1+2j)(2j)(2j-1)$ and $\frac{F_m(i)}{i}=\sum_{j=0}^{13+384d}\tbinom{14+384d+j}{1+2j}(-1)^j$, we obtain
\begin{align*}
F_m(663&+80k+768k^2)\\
&=\sum_{j=0}^{13+384d}\tbinom{14+384d+j}{1+2j}(663+80k+768k^2)^{1+2j}\\
&\equiv\sum_{j=0}^{13+384d}\tbinom{14+384d+j}{1+2j}\big((5+16k)^{1+2j} +(94+768k^2)\\
&\quad +(52+320k)(1+2j)+(56+448k)(1+2j)(2j)\\
&\quad +608(1+2j)(2j)(2j-1)+(512+1792k)(-1)^j\big)\\
&\equiv F_m(5+16k)+(94+768k^2)F_m(1)+(52+320k)F_m'(1)\\
&\quad +(56+448k)F_m''(1)+608F_m'''(1)+(512+1792k)\frac{F_m(i)}{i}\\
&\equiv (663+80k+768k^2+2^{10}A)+(94+768k^2)(371+1024h)\\
&\quad +(52+320k)(285+512g) +(56+448k)(174+256e)\\
&\quad +608(62+64r)+(512+1792k)\\
&\equiv 5+16k+2^{10}A~(\mathrm{mod}~2^{11}).
\end{align*}
From expression (\ref{m=4+12q and q=2+64d A,B coeff_2}), we have that
\[ 5+16k+2^{10}B\equiv 5+16k+2^{10}A~(\mathrm{mod}~2^{11}), \]
so $A+B\equiv 0$ (mod 2).

We compute that
\begin{align*}
(663+&80k+768k^2+2^{10}A)^{1+2j}\\
&\equiv (663+80k+768k^2)^{1+2j}+\tbinom{1+2j}{1}(663+80k+768k^2)^{2j}2^{10}A\\
&\equiv (663+80k+768k^2)^{1+2j}+2^{10}A~(\mathrm{mod}~2^{11}).
\end{align*}
Therefore,
\begin{align*}
F_m^2(5+16k)&=F_m(663+80k+768k^2+2^{10}A)\\
&=\sum_{j=0}^{13+384d}\tbinom{14+384d+j}{1+2j}(663+80k+768k^2+2^{10}A)^{1+2j}\\
&\equiv\sum_{j=0}^{13+384d}\tbinom{14+384d+j}{1+2j}\big((663+80k+768k^2)^{1+2j}+2^{10}A\big)\\
&\equiv F_m(663+80k+768k^2)+2^{10}A\cdot F_m(1)\\
&\equiv 5+16k+2^{10}B+2^{10}A(371+1024h)\\
&\equiv 5+16k+2^{10}(A+B)\\
&\equiv 5+16k~(\mathrm{mod}~2^{11}).
\end{align*}
So, $F_m^2(5+16k)=5+16k+2^{11}Q$ for some $Q\in\Z$. Finally,
\begin{equation*}
b_{10}(5+16k)=\frac{F_m^2(5+16k)-(5+16k)}{2^{10}}=2Q\equiv 0~(\mathrm{mod}~2).
\end{equation*}

Because $F_m^2(x)$ is an odd function and the cycle  $\{-5-16k,361+944k+256k^2\}$ is the same as $\{-(5+16k),-(663+80k+768k^2)\}$ modulo $2^{10}$,
\begin{align*}
b_{10}(-5-16k)&=b_{10}\big(-(5+16k)\big)\\
&=\frac{F_m^2\big(-(5+16k)\big)-\big(-(5+16k)\big)}{2^{10}}\\
&=-\frac{F_m^2(5+16k)-(5+16k)}{2^{10}}\\
&=-b_{10}(5+16k)\\
&\equiv 0~(\mathrm{mod}~2).
\end{align*}

Therefore, when $d$ is even, the cycles $\{1+16k,371+464k+768k^2\}$ and $\{-1-16k,653+560k+256k^2\}$ strongly grow at level 10 and the cycles $\{5+16k,663+80k+768k^2\}$ and $\{-5-16k,361+944k+256k^2\}$ strongly split at level 10, which completes the proof.

2. The proof of this statement is similar to that of above statement, so we omit it.
\end{proof}

\begin{lemma}\label{lem:m=4+12q, q=2+64d, g_l}
Let $m=4+12q$ with $q=2+64d$ for some nonnegative integer $d$. If $\{g_l, F_m(g_l)\}$ is a 2-cycle at level $l\geq 3$, then $\{g_l, F_m(g_l)\}\equiv \{1,3\}$ or $\{5,7\}$ (mod 8).
\end{lemma}

\begin{proof}
Let $k=0$ or 1. We compute that $F_{4}(1+4k)\equiv 3+4k$ and $F_{4}(3+4k)\equiv 1+4k$ (mod 8). Since the sequence $\{F_m(1+4k)~(\mathrm{mod}~8)\}_m$ is periodic of period 12 by Proposition \ref{prop:Fib periodic}, for $m=28+768d$,
\begin{align*}
F_m(1+4k)&\equiv F_{4}(1+4k)\equiv 3+4k\\
F_m(3+4k)&\equiv F_{4}(3+4k)\equiv 1+4k~(\mathrm{mod}~8).
\end{align*}

Now we compute the quantity $a_{3}$ for the above cycles, as defined in (\ref{a_l}). We have that $F'_{4}(1+4k)\cdot F'_{4}(3+4k)\equiv 1~(\mathrm{mod}~4)$. Since the sequence $\{F'_m(1+4k)\cdot F'_m(3+4k)~(\mathrm{mod}~4)\}_m$ is periodic of period 6 by Lemma \ref{lem:Fib a_n}, for $m=28+768d$, we obtain the quantity $a_{3}$,
\begin{align*}
a_{3}(1+4k)&=F'_m(1+4k)\cdot F'_m(3+4k)\\
&\equiv F'_{4}(1+4k)\cdot F'_{4}(3+4k)\equiv 1~(\mathrm{mod}~4).
\end{align*}

Now we compute the quantity $b_{3}$ for the above cycles, as defined in (\ref{b_l}). We compute that $F_{4}(1+4k)\equiv 3+4k$ and $F_{4}(3+4k)\equiv 1+4k$ (mod $2^4$). Since the sequence $\{F_m(1+4k)~(\mathrm{mod}~16)\}_m$ is periodic of period 24 by Proposition \ref{prop:Fib periodic}, for $m=28+768d$,
\begin{align*}
F_m(1+4k)&\equiv F_{4}(1+4k)\equiv 3+4k\\
F_m(3+4k)&\equiv F_{4}(3+4k)\equiv 1+4k~(\mathrm{mod}~2^4).
\end{align*}
So, we obtain
\begin{align*}
F_m^2(1+4k)&\equiv F_{m}(3+4k)\equiv 1+4k~(\mathrm{mod}~2^4).
\end{align*}
So, $F_m^2(1+4k)=1+4k+2^{4}Q$ for some $Q\in\Z$. Finally,
\begin{equation*}
b_{3}(1+4k)=\frac{F_m^2(1+4k)-(1+4k)}{2^{3}}=2Q\equiv 0~(\mathrm{mod}~2).
\end{equation*}
Therefore, the cycle $\{1+4k,3+4k\}$ strongly splits at level 3. All the 2-cycles at level $l\geq 4$ are lifts of $\{1+4k,3+4k\}$, which completes the proof.
\end{proof}

\begin{lemma}\label{lem:Fib a_n (mod 2^7) for m=28+768d}
Let $m=4+12q$ with $q=2+64d$ for some nonnegative integer $d$. Then,
\[ F_m'(1+8a)\cdot F_m'\big(F_m(1+8a)\big)\equiv 1~(\mathrm{mod}~2^7) \]
for some $a\in\Z_2$.
\end{lemma}

\begin{proof}
We compute that $F_{28}(1+8a)\equiv 115+104a+64a^2~(\mathrm{mod}~2^7)$. Since the sequence $\{F_m(1+8a)~(\mathrm{mod}~2^7)\}_m$ is periodic of period 192 by Proposition \ref{prop:Fib periodic}, for $m=28+768d$,
\[ F_m(1+8a)\equiv F_{28}(1+8a)\equiv 115+104a+64a^2~(\mathrm{mod}~2^7). \]

We have that $F_{28}'(1+8a)\cdot F_{28}'(115+104a+64a^2)\equiv 1~(\mathrm{mod}~2^7)$. Since the sequence $\{F_m'(1+8a)~(\mathrm{mod}~2^7)\}_m$ and $\{F_m'(115+104a+64a^2)~(\mathrm{mod}~2^7)\}_m$ are periodic of period 192 by Lemma \ref{lem:Fib a_n}, then the sequence $\{F_m'(1+8a)\cdot F_m'(115+104a+64a^2)~(\mathrm{mod}~2^7)\}_m$ is periodic of period which is a divisor of 192. So, for $m=28+768d$, we obtain
\begin{align*}
F_m'(1+8a)\cdot F_m'\big(F_m(1+8a)\big)&\equiv F_m'(1+8a)\cdot F_m'(115+104a+64a^2)\\
&\equiv F_{28}'(1+8a)\cdot F_{28}'(115+104a+64a^2)\\
&\equiv 1~(\mathrm{mod}~2^7),
\end{align*}
which completes the proof.
\end{proof}

\begin{proposition}\label{prop:m=4+12(2+64d)}
Let $m=4+12q$ with $q=2+64d$ for some nonnegative integer $d$. Let $l\geq 4$ and $g_l$ be a positive integer. If $\{g_l,F_m(g_l)\}$ forms a cycle of length 2 which strongly grows at level $l+6$, then $\{g_l+2^l k, F_m(g_l+2^l k)\}$ and $\{-(g_l+2^l k), -F_m(g_l+2^l k)\}$ form cycles of length 2 which strongly grow at level $l+6$ for $k=0,1,\dots,2^6-1$.
\end{proposition}

\begin{proof}
Assume that $\{g_l,F_m(g_l)\}$ forms a cycle of length 2 which strongly grows at level $l+6$. By Lemma \ref{lem:m=4+12q, q=2+64d, g_l}, $\{g_l,F_m(g_l)\}\equiv \{1,3\}$ or $\{5,7\}$ (mod 8).

First, assume that $g_l\equiv 1~(\mathrm{mod}~8)$. We can write
\[ F_m^2(g_l)=g_l+2^{l+6}+2^{l+7}c \]
for some $c\in\Z_2$.

For integers $j\geq 0$, the following expression holds that
\begin{align}
(1+2j)j&\equiv 1+j-(-1)^j\label{exp (1+2j)j (mod 8)}\\
&\equiv \frac{1}{2}+\frac{1+2j}{2}-(-1)^j~(\mathrm{mod}~8).\notag
\end{align}
Using the expression (\ref{exp (1+2j)j (mod 8)}), we compute that
\begin{align*}
(g_l+2^l k)^{1+2j}&\equiv g_l^{1+2j}+\tbinom{1+2j}{1}g_l^{2j}2^l k+\tbinom{1+2j}{2}g_l^{2j-1}2^{2l}k^2\\
&\equiv g_l^{1+2j}+(1+2j)g_l^{2j}2^l k+(1+2j)j 2^{2l}k^2\\
&\equiv g_l^{1+2j}+(1+2j)g_l^{2j}2^l k+2^{2l-1}k^2+(1+2j)2^{2l-1}k^2\\
&\quad -2^{2l}k^2(-1)^j~(\mathrm{mod}~2^{l+7}).
\end{align*}
Using the equalities $F_m(1)=\sum_{j=0}^{13+384d}\tbinom{14+384d+j}{1+2j}$, $F_m'(1)=\sum_{j=0}^{13+384d}\tbinom{14+384d+j}{1+2j}(1+2j)$ and $\frac{F_m(i)}{i}=\sum_{j=0}^{13+384d}\tbinom{14+384d+j}{1+2j}(-1)^j$,
\begin{align*}
F_m(g_l+2^l k)&=\sum_{j=0}^{13+384d}\tbinom{14+384d+j}{1+2j}(g_l+2^l k)^{1+2j}\\
&\equiv \sum_{j=0}^{13+384d}\tbinom{14+384d+j}{1+2j}\big(g_l^{1+2j}+\tbinom{1+2j}{1}g_l^{2j}2^l k+\tbinom{1+2j}{2}g_l^{2j-1}2^{2l}k^2\big)\\
&\equiv \sum_{j=0}^{13+384d}\tbinom{14+384d+j}{1+2j}\big(g_l^{1+2j}+(1+2j)g_l^{2j}2^l k+2^{2l-1}k^2\\
&\quad +(1+2j)2^{2l-1}k^2-2^{2l}k^2(-1)^j\big)\\
&\equiv F_m(g_l)+F_m'(g_l)2^l k+F_m(1)2^{2l-1}k^2+F_m'(1)2^{2l-1}k^2\\
&\quad -\frac{F_m(i)}{i}2^{2l}k^2~(\mathrm{mod}~2^{l+7}).
\end{align*}

We know that the sequence $\{F_m(1)~(\mathrm{mod}~2^{4})\}_m$ is periodic of period 24 by Proposition \ref{prop:Fib periodic}. Since $m\equiv 4~$ (mod 24), $F_{4}(1)\equiv 3~(\mathrm{mod}~2^{4})$ implies $F_m(1)\equiv 3~(\mathrm{mod}~2^{4})$. Write $F_m(1)=3+16h$ for some $h\in\Z_2$. Since $g_l\equiv 1$ (mod 8), we have $F_m(g_l)\equiv F_m(1)\equiv 3$ (mod 8).

We know that the sequence $\{F_m'(1)~(\mathrm{mod}~2^4)\}_m$ is periodic of period 48 by Lemma \ref{lem:Fib a_n}. Since $m\equiv 28$~(mod 48), $F_{28}'(1)\equiv 13~(\mathrm{mod}~2^4)$ implies $F_m'(1)\equiv 13~(\mathrm{mod}~2^4)$. Write $F_m'(1)=13+16s$ for some $s\in\Z_2$. Since $g_l\equiv 1$ (mod 8), we have $F_m'(g_l)\equiv F_m'(1)\equiv 5$ (mod 8).

We know that the sequence $\{\frac{F_m(i)}{i}\}_m$ is periodic of period 12 by Lemma \ref{lem:Fib input i}. Since $m\equiv 4$~(mod 12), $\frac{F_4(i)}{i}=1$ implies $\frac{F_m(i)}{i}=1$.

Therefore, we obtain
\begin{align*}
F_m(g_l+2^l k)
&\equiv F_m(g_l)+F_m'(g_l)2^l k+(3+16h)2^{2l-1}k^2+(13+16s)2^{2l-1}k^2\\
&\quad -2^{2l}k^2\\
&\equiv F_m(g_l)+F_m'(g_l)2^l k+7\cdot 2^{2l}k^2~(\mathrm{mod}~2^{l+7}).
\end{align*}
Using the expression (\ref{exp (1+2j)j (mod 8)}), we compute that
\begin{align*}
\big(F_m(g_l)&+F_m'(g_l)2^l k+7\cdot 2^{2l}k^2\big)^{1+2j}\\
&\equiv \big(F_m(g_l)+F_m'(g_l)2^l k\big)^{1+2j}+\tbinom{1+2j}{1}\big(F_m(g_l)+F_m'(g_l)2^l k\big)^{2j}\cdot 7\cdot 2^{2l}k^2\\
&\equiv F_m(g_l)^{1+2j}+\tbinom{1+2j}{1}F_m(g_l)^{2j}F_m'(g_l)2^l k\\
&\quad +\tbinom{1+2j}{2}F_m(g_l)^{2j-1}F_m'(g_l)^2 2^{2l} k^2+(1+2j)\cdot 7\cdot 2^{2l}k^2\\
&\equiv F_m(g_l)^{1+2j}+(1+2j)F_m(g_l)^{2j}F_m'(g_l)2^l k\\
&\quad +(1+2j)j\cdot F_m(g_l)^{2j-1}F_m'(g_l)^2 2^{2l} k^2+(1+2j)\cdot 7\cdot 2^{2l}k^2\\
&\equiv F_m(g_l)^{1+2j}+(1+2j)F_m(g_l)^{2j}F_m'(g_l)2^l k+(1+2j)j\cdot 3\cdot 2^{2l} k^2\\
&\quad +(1+2j)\cdot 7\cdot 2^{2l}k^2\\
&\equiv F_m(g_l)^{1+2j}+(1+2j)F_m(g_l)^{2j}F_m'(g_l)2^l k+3\cdot 2^{2l-1} k^2\\
&\quad +(1+2j)3\cdot 2^{2l-1} k^2-(-1)^j 3\cdot 2^{2l} k^2+(1+2j)\cdot 7\cdot 2^{2l}k^2~(\mathrm{mod}~2^{l+7}).
\end{align*}
By Lemma \ref{lem:Fib a_n (mod 2^7) for m=28+768d}, we obtain
\begin{align*}
F_m^2(g_l+2^l k)&\equiv F_m\big(F_m(g_l)+F_m'(g_l)2^l k+7\cdot 2^{2l}k^2\big)\\
&\equiv \sum_{j=0}^{13+384d}\tbinom{14+384d+j}{1+2j}\big(F_m(g_l)+F_m'(g_l)2^l k+7\cdot 2^{2l}k^2\big)^{1+2j}\\
&\equiv \sum_{j=0}^{13+384d}\tbinom{14+384d+j}{1+2j}\big(F_m(g_l)^{1+2j}+(1+2j)F_m(g_l)^{2j}F_m'(g_l)2^l k\\
&\quad+3\cdot 2^{2l-1} k^2+(1+2j)3\cdot 2^{2l-1} k^2-(-1)^j 3\cdot 2^{2l} k^2\\
&\quad +(1+2j)\cdot 7\cdot 2^{2l}k^2\big)^{1+2j}\\
&\equiv F_m^2(g_l)+F_m'\big(F_m(g_l)\big)F_m'(g_l)2^l k+F_m(1)\cdot 3\cdot 2^{2l-1} k^2\\
&\quad +F_m'(1)\cdot 3\cdot 2^{2l-1} k^2-\frac{F_m(i)}{i} 3\cdot 2^{2l} k^2+F_m'(1)\cdot 7\cdot 2^{2l}k^2\\
&\equiv (g_l+2^{l+6})+2^l k+(3+16h)3\cdot 2^{2l-1} k^2\\
&\quad +(13+16s)3\cdot 2^{2l-1} k^2-3\cdot 2^{2l} k^2+(13+16s)7\cdot 2^{2l}k^2\\
&\equiv g_l+2^l k+2^{l+6}~(\mathrm{mod}~2^{l+7}).
\end{align*}
So, $F_m^2(g_l+2^l k)=g_l+2^l k+2^{l+6}+2^{l+7}Q$ for some $Q\in\Z$. Finally,
\begin{equation*}
b_{l+6}(g_l+2^l k)=\frac{F_m^2(g_l+2^l k)-(g_l+2^l k)}{2^{l+6}}=1+2Q\equiv 1~(\mathrm{mod}~2).
\end{equation*}

Because $F_m^2(x)$ is an odd function,
\begin{align*}
b_{l+6}\big(-(g_l+2^l k)\big)&=\frac{F_m^2\big(-(g_l+2^l k)\big)-\big(-(g_l+2^l k)\big)}{2^{l+6}}\\
&=-\frac{F_m^2(g_l+2^l k)-(g_l+2^l k)}{2^{l+6}}\\
&=-b_{l+6}(g_l+2^l k)\\
&\equiv 1~(\mathrm{mod}~2).
\end{align*}

Therefore, the cycles $\{g_l+2^l k,F_m(g_l+2^l k)\}$ and $\{-(g_l+2^l k),-F_m(g_l+2^l k)\}$ strongly grow at level $l+6$.

Second, assume that $g_l\equiv 5$ (mod 8). 
In this case, the proof is similar to that of the first case, so we omit it.
\end{proof}

In order to explain all strongly growing cycles in the case $m=4+12q$ with $q=2+64d$, we need to construct a sequence.

\begin{definition}\label{def:m=4+12(2+64d)g_l}
	Let $m=4+12q$ with $q=2+64d$ for some nonnegative integer $d$. We define a sequence $\{g_l\}_{l\ge 10}$ recurrently. Let $g_{10}=1$ and $g_{10}'=1+2^3$ if $d$ is even, or $g_{10}=5$ and $g_{10}'=5+2^3$ if $d$ is odd. For $l> 10$, if $\{g_{l-1}', F_m(g_{l-1}')\}$ is a cycle of length 2 and strongly grows at level $l$, then set $g_l=g_{l-1}'$ and $g_l'=g_{l-1}'+2^{l-7}$, or otherwise set $g_l=g_{l-1}'+2^{l-7}$ and $g_l'=g_{l-1}'$.
\end{definition}

A computation, for example a Mathematica experiment, shows that $\{g_l, F_m(g_l)\}$ is a cycle of length 2 and strongly grows at level $l$ for many $l$'s. But the numbers in the sequence $\{g_l\}$ appear randomly, so the proof of strong growth cannot be done, so we leave as a conjecture.

\begin{conjecture}\label{conj}
	Let $m=4+12q$ with $q=2+64d$ for some nonnegative integer $d$. For each $g_l$ for $l\ge 10$ defined in Definition \ref{def:m=4+12(2+64d)g_l}, the Fibonacci polynomial $F_m(x)$ has a cycle $\{g_l, F_m(g_l)\}$ of length 2 which strongly grows at level $l$.
\end{conjecture}

If we assume that the conjecture is true, then by Proposition \ref{prop:m=4+12(2+64d)}, the minimal decomposition of $F_m(x)$ in this case is the following.

\begin{theorem}\label{thm:m=4+12q,q=2+64d}
Suppose that Conjecture \ref{conj} holds. Then the minimal decomposition of $\Z_2$ for $F_m(x)$ with $m=4+12q$ and $q=2+64d$ for nonnegative integers $d$ is
\[ \Z_2=\{0\}
   \bigsqcup\Bigl(\bigcup_{l\geq 10}\bigcup_{k=0}^{2^6-1}
   ( M_{l,k,1}\cup M_{l,k,-1})\Bigr)
   \bigsqcup(2\Z_2-\{0\}), \]
where 
\begin{align*}
	& M_{l,k,1}=(g_l+2^{l-6} k +2^l \Z_2) \cup (F_m(g_l+2^{l-6} k) +2^l\Z_2)\ \text{and} \\
	& M_{l,k,-1}=(-(g_l+2^{l-6} k) +2^l \Z_2) \cup (-F_m(g_l+2^{l-6} k) +2^l \Z_2).
\end{align*} Here, $\{0\}$ is the set of a fixed point and $M_{l,k,1}$'s and $M_{l,k,-1}$'s are the minimal components. The set $2\Z_2-\{0\}$ is the attracting basin of the fixed point $0$.
\end{theorem}

\section{Minimal decompositions for $F_m(x)$ with $m\equiv 6$ (mod 12)}

We consider the case $m=6+12q$ with nonnegative integers $q$. The Fibonacci polynomial becomes
\[ F_m(x)=\sum_{j=0}^{2+6q}\tbinom{3+6q+j}{1+2j}x^{1+2j}. \]

\begin{proposition}\label{prop:m=6+12q odd number}
Let $m=6+12q$ for some nonnegative integer $q$. Then, the Fibonacci polynomial has the property
\[ F_m(1+2\Z_2)\subseteq 2^3+2^4\Z_2. \]
Therefore, the clopen set $1+2\Z_2$ is included in the attracting basin of $F_m$.
\end{proposition}
\begin{proof}
Let $k\in\Z_2$. We show that $F_m(1+2k)\equiv 8~(\mathrm{mod}~2^4)$. Using the fact that $(1+2k)^4\equiv 1~(\mathrm{mod}~2^4)$, for every integer $j\geq 0$, the following holds
\begin{equation}\label{m=6+12q intro equ}
(1+2k)^{1+2j}\equiv (1+4k+2k^2)-(2k+2k^2)(-1)^j~(\mathrm{mod}~2^4).
\end{equation}
Using expression (\ref{m=6+12q intro equ}) and the equality $\frac{F_m(i)}{i}=\sum_{j=0}^{2+6q}\tbinom{3+6q+j}{1+2j}(-1)^{j}$, we obtain
\begin{align*}
F_m(1+2k)&=\sum_{j=0}^{2+6q}\tbinom{3+6q+j}{1+2j}(1+2k)^{1+2j}\\
&\equiv \sum_{j=0}^{2+6q}\tbinom{3+6q+j}{1+2j}\big((1+4k+2k^2)-(2k+2k^2)(-1)^j\big)\\
&\equiv (1+4k+2k^2)F_m(1)-(2k+2k^2)\frac{F_m(i)}{i}~(\mathrm{mod}~2^4).
\end{align*}

We know that the sequence $\{F_m(1)~(\mathrm{mod}~2^4)\}_m$ is periodic of period 24 by Proposition \ref{prop:Fib periodic}. Since $m\equiv 6$ or 18~(mod 24), $F_{6}(1)\equiv 8~\mathrm{and}~ F_{18}(1)\equiv 8~(\mathrm{mod}~2^4)$ imply $F_m(1)\equiv 8~(\mathrm{mod}~2^4)$. Write $F_m(1)=8+16h$ for some $h\in\Z_2$.

We know that the sequence $\{\frac{F_m(i)}{i}\}_m$ is periodic of period 12 by Lemma \ref{lem:Fib input i}. Since $m\equiv 6$~(mod 12), $\frac{F_6(i)}{i}=0$ implies $\frac{F_m(i)}{i}=0$.

Therefore, we obtain
\begin{align*}
F_m(1+2k)&\equiv (1+4k+2k^2)(8+16h)-(2k+2k^2)\cdot 0\\
&\equiv 8~(\mathrm{mod}~2^4).
\end{align*}
Since $F_m(8)\equiv \tbinom{3+6q}{1}8\equiv 8~(\mathrm{mod}~2^4)$, we conclude that $1+2\Z_2$ is included in the attracting basin of $F_m$.
\end{proof}

Now, we consider elements in the set $2\Z_2$, which is the complement of $1+2\Z_2$ in $\Z_2$. For this purpose, we divide nonnegative integers $q$ into cases $q=4d,\ q=1+4d,\ q=2+4d$ and $q=3+4d$ for nonnegative integers $d$.

\subsection{Case: $m=6+12q$ with $q=4d$}

We consider the case $m=6+12q$ with $q=4d$ for nonnegative integers $d$. The Fibonacci polynomial becomes
\[ F_m(x)=\sum_{j=0}^{2+24d}\tbinom{3+24d+j}{1+2j}x^{1+2j}. \]

\begin{proposition}\label{prop:m=6+12q,q=0}
Let $m=6+12q$ with $q=4d$ for some nonnegative integer $d$. The Fibonacci polynomial $F_m(x)$ has cycles $\{(1+4k)2^n,(3+4k)2^n\}$ of length 2 which strongly grow at level $n+3$ for $k=0,1$ and $n\geq 1$.
\end{proposition}

\begin{proof}
We compute the following:
\begin{align*}
F_m\big((1+4k)2^n\big)=&\tbinom{3+24d}{1}(1+4k)2^n+\tbinom{4+24d}{3}(1+4k)^3 2^{3n}\\
&\quad+\tbinom{5+24d}{5} (1+4k)^5 2^{5n}+\dots\\
\equiv&\tbinom{3+24d}{1}(1+4k) 2^n+\tbinom{4+24d}{3}(1+4k)^3 2^{3n}\\
\equiv&(3+4k) 2^n\ (\mathrm{mod}~2^{n+3})~\mathrm{and}
\end{align*}
\begin{align*}
F_m\big((3+4k)2^n\big)=&\tbinom{3+24d}{1}(3+4k) 2^n+\tbinom{4+24d}{3}(3+4k)^3 2^{3n}\\
&\quad+\tbinom{5+24d}{5}(3+4k)^5 2^{5n}+\dots\\
\equiv&\tbinom{3+24d}{1}(3+4k) 2^n+\tbinom{4+24d}{3} (3+4k)^3 2^{3n}\\
\equiv&(1+4k) 2^n~(\mathrm{mod}~2^{n+3}).
\end{align*}
Therefore, $\{(1+4k)2^n,(3+4k)2^n\}$ is a cycle of length 2 at level $n+3$.

Now we compute the quantity $a_{n+3}$ for the above cycles, as defined in (\ref{a_l}).
\begin{align*}
a_{n+3}\big((1+4k) 2^n\big)&=F_m'\big((1+4k) 2^n\big)\cdot F_m'\big((3+4k) 2^n\big)\equiv \tbinom{3+24d}{1}\tbinom{3+24d}{1}\\
&\equiv 1~(\mathrm{mod}~4).
\end{align*}

Now we compute the quantity $b_{n+3}$ for the above cycles, as defined in (\ref{b_l}).
\begin{align*} b_{n+3}\big((1+4k)2^n\big)=\frac{F_m^2\big((1+4k)2^n\big)-(1+4k)2^n}{2^{n+3}}.
\end{align*}
We show that $\nu_2\Big(F_m^2\big((1+4k)2^n\big)-(1+4k)2^n\Big)=n+3$.
\begin{align*}
F_m\big((1+4k)2^n\big)=\sum_{j=0}^{2+24d}\tbinom{3+24d+j}{1+2j}\big((1+4k)2^n\big)^{1+2j}.
\end{align*}
For $j=1$, the summand has the valuation
\begin{align*}
\nu_2\Big(\tbinom{4+24d}{3}&\big((1+4k)2^n\big)^3\Big)\\
&=\nu_2\big(2^{3n+2}(1+6d)(1+8d)(1+12d)(1+4k)^3\big)\\
&\geq n+4.
\end{align*}
For $j\geq 2$, the summand has the valuation
\begin{align*}
\nu_2\Big(\tbinom{3+24d+j}{1+2j}\big((1+4k)2^n\big)^{1+2j}\Big)
\geq (1+2j)n \geq n+4.
\end{align*}
So, we obtain
\begin{align*}
F_m\big((1+4k)2^n\big)&\equiv \tbinom{3+24d}{1}(1+4k)2^n\\
&\equiv (3+8d)(1+4k)2^n~(\mathrm{mod}~2^{n+4}).
\end{align*}
Now we compute the following.
\begin{align*}
F_m^2&\big((1+4k)2^n\big)-(1+4k)2^n\\
&\equiv F_m\big((3+8d)(1+4k)2^n\big)-(1+4k)2^n\\
&\equiv \sum_{j=0}^{2+24d}\tbinom{3+24d+j}{1+2j}\big((3+8d)(1+4k)2^n\big)^{1+2j}-(1+4k)2^n~(\mathrm{mod}~2^{n+4}).
\end{align*}
For $j=0$, with the term $(1+4k)2^n$, the summand has the valuation
\begin{align*}
\nu_2\big(\tbinom{3+24d}{1}&(3+8d)(1+4k)2^n-(1+4k)2^n\big)\\
&=\nu_2\big(2^{n+3}(1+12d+24d^2)(1+4k)\big)\\
&=n+3.
\end{align*}
For $j=1$, the summand has the valuation
\begin{align*}
\nu_2\Big(\tbinom{4+24d}{3}&\big((3+8d)(1+4k)2^n\big)^{3}\Big)\\
&=\nu_2\big(2^{3n+2}(1+6d)(1+8d)(1+12d)(3+8d)^3 (1+4k)^3 \big)\\
&\geq n+4.
\end{align*}
For $j\geq 2$, the summand has the valuation
\begin{align*}
\nu_2\Big(\tbinom{3+24d+j}{1+2j}\big((3+8d)(1+4k)2^n\big)^{1+2j}\Big)
\geq (1+2j)n \geq n+4.
\end{align*}
Combining these, we obtain
\[ \nu_2\Big(F_m^2\big((1+4k)2^n\big)-(1+4k)2^n\Big)=n+3.\]
Hence, $b_{n+3}\big((1+4k)2^n\big)\equiv 1~(\mathrm{mod}~2)$. Therefore, the cycle $\{(1+4k)2^n,(3+4k)2^n\}$ strongly grows at level $n+3$, which completes the proof.
\end{proof}

\begin{theorem}
The minimal decomposition of $\Z_2$ for $F_m(x)$ with $m=6+12q$ and $q=4d$ for nonnegative integers $d$ is
\[ \Z_2=\{0\}\bigsqcup\Bigl(\bigcup_{n\geq 1}\bigcup_{k=0}^1 M_{n,k}\Bigr)\bigsqcup(1+2\Z_2), \]
where 
\[ M_{n,k}=\big((1+4k)2^n+2^{n+3}\Z_2\big)\cup\big((3+4k)2^n+2^{n+3}\Z_2\big). \]
Here, $\{0\}$ is the set of a fixed point and $M_{n,k}$'s are the minimal components. The clopen set $1+2\Z_2$ is the attracting basin of $\{0\}$ and $M_{n,k}$'s.
\end{theorem}

\subsection{Case: $m=6+12q$ with $q=1+2d$}

We consider the case $m=6+12q$ with $q=1+2d$ for nonnegative integers $d$.

\begin{proposition}\label{prop:m=6+12q,q=1}
Let $m=6+12q$ with $q=1+2d$ for some nonnegative integer $d$. Let $t=t(q)$, as defined in (\ref{special value t(q)}). The Fibonacci polynomial $F_m(x)$ has cycles $\{(1+2k)2^n\}$ of length 1 which strongly grow at level $n+t+2$ for $k=0,1,\dots,2^{t+1}-1$ and $n\geq 1$.
\end{proposition}

\begin{proof}
First, assume that $t$ is odd. By Lemma \ref{the relation between q and t}, $c_t=c_{t+1}=0$. Then,
\begin{align*}
q=&1+2^2+2^4+\dots+2^{t-1}+0\cdot 2^t+0\cdot 2^{t+1}+2^{t+2}r\\
=&\frac{2\cdot 2^t-1}{3}+2^{t+2}r
\end{align*}
for some $r\in\Z$. The Fibonacci polynomial becomes \\ $F_m(x)=\sum_{j=0}^{4\cdot 2^t+24\cdot 2^t r}\tbinom{4\cdot 2^t+24\cdot 2^t r+1+j}{1+2j}x^{1+2j}$. Consider the form
\[ F_m\big((1+2k)2^n\big)=\sum_{j=0}^{4\cdot 2^t+24\cdot 2^t r}\tbinom{4\cdot 2^t+24\cdot 2^t r+1+j}{1+2j}\big((1+2k)2^n\big)^{1+2j}. \]
For $j\geq 1$,
\begin{multline*}
\tbinom{4\cdot 2^t+24\cdot 2^t r+1+j}{1+2j}\big((1+2k)2^n\big)^{1+2j}\\
=\frac{\big(2^{t+2}(1+6r)+1+j\big)\dots\big(2^{t+2}(1+6r)+1-j\big)}{(1+2j)!}\big((1+2k)2^n\big)^{1+2j}.
\end{multline*}
In the above expression, the numerator has the factor $2^{t+2}(1+6r)$ and the number of even factors in the numerator, excluding the term $2^{t+2}(1+6r)$, is $j-1$. It is known that for $x=\sum_{i=0}^{s}x_i 2^i$ with $x_i=0$ or 1, $\nu_2(x!)=x-\sum_{i=0}^{s}x_i$. Hence, $\nu_2\big((1+2j)!\big)\leq 2j$. Therefore,
\begin{align}
\nu_2\Big(\tbinom{4\cdot 2^t+24\cdot 2^t r+1+j}{1+2j}&\big((1+2k)2^n\big)^{1+2j}\Big)\label{m=6+12q,q=1+2d coefficient}\\
&\geq(t+2)+(j-1)+(1+2j)n-2j\notag\\
&=1+t+n+j(2n-1)\notag\\
&\geq n+t+2\notag.
\end{align}
So,
\begin{align*}
F_m\big((1+2k)2^n\big)\equiv&\tbinom{4\cdot 2^t+24\cdot 2^t r+1}{1}(1+2k)2^n\\
\equiv&(1+2k)2^n~(\mathrm{mod}~2^{n+t+2}).
\end{align*}
Therefore, $\{(1+2k)2^n\}$ is a cycle of length 1 at level $n+t+2$.

Now we compute the quantity $a_{n+t+2}$ for the above cycles, as defined in (\ref{a_l}).
\[ a_{n+t+2}\big((1+2k)2^n\big)=F_m'\big((1+2k)2^n\big)\equiv \tbinom{4\cdot 2^t+24\cdot 2^t r+1}{1}\equiv 1~(\mathrm{mod}~4). \]

Now we compute the quantity $b_{n+t+2}$ for the above cycles, as defined in (\ref{b_l}).
\begin{align*} b_{n+t+2}\big((1+2k)2^n\big)=\frac{F_m\big((1+2k)2^n\big)-(1+2k)2^n}{2^{n+t+2}}.
\end{align*}
We show that $\nu_2\Big(F_m\big((1+2k)2^n\big)-(1+2k)2^n\Big)=n+t+2$.
\begin{multline*}
F_m\big((1+2k)2^n\big)-(1+2k)2^n\\
=\sum_{j=0}^{4\cdot 2^t+24\cdot 2^t r}\tbinom{4\cdot 2^t+24\cdot 2^t r+1+j}{1+2j}\big((1+2k)2^n\big)^{1+2j}-(1+2k)2^n.
\end{multline*}
For $j=0$, with the term $(1+2k)2^n$, the summand has the valuation
\begin{align*}
\nu_2\big(\tbinom{4\cdot 2^t+24\cdot 2^t r+1}{1}&(1+2k)2^n-(1+2k)2^n\big)\\
&=\nu_2\big(2^{n+t+2}(1+6r)(1+2k)\big)\\
&=n+t+2.
\end{align*}
For $j=1$, the summand has the valuation
\begin{align*}
\nu_2\Big(\tbinom{4\cdot 2^t+24\cdot 2^t r+2}{3}&\big((1+2k)2^n\big)^3\Big)\\
&=\nu_2\Big(\frac{1}{6}(2^{t+2}+6\cdot 2^{t+2}r+2)(2^{t+2}+6\cdot 2^{t+2}r+1)\\
&\qquad\cdot(2^{t+2}+6\cdot 2^{t+2}r)(1+2k)^3 2^{3n}\Big)\\
&\geq n+t+3.
\end{align*}
For $j\geq2$, by expression (\ref{m=6+12q,q=1+2d coefficient}),
\begin{align*}
\nu_2\Big(\tbinom{4\cdot 2^t+24\cdot 2^t r+1+j}{1+2j}\big((1+2k)2^n\big)^{1+2j}\Big)&\geq 1+t+n+j(2n-1)\\
&\geq n+t+3.
\end{align*}
Combining these, we obtain
\[ \nu_2\Big(F_m\big((1+2k)2^n\big)-(1+2k)2^n\Big)=n+t+2. \]
Hence, $b_{n+t+2}\big((1+2k)2^n\big)\equiv 1~(\mathrm{mod}~2)$.

Second, assume that $t$ is even. By Lemma \ref{the relation between q and t}, $c_t=c_{t+1}=1$. Then,
\begin{align*}
q=&1+2^2+2^4+\dots+2^t+2^{t+1}+2^{t+2} r\\
=&\frac{10\cdot 2^t-1}{3}+2^{t+2} r
\end{align*}
for some $r\in\Z$. The Fibonacci polynomial becomes\\
$F_m(x)=\sum_{j=0}^{20\cdot 2^t+24\cdot 2^t r}\tbinom{20\cdot 2^t+24\cdot 2^t r+1+j}{1+2j}x^{1+2j}$. Consider the form
\[ F_m\big((1+2k)2^n\big)=\sum_{j=0}^{20\cdot 2^t+24\cdot 2^t r}\tbinom{20\cdot 2^t+24\cdot 2^t r+1+j}{1+2j}\big((1+2k)2^n\big)^{1+2j}. \]
For $j\geq1$,
\begin{multline*}
\tbinom{20\cdot 2^t+24\cdot 2^t r+1+j}{1+2j}\big((1+2k)2^n\big)^{1+2j}\\
=\frac{\big(2^{t+2}(5+6r)+1+j\big)\dots\big(2^{t+2}(5+6r)+1-j\big)}{(1+2j)!}\big((1+2k)2^n\big)^{1+2j}
\end{multline*}
In the above expression, the numerator has the factor $2^{t+2}(5+6r)$ and the number of even factors in the numerator, excluding the term $2^{t+2}(5+6r)$, is $j-1$. It is known that for $x=\sum_{i=0}^{s}x_i 2^i$ with $x_i=0$ or 1, $\nu_2(x!)=x-\sum_{i=0}^{s}x_i$. Hence, $\nu_2\big((1+2j)!\big)\leq 2j$. Therefore,
\begin{align}
\nu_2\Big(\tbinom{20\cdot 2^t+24\cdot 2^t r+1+j}{1+2j}&\big((1+2k)2^n\big)^{1+2j}\Big)\label{m=6+12q,q=1+2d coefficient_2}\\
&\geq(t+2)+(j-1)+(1+2j)n-2j\notag\\
&=1+t+n+j(2n-1)\notag\\
&\geq n+t+2.\notag
\end{align}
So,
\begin{align*}
F_m\big((1+2k)2^n\big)\equiv&\tbinom{20\cdot 2^t+24R\cdot 2^t+1}{1}(1+2k)2^n\\
\equiv&(1+2k)2^n~(\mathrm{mod}~2^{n+t+2}).
\end{align*}
Therefore, $\{(1+2k)2^n\}$ is a cycle of length 1 at level $n+t+2$.

Now we compute the quantity $a_{n+t+2}$ for the above cycles, as defined in (\ref{a_l}).
\[ a_{n+t+2}\big((1+2k)2^n\big)=F_m'\big((1+2k)2^n\big)\equiv \tbinom{20\cdot 2^t+24\cdot 2^t r+1}{1}\equiv 1~(\mathrm{mod}~4). \]

Now we compute the quantity $b_{n+t+2}$ for the above cycles, as defined in (\ref{b_l}).
\begin{align*}
b_{n+t+2}\big((1+2k)2^n\big)=\frac{F_m\big((1+2k) 2^n\big)-(1+2k) 2^n}{2^{n+t+2}}.
\end{align*}
We show that $\nu_2\Big(F_m\big((1+2k) 2^n\big)-(1+2k) 2^n\Big)=n+t+2$.
\begin{multline*}
F_m\big((1+2k)2^n\big)-(1+2k)2^n\\
=\sum_{j=0}^{20\cdot 2^t+24\cdot 2^t r}\tbinom{20\cdot 2^t+24\cdot 2^t r+1+j}{1+2j}\big((1+2k)2^n\big)^{1+2j}-(1+2k)2^n.
\end{multline*}
For $j=0$, with the term $(1+2k)2^n$, the summand has the valuation
\begin{align*}
\nu_2\big(\tbinom{20\cdot 2^t+24\cdot 2^t r+1}{1}&(1+2k)2^n-(1+2k)2^n\big)\\
&=\nu_2\big(2^{n+t+2}(5+6r)(1+2k)\big)\\
&=n+t+2.
\end{align*}
For $j=1$, the summand has the valuation
\begin{align*}
\nu_2\Big(\tbinom{20\cdot 2^t+24\cdot 2^t r+2}{3}&\big((1+2k)2^n\big)^3\Big)\\
&=\nu_2\Big(\frac{1}{6}(5\cdot 2^{t+2}+6\cdot 2^{t+2}r+2)(5\cdot 2^{t+2}+6\cdot 2^{t+2}r+1)\\
&\qquad\cdot(5\cdot 2^{t+2}+6\cdot 2^{t+2}r)(1+2k)^3 2^{3n}\Big)\\
&\geq n+t+3.
\end{align*}
For $j\geq2$, by expression (\ref{m=6+12q,q=1+2d coefficient_2}),
\begin{align*}
\nu_2\Big(\tbinom{20\cdot 2^t+24\cdot 2^t r+1+j}{1+2j}\big((1+2k)2^n\big)^{1+2j}\Big)&\geq 1+t+n+j(2n-1)\\
&\geq n+t+3.
\end{align*}
Combining these, we obtain
\[ \nu_2\Big(F_m\big((1+2k)2^n\big)-(1+2k)2^n\Big)=n+t+2. \]
Hence, $b_{n+t+2}\big((1+2k)2^n\big)\equiv 1~(\mathrm{mod}~2)$. Therefore, the cycle $\{(1+2k)2^n\}$ strongly grows at level $n+t+2$, which completes the proof.
\end{proof}

By Proposition \ref{prop:m=6+12q odd number} and \ref{prop:m=6+12q,q=1}, we conclude that the following is true.

\begin{theorem}\label{thm:m=6+12q,q=1}
The minimal decomposition of $\Z_2$ for $F_m(x)$ with $m=6+12q$ and $q=1+2d$ for nonnegative integers $d$ is
\[ \Z_2=\{0\}\bigsqcup\Bigl(\bigcup_{n\geq 1}\bigcup_{k=0}^{2^{t+1}-1} M_{n,k}\Bigr)\bigsqcup(1+2\Z_2), \]
where 
\[ M_{n,k}=(1+2k) 2^n+2^{n+t+2}\Z_2. \]
 Here, $\{0\}$ is the set of a fixed point and $M_{n,k}$'s are the minimal components. The clopen set $1+2\Z_2$ is the attracting basin of $\{0\}$ and $M_{n,k}$'s.
\end{theorem}

\subsection{Case: $m=6+12q$ with $q=2+4d$}

We consider the case $m=6+12q$ with $q=2+4d$ for nonnegative integers $d$.

\begin{proposition}\label{prop:m=6+12q,q=2}
Let $m=6+12q$ with $q=2+4d$ for some nonnegative integer $d$. Let $t=t(q)$, as defined in (\ref{special value t(q)}). The Fibonacci polynomial $F_m(x)$ has cycles $\{(1+4k)2^n,(2^{t+2}-1-4k)2^n\}$ of length 2 which strongly grow at level $n+t+3$ for $k=0,1,\dots,2^{t+1}-1$ and $n\geq 1$.
\end{proposition}

The proof of Proposition \ref{prop:m=6+12q,q=2} is similar to that of the second statement of Proposition \ref{prop:m=2+12q,q=1}, so we omit it.

By Proposition \ref{prop:m=6+12q odd number} and \ref{prop:m=6+12q,q=2}, we conclude that the following is true.

\begin{theorem}\label{thm:m=6+12q,q=2}
The minimal decomposition of $\Z_2$ for $F_m(x)$ with $m=6+12q$ and $q=2+4d$ for nonnegative integers $d$ is
\[ \Z_2=\{0\}\bigsqcup\Big(\bigcup_{n\geq 1}\bigcup_{k=0}^{2^{t+1}-1} M_{n,k}\Big)\bigsqcup(1+2\Z_2), \]
where 
\[ M_{n,k}=\{(1+4k)2^n+2^{n+t+3}\Z_2\}\cup\{(2^{t+2}-1-4k)2^n+2^{n+t+3}\Z_2\}. \]
 Here, $\{0\}$ is the set of a fixed point and $M_{n,k}$'s are the minimal components. The clopen set $1+2\Z_2$ is the attracting basin of $\{0\}$ and $M_{n,k}$'s.
\end{theorem}

\section{Minimal decompositions for $F_m(x)$ with $m\equiv 8$ (mod 12)}

We consider the case $m=8+12q$ with nonnegative integers $q$. For the formula (\ref{eq:Fibonacci polynomial_even&odd}), we substitute $m=8+12q$, then we obtain the Fibonacci polynomial
\[ F_m(x)=\sum_{j=0}^{3+6q}\tbinom{4+6q+j}{1+2j}x^{1+2j}. \]

\begin{proposition}\label{prop:m=8+12q even number}
Let $m=8+12q$ for some nonnegative integer $q$. The Fibonacci polynomial $F_m(x)$ has a fixed point 0 in the clopen set $2\Z_2$ with $2\Z_2$ lying its attracting basin.
\end{proposition}
We omit the proof because it is similar to that of Proposition \ref{m=4+12q even numbers}, so we omit it.

Now, we consider elements in the set $1+2\Z_2$, which is the complement of $2\Z_2$ in $\Z_2$. For this purpose, we divide nonnegative integers $q$ into the cases that $q$ is $2d,\ 3+4d,\ 1+8d, 5+16d, 13+32d,\ 29+64d $ and $61+64d$ for nonnegative integers $d$.

\subsection{Case: $m=8+12q$ with $q=2d$}

We consider the case $m=8+12q$ with $q=2d$ for nonnegative integers $d$. The Fibonacci polynomial becomes
\[ F_m(x)=\sum_{j=0}^{3+12d}\tbinom{4+12d+j}{1+2j}x^{1+2j}. \]

\begin{proposition}\label{prop:m=8+12q,q=0}
Let $m=8+12q$ with $q=2d$ for some nonnegative integer $d$. The Fibonacci polynomial $F_m(x)$ has cycles $\{3+6k,7+10k+12k^2\}$ of length 2 which strongly grow at level 4 where $k=0,1,2,3$.
\end{proposition}

\begin{proof}
For the proof, we divide the cycles $\{3+6k,7+10k+12k^2\}$ into two types of cycles $\{1+8k,5+8k\}$ and $\{3+8k,7+8k\}$ where $k=0,1$.

We omit the further proof because it is similar to that of Proposition \ref{prop:m=4+12q,q=1}.
\end{proof}

By Proposition \ref{prop:m=8+12q even number} and \ref{prop:m=8+12q,q=0}, we conclude that the following is true.
\begin{theorem}
The minimal decomposition of $\Z_2$ for $F_m(x)$ with $m=8+12q$ and $q=2d$ for nonnegative integers $d$ is
\[ \Z_2=\{0\}\bigsqcup\big(\bigcup_{k=0}^3 M_{k}\big)\bigsqcup(2\Z_2-\{0\}), \]
where 
\[ M_{k}=\big(3(1+2k)+2^4\Z_2\big)\cup(7+10k+4k^2+8k^3+2^4\Z_2). \]
 Here, $\{0\}$ is the set of a fixed point and $M_{k}$'s are the minimal components. The set $2\Z_2-\{0\}$ is the attracting basin of the fixed point $0$.
\end{theorem}

\subsection{Case: $m=8+12q$ with $q=3+4d$}

We consider the case $m=8+12q$ with $q=3+4d$ for nonnegative integers $d$. The Fibonacci polynomial becomes
\[ F_m(x)=\sum_{j=0}^{21+24d}\tbinom{22+24d+j}{1+2j}x^{1+2j}. \]

\begin{proposition}\label{prop:m=8+12q,q=3}
Let $m=8+12q$ with $q=3+4d$ for some nonnegative integer $d$. The Fibonacci polynomial $F_m(x)$ has cycles $\{1+2k,29+22k+20k^2+24k^3\}$ of length 2 which strongly grow at level 5 where $k=0,\dots,7$.
\end{proposition}

\begin{proof}
For the proof, we divide the cycles $\{1+2k,29+22k+20k^2+24k^3\}$ into two types of cycles $\{1+4k,29-4k\}$ and $\{3+4k,31-4k\}$ where $k=0,1,2,3$.

We omit the further proof because it is similar to that of Proposition \ref{prop:m=4+12q,q=1}.
\end{proof}

By Proposition \ref{prop:m=8+12q even number} and \ref{prop:m=8+12q,q=3}, we conclude that the following is true.
\begin{theorem}
The minimal decomposition of $\Z_2$ for $F_m(x)$ with $m=8+12q$ and $q=3+4d$ for nonnegative integers $d$ is
\[ \Z_2=\{0\}\bigsqcup\big(\bigcup_{k=0}^7 M_{k}\big)\bigsqcup(2\Z_2-\{0\}), \]
where 
\[ M_{k}=(1+2k+2^5\Z_2)\cup(29+22k+20k^2+24k^3+2^5\Z_2). \]
Here, $\{0\}$ is the set of a fixed point and $M_{k}$'s are the minimal components. The set $2\Z_2-\{0\}$ is the attracting basin of the fixed point $0$.
\end{theorem}

\subsection{Case: $m=8+12q$ with $q=1+8d$}

We consider the case $m=8+12q$ with $q=1+8d$ for nonnegative integers $d$. The Fibonacci polynomial becomes
\[ F_m(x)=\sum_{j=0}^{9+48d}\tbinom{10+48d+j}{1+2j}x^{1+2j}. \]

\begin{proposition}\label{prop:m=8+12q,q=1}
Let $m=8+12q$ with $q=1+8d$ for some nonnegative integer $d$. The Fibonacci polynomial $F_m(x)$ has cycles $\{(-1)^k+(-1)^{k+1}\cdot 4k,\big(32+(-1)^k\cdot 13\big)+\big(32+(-1)^k\cdot 20\big)k+\big(32+(-1)^{k+1}\cdot 16\big)k^2\}$ of length 2 which strongly grow at level 6 where $k=0,\dots,15$.
\end{proposition}

\begin{proof}
For the proof, we divide the cycles $\{(-1)^k+(-1)^{k+1}\cdot 4k,\big(32+(-1)^k\cdot 13\big)+\big(32+(-1)^k\cdot 20\big)k+\big(32+(-1)^{k+1}\cdot 16\big)k^2\}$ into two types of cycles $\{1+8k,45+24k\}$ and $\{3+8k,15+24k\}$ where $k=0,\dots,7$.

We omit the further proof because it is similar to that of Proposition \ref{prop:m=4+12q,q=1}.
\end{proof}

By Proposition \ref{prop:m=8+12q even number} and \ref{prop:m=8+12q,q=1}, we conclude that the following is true.
\begin{theorem}
The minimal decomposition of $\Z_2$ for $F_m(x)$ with $m=8+12q$ and $q=1+8d$ for nonnegative integers $d$ is
\[ \Z_2=\{0\}\bigsqcup\big(\bigcup_{k=0}^{15} M_{k}\big)\bigsqcup(2\Z_2-\{0\}), \]
where 
\begin{multline*}
	M_{k}=\big((-1)^k+(-1)^{k+1}\cdot 4k+2^6\Z_2\big) \\ \cup\Big(\big(32+(-1)^k\cdot 13\big)+\big(32+(-1)^k\cdot 20\big)k+\big(32+(-1)^{k+1}\cdot 16\big)k^2+2^6\Z_2\Big).
\end{multline*}
Here, $\{0\}$ is the set of a fixed point and $M_{k}$'s are the minimal components. The set $2\Z_2-\{0\}$ is the attracting basin of the fixed point $0$.
\end{theorem}

\subsection{Case: $m=8+12q$ with $q=5+16d$}

We consider the case $m=8+12q$ with $q=5+16d$ for nonnegative integers $d$. The Fibonacci polynomial becomes
\[ F_m(x)=\sum_{j=0}^{33+96d}\tbinom{34+96d+j}{1+2j}x^{1+2j}. \]

\begin{proposition}\label{prop:m=8+12q,q=5}
Let $m=8+12q$ with $q=5+16d$ for some nonnegative integer $d$. The Fibonacci polynomial $F_m(x)$ has cycles $\{(-1)^k+(-1)^{k+1}\cdot 4k,\big(64+(-1)^k\cdot 13\big)+\big(32+(-1)^k\cdot 20\big)k+\big(32+(-1)^{k+1}\cdot 16\big)k^2\}$ of length 2 which strongly grow at level 7 where $k=0,\dots,31$.
\end{proposition}

\begin{proof}
For the proof, we divide the cycles $\{(-1)^k+(-1)^{k+1}\cdot 4k,\big(64+(-1)^k\cdot 13\big)+\big(32+(-1)^k\cdot 20\big)k+\big(32+(-1)^{k+1}\cdot 16\big)k^2\}$ into two types of cycles $\{1+8k,77+88k\}$ and $\{3+8k,111+24k\}$ where $k=0,\dots,15$.

We omit the further proof because it is similar to that of Proposition \ref{prop:m=4+12q,q=1}.
\end{proof}

By Proposition \ref{prop:m=8+12q even number} and \ref{prop:m=8+12q,q=5}, we conclude that the following is true.
\begin{theorem}
The minimal decomposition of $\Z_2$ for $F_m(x)$ with $m=8+12q$ and $q=5+16d$ for nonnegative integers $d$ is
\[ \Z_2=\{0\}\bigsqcup\big(\bigcup_{k=0}^{31} M_{k}\big)\bigsqcup(2\Z_2-\{0\}), \]
where 
\begin{multline*}
M_{k}=\big((-1)^k+(-1)^{k+1}\cdot 4k+2^7\Z_2\big) \\ \cup  \Big(\big(64+(-1)^k\cdot 13\big) 
    +\big(32+(-1)^k\cdot 20\big)k+\big(32+(-1)^{k+1}\cdot 16\big)k^2+2^7\Z_2\Big).	
\end{multline*} 
Here, $\{0\}$ is the set of a fixed point and $M_{k}$'s are the minimal components. The set $2\Z_2-\{0\}$ is the attracting basin of the fixed point $0$.
\end{theorem}

\subsection{Case: $m=8+12q$ with $q=13+32d$}

We consider the case $m=8+12q$ with $q=13+32d$ for nonnegative integers $d$. The Fibonacci polynomial becomes
\[ F_m(x)=\sum_{j=0}^{81+192d}\tbinom{82+192d+j}{1+2j}x^{1+2j}. \]

\begin{proposition}\label{prop:m=8+12q,q=13}
Let $m=8+12q$ with $q=13+32d$ for some nonnegative integer $d$. The Fibonacci polynomial $F_m(x)$ has cycles $\{(-1)^k+(-1)^{k+1}\cdot 4k,\big(128+(-1)^{k+1}\cdot 115\big)+\big(128+(-1)^{k+1}\cdot 12\big)k+\big(128+(-1)^k\cdot 16\big)k^2+\big(128+(-1)^{k+1}\cdot 64\big)k^3\}$ of length 2 which strongly grow at level 8 where $k=0,\dots,63$.
\end{proposition}

\begin{proof}
For the proof, we divide the cycles $\{(-1)^k+(-1)^{k+1}\cdot 4k,\big(128+(-1)^{k+1}\cdot 115\big)+\big(128+(-1)^{k+1}\cdot 12\big)k+\big(128+(-1)^k\cdot 16\big)k^2+\big(128+(-1)^{k+1}\cdot 64\big)k^3\}$ into four types of cycles $\{1+16k,13+48k\}$, $\{9+16k,101+48k\}$, $\{3+16k,175+176k\}$ and $\{11+16k,199+176k\}$ where $k=0,\dots, 15$.

We omit the further proof because it is similar to that of Proposition \ref{prop:m=4+12q,q=1}.
\end{proof}

By Proposition \ref{prop:m=8+12q even number} and \ref{prop:m=8+12q,q=13}, we conclude that the following is true.
\begin{theorem}
The minimal decomposition of $\Z_2$ for $F_m(x)$ with $m=8+12q$ and $q=13+32d$ for nonnegative integers $d$ is
\[ \Z_2=\{0\}\bigsqcup\big(\bigcup_{k=0}^{63} M_{k}\big)\bigsqcup(2\Z_2-\{0\}), \]
where 
\begin{multline*}
M_{k}=\big((-1)^k+(-1)^{k+1}\cdot 4k+2^8\Z_2\big)\cup\Big(\big(128+(-1)^{k+1}\cdot 115\big) \\ 
+\big(128+(-1)^{k+1}\cdot 12\big)k+\big(128+(-1)^k\cdot 16\big)k^2+\big(128+(-1)^{k+1}\cdot 64\big)k^3+2^8\Z_2\Big).	
\end{multline*} 
Here, $\{0\}$ is the set of a fixed point and $M_{k}$'s are the minimal components. The set $2\Z_2-\{0\}$ is the attracting basin of the fixed point $0$.
\end{theorem}

\subsection{Case: $m=8+12q$ with $q=29+64d$}

We consider the case $m=8+12q$ with $q=29+64d$ for nonnegative integers $d$. The Fibonacci polynomial becomes
\[ F_m(x)=\sum_{j=0}^{177+384d}\tbinom{178+384d+j}{1+2j}x^{1+2j}. \]

\begin{proposition}\label{prop:m=8+12q,q=29}
Let $m=8+12q$ with $q=29+64d$ for some nonnegative integer $d$. The Fibonacci polynomial $F_m(x)$ has cycles $\{(-1)^k+(-1)^{k+1}\cdot 4k,\big(256+(-1)^k\cdot 141\big)+\big(256+(-1)^{k+1}\cdot 140\big)k+\big(256+(-1)^(k+1)\cdot 112\big)k^2+\big(256+(-1)^k\cdot 64\big)k^3\}$ of length 2 which strongly grow at level 9 where $k=0,\dots,127$.
\end{proposition}

\begin{proof}
For the proof, we divide the cycles $\{(-1)^k+(-1)^{k+1}\cdot 4k,\big(256+(-1)^k\cdot 141\big)+\big(256+(-1)^{k+1}\cdot 140\big)k+\big(256+(-1)^(k+1)\cdot 112\big)k^2+\big(256+(-1)^k\cdot 64\big)k^3\}$ into eight types of cycles $\{1+32k,397+96k\}$, $\{9+32k,229+96k\}$, $\{17+32k,189+96k\}$, $\{25+32k,277+96k\}$, $\{3+32k,47+352k\}$, $\{11+32k,71+352k\}$, $\{19+32k,479+352k\}$ and $\{27+32k,247+352k\}$ where $k=0,\dots,15$.

We omit the further proof because it is similar to that of Proposition \ref{prop:m=4+12q,q=1}.
\end{proof}

By Proposition \ref{prop:m=8+12q even number} and \ref{prop:m=8+12q,q=29}, we conclude that the following is true.
\begin{theorem}
The minimal decomposition of $\Z_2$ for $F_m(x)$ with $m=8+12q$ and $q=29+64d$ for nonnegative integers $d$ is
\[ \Z_2=\{0\}\bigsqcup\big(\bigcup_{k=0}^{127} M_{k}\big)\bigsqcup(2\Z_2-\{0\}), \]
where 
\begin{multline*}
	M_{k}=\big((-1)^k+(-1)^{k+1}\cdot 4k+2^9\Z_2\big) 
	      \cup\Big(\big(256+(-1)^k\cdot 141\big) \\
	      +\big(256+(-1)^{k+1}\cdot 140\big)k  
	      +\big(256+(-1)^{k+1}\cdot 112\big)k^2+\big(256+(-1)^k\cdot 64\big)k^3+2^9\Z_2\Big).
\end{multline*}
Here, $\{0\}$ is the set of a fixed point and $M_{k}$'s are the minimal components. The set $2\Z_2-\{0\}$ is the attracting basin of the fixed point $0$.
\end{theorem}

\subsection{Case: $m=8+12q$ with $q=61+64d$}

We consider the case $m=8+12q$ with $q=61+64d$ for nonnegative integers $d$. The Fibonacci polynomial becomes
\[ F_m(x)=\sum_{j=0}^{369+384d}\tbinom{370+384d+j}{1+2j}x^{1+2j}. \]

\begin{proposition}\label{prop:m=8+12q, q=61+64d sg at lev10}
Let $m=8+12q$ with $q=61+64d$ for some nonnegative integer $d$.
\begin{enumerate}
\item When $d$ is even, the Fibonacci polynomial $F_m(x)$ has strongly growing cycles $\{5+16k,873+944k+256k^2\}$ and $\{-5-16k,151+80k+768k^2\}$ and strongly splitting cycles $\{1+16k,141+560k+256k^2\}$ and $\{-1-16k,883+464k+768k^2\}$ with $k=0,1,\dots,2^6-1$ at level 10.

\item When $d$ is odd, the Fibonacci polynomial $F_m(x)$ has strongly growing cycles $\{1+16k,653+560k+256k^2\}$ and $\{-1-16k,371+464k+768k^2\}$ and strongly splitting cycles $\{5+16k,361+944k+256k^2\}$ and $\{-5-16k,663+80k+768k^2\}$ with $k=0,1,\dots,2^6-1$ at level 10.
\end{enumerate}
\end{proposition}

The proof of Proposition \ref{prop:m=8+12q, q=61+64d sg at lev10} is similar to that of Proposition \ref{prop:m=4+12q, q=2+64d sg at lev10}, so we omit it.

\begin{proposition}\label{prop:m=8+12(61+64d)}
Let $m=8+12q$ with $q=61+64d$ for some nonnegative integer $d$. Let $l\geq 4$ and $g_l$ be a positive integer. If $\{g_l,F_m(g_l)\}$ forms a cycle of length 2 which strongly grows at level $l+6$, then $\{g_l+2^l k, F_m(g_l+2^l k)\}$ and $\{-(g_l+2^l k), -F_m(g_l+2^l k)\}$ form cycles of length 2 which strongly grow at level $l+6$ for $k=0,1,\dots,2^6-1$.
\end{proposition}

The proof of Proposition \ref{prop:m=8+12(61+64d)} is similar to that of Proposition \ref{prop:m=4+12(2+64d)}, so we omit it.
In order to explain all strongly growing cycles in the case $m=8+12q$ with $q=61+64d$, we need to construct a sequence.

\begin{definition}\label{def:m=8=12(61+64d)g_l}
	Let $m=8+12q$ with $q=61+64d$ for some nonnegative integer $d$. We define a sequence $\{g_l\}_{l\ge 10}$ recurrently. Let $g_{10}=5$ and $g_{10}'=5+2^3$ if $d$ is even, or $g_{10}=1$ and $g_{10}'=1+2^3$ if $d$ is odd. For $l>10$, if $\{g_{l-1}', F_m(g_{l-1}')\}$ is a cycle of length 2 and strongly grows at level $l$, then set $g_l=g_{l-1}'$ and $g_l'=g_{l-1}'+2^{l-7}$, or otherwise set $g_l=g_{l-1}'+2^{l-7}$ and $g_l'=g_{l-1}'$.
\end{definition}

A computation, for example a Mathematica experiment, shows that $\{g_l, F_m(g_l)\}$ is a cycle of length 2 and strongly grows at level $l$ for many $l$'s. But the numbers in the sequence $\{g_l\}$ appear randomly, so the proof of strong growth cannot be done, so we leave as a conjecture.

\begin{conjecture}\label{conj:m=8+12(61+64d)}
	Let $m=8+12q$ with $q=61+64d$ for some nonnegative integer $d$. For each $g_l$ for $l\ge 10$ defined in Definition \ref{def:m=8=12(61+64d)g_l}, the Fibonacci polynomial $F_m(x)$ has a cycle $\{g_l, F_m(g_l)\}$ of length 2 which strongly grows at level $l$.
\end{conjecture}

If we assume that the conjecture is true, then by Proposition \ref{prop:m=8+12(61+64d)}, the minimal decomposition of $F_m(x)$ in this case is the following.

\begin{theorem}\label{thm:m=8+12q,q=61+64d}
Suppose that Conjecture \ref{conj} holds. Then the minimal decomposition of $\Z_2$ for $F_m(x)$ with $m=8+12q$ and $q=61+64d$ for nonnegative integers $d$ is
\[ \Z_2=\{0\}
   \bigsqcup\Bigl(\bigcup_{l\geq 10}\bigcup_{k=0}^{2^6-1}
   (M_{l,k,1}\cup M_{l,k,-1})\Bigr)
   \bigsqcup(2\Z_2-\{0\}), \]
where 
\begin{align*}
	& M_{l,k,1}=(g_l+2^{l-6} k +2^{l} \Z_2) \cup (F_m(g_l+2^{l-6} k) +2^{l} \Z_2) \ \text{and} \\
	& M_{l,k,-1}=(-(g_l+2^{l-6} k) +2^{l}\Z_2) \cup (-F_m(g_l+2^{l-6} k )+2^{l}\Z_2).
\end{align*} 
 Here, $\{0\}$ is the set of a fixed point and $M_{l,k,1}$'s and $M_{l,k,-1}$'s are the minimal components. The set $2\Z_2-\{0\}$ is the attracting basin of the fixed point $0$.
\end{theorem}

\section{Minimal decompositions for $F_m(x)$ with $m\equiv 10$ (mod 12)}

We consider the case $m=10+12q$ with nonnegative integers $q$. For the formula (\ref{eq:Fibonacci polynomial_even&odd}), we substitute $m=8+12q$, then we obtain the Fibonacci polynomial
\[ F_m(x)=\sum_{j=0}^{4+6q}\tbinom{5+6q+j}{1+2j}x^{1+2j}. \]
We divide nonnegative integers $q$ into the cases $q=2d,\ q=1+4d$ and $q=3+4d$ for nonnegative integers $d$.

\subsection{Case: $m=10+12q$ with $q=2d$}

We consider the case $m=10+12q$ with $q=2d$ for nonnegative integers $d$. The Fibonacci polynomial becomes
\[ F_m(x)=\sum_{j=0}^{16+24d}\tbinom{17+24d+j}{1+2j}x^{1+2j}. \]

\begin{proposition}\label{prop:m=10+12q,q=2}
Let $m=10+12q$ with $q=2d$ for some nonnegative integer $d$. Let $t=t(q)$, as defined in (\ref{special value t(q)}). The Fibonacci polynomial $F_m(x)$ has two types of cycles:
\begin{enumerate}
\item the cycles $\{1+4k,7-4k\}$ of length 2 strongly grow at level 4 where $k=0,1,2,3$, and
\item the cycles $\{(1+2k)2^n\}$ of length 1 strongly grow at level $n+t+2$ with $k=0,1,\dots,2^{t+1}-1$ and $n\geq1$.
\end{enumerate}
\end{proposition}

We omit the proof because proofs of the first and second statement are similar to those of Proposition \ref{prop:m=4+12q,q=1} and  \ref{prop:m=6+12q,q=1}, respectively.

By Proposition \ref{prop:m=10+12q,q=2}, we conclude that the following is true.
\begin{theorem}\label{thm:m=10+12q,q=2}
The minimal decomposition of $\Z_2$ for $F_m(x)$ with $m=10+12q$ and $q=2d$ for nonnegative integers $d$ is
\[ \Z_2=\{0\}\bigsqcup\Bigl(\big(\bigcup_{k=0}^3 A_k\big)\cup\big(\bigcup_{n\geq 1}\bigcup_{k=0}^{2^{t+1}-1} M_{n,k}\big)\Bigr), \]
where 
\begin{align*}
	& A_k=(1+4k+2^4\Z_2)\cup(7-4k+2^4\Z_2) \ \text{and} \\
	& M_{n,k}=(1+2k)2^n+2^{n+t+2}\Z_2.
\end{align*} 
Here, $\{0\}$ is the set of a fixed point and $A_k$'s and $M_{n,k}$'s are the minimal components.
\end{theorem}

\subsection{Case: $m=10+12q$ with $q=1+4d$}

We consider the case $m=10+12q$ with $q=1+4d$ for nonnegative integers $d$. The Fibonacci polynomial becomes
\[ F_m(x)=\sum_{j=0}^{10+24d}\tbinom{11+24d+j}{1+2j}x^{1+2j}. \]

\begin{proposition}\label{prop:m=10+12q,q=1}
Let $m=10+12q$ with $q=1+4d$ for some nonnegative integer $d$.
The Fibonacci polynomial $F_m(x)$ has two types of cycles:
\begin{enumerate}
\item the cycles $\{1+4k,15-4k\}$ of length 2 strongly grow at level 5 where $k=0,1,\dots,7$, and

\item the cycles $\{(1+4k)2^n,(3+4k)2^n\}$ of length 2 strongly grow at level $n+3$ for $k=0,1$ and $n\geq1$.
\end{enumerate}
\end{proposition}

We omit the proof because proofs of the first and second statement are similar to those of Proposition \ref{prop:m=4+12q,q=1} and  \ref{prop:m=6+12q,q=0}, respectively.

By Proposition \ref{prop:m=10+12q,q=1}, we conclude that the following is true.
\begin{theorem}
The minimal decomposition of $\Z_2$ for $F_m(x)$ with $m=10+12q$ and $q=1+4d$ for nonnegative integers $d$ is
\[ \Z_2=\{0\}\bigsqcup\Bigl(\big(\bigcup_{k=0}^7 A_k\big)\cup\big(\bigcup_{n\geq 1}\bigcup_{k=0}^1 M_{n,k}\big)\Bigr), \]
where 
\begin{align*}
	& A_k=(1+4k+2^5\Z_2)\cup(15-4k+2^5\Z_2) \ \text{and} \\
	& M_{n,k}=\big((1+4k)2^n+2^{n+3}\Z_2\big)\cup\big((3+4k)2^n+2^{n+3}\Z_2\big).
\end{align*} 
Here, $\{0\}$ is the set of a fixed point and $A_k$'s and $M_{n,k}$'s are the minimal components.
\end{theorem}

\subsection{Case: $m=10+12q$ with $q=3+4d$}

We consider the case $m=10+12q$ with $q=3+4d$ for nonnegative integers $d$. Let $u=u_0(q)$, as defined in (\ref{special value u_0(q)}). Then,
\begin{align*}
q=&1+2+2^2+\dots+2^{u-1}+0\cdot 2^u+2^{u+1} r\\
=&2^u-1+2^{u+1} r
\end{align*}
for some $r\in\Z$. Then, $m=12\cdot 2^{u}-2+24\cdot 2^{u} r$ and the Fibonacci polynomial becomes
\[ F_m(x)=\sum_{j=0}^{6\cdot 2^{u}+12\cdot 2^{u} r-2}\tbinom{6\cdot 2^{u}+12\cdot 2^{u} r-1+j}{1+2j}x^{1+2j}. \]

\begin{proposition}\label{prop:m=10+12q,q=3}
Let $m=10+12q$ with $q=3+4d$ for some nonnegative integer $d$. Let $u=u_0(q)$, as defined in (\ref{special value u_0(q)}). The Fibonacci polynomial $F_m(x)$ has two types of cycles:
\begin{enumerate}
\item the cycles $\{1+4k,2^{u+3}-1-4k\}$ of length 2 strongly grow at level u+4 where $k=0,1,\dots,2^{u+2}-1$, and

\item the cycles $\{(1+4k)2^n,(2^{u+1}-1-4k)2^n\}$ of length 2 strongly grow at level $n+u+2$ with $k=0,\dots,2^u-1$ and $n\geq1$.
\end{enumerate}
\end{proposition}

\begin{proof}
1. For the proof, we divide the cycles $\{1+4k,2^{u+3}-1-4k\}$ into two types of cycles $\{1+2k,2^{u+3}-1-2k\}$ and $\{2^{u+3}+1+2k,2^{u+4}-1-2k\}$ where $k=0,1,\dots,2^{u+1}-1$.

By Proposition \ref{prop:Fib periodic} and the recursive relation (\ref{eq:Fibonacci polynomial_recurrence relation}), we compute that $F_{12\cdot 2^{u}-2}(1+2k)\equiv 2^{u+3}-1-2k$, $F_{12\cdot 2^{u}-2}(2^{u+3}-1-2k)\equiv 1+2k$, $F_{12\cdot 2^{u}-2}(2^{u+3}+1+2k)\equiv 2^{u+4}-1-2k$ and $F_{12\cdot 2^{u}-2}(2^{u+4}-1-2k)\equiv 2^{u+3}+1+2k$ (mod $2^{u+4}$). Since the sequence $\{F_m(s)~(\mathrm{mod}~2^{u+4})\}_m$ is periodic of period $24\cdot 2^{u}$ for any odd number $s$ by Proposition \ref{prop:Fib periodic}, for $m=12\cdot 2^{u}-2+24\cdot 2^{u} r$,
\begin{align}
&F_m(1+2k)\equiv F_{12\cdot 2^{u}-2}(1+2k)\equiv 2^{u+3}-1-2k~(\mathrm{mod}~2^{u+4}),\label{thm:m=10+12q,q=3_1}\\
&F_m(2^{u+3}-1-2k)\equiv F_{12\cdot 2^{u}-2}(2^{u+3}-1-2k)\equiv 1+2k~(\mathrm{mod}~2^{u+4}),\label{thm:m=10+12q,q=3_2}\\
&F_m(2^{u+3}+1+2k)\equiv F_{12\cdot 2^{u}-2}(2^{u+3}+1+2k)\equiv 2^{u+4}-1-2k~(\mathrm{mod}~2^{u+4})~\mathrm{and}\notag\\
&F_m(2^{u+4}-1-2k)\equiv F_{12\cdot 2^{u}-2}(2^{u+4}-1-2k)\equiv 2^{u+3}+1+2k~(\mathrm{mod}~2^{u+4}).\notag
\end{align}
Therefore, $\{1+2k,2^{u+3}-1-2k\}$ and $\{2^{u+3}+1+2k,2^{u+4}-1-2k\}$ are cycles of length 2 at level u+4.

Now we compute the quantity $a_{u+4}$ for the above cycles, as defined in (\ref{a_l}). We have that $F_{4}'(1+2k)\cdot F_{4}'(2^{u+3}-1-2k)\equiv 1$ and $F_{4}'(2^{u+3}+1+2k)\cdot F_{4}'(2^{u+4}-1-2k)\equiv 1~(\mathrm{mod}~4)$. Since the sequence $\{F_m'(1+2k)\cdot F_m'(2^{u+3}-1-2k)~(\mathrm{mod}~4)\}_m$ and $\{F_m'(2^{u+3}+1+2k)\cdot F_m'(2^{u+4}-1-2k)~(\mathrm{mod}~4)\}_m$ are periodic of period 6 by Lemma \ref{lem:Fib a_n}, for $m=12\cdot 2^{u}-2+24\cdot 2^{u} r$, we obtain the quantity $a_{u+4}$
\begin{align*}
a_{u+4}(1+2k)&=F_m'(1+2k)\cdot F_m'(2^{u+3}-1-2k)\\
&\equiv F_{4}'(1+2k)\cdot F_{4}'(2^{u+3}-1-2k)\\
&\equiv 1~(\mathrm{mod}~4)~\mathrm{and}
\end{align*}
\begin{align*}
a_{u+4}(2^{u+3}+1+2k)&=F_m'(2^{u+3}+1+2k)\cdot F_m'(2^{u+4}-1-2k)\\
&\equiv F_{4}'(2^{u+3}+1+2k)\cdot F_{4}'(2^{u+4}-1-2k)\\
&\equiv 1~(\mathrm{mod}~4).
\end{align*}

Now we compute the value $b_{u+4}$ for the above cycles, as defined in (\ref{b_l}). From expressions (\ref{thm:m=10+12q,q=3_1}) and (\ref{thm:m=10+12q,q=3_2}), we write
\begin{align}\label{m=10+12q and q=3+4d A,B coeff}
\begin{split}
&F_m(1+2k)=2^{u+3}-1-2k+2^{u+4} A~\mathrm{and}\\
&F_m(2^{u+3}-1-2k)=1+2k+2^{u+4} B
\end{split}
\end{align}
for some $A,B\in\Z$. To obtain the value $b_{u+4}$, the following expression is useful. For integers $j\geq 0$, the following holds
\begin{align}\label{b_l with m=10+12q and q=3+4d}
(2^{u+3}-1-2k)^{1+2j}&\equiv -(1+2k)^{1+2j}+\tbinom{1+2j}{1}(1+2k)^{2j}\cdot 2^{u+3}\\
&\equiv -(1+2k)^{1+2j}+(1+2j)(1+4k+4k^2)^j\cdot 2^{u+3}\notag\\
&\equiv -(1+2k)^{1+2j}+(1+2j) 2^{u+3}\notag\\
&\equiv -(1+2k)^{1+2j}+ 2^{u+3}(-1)^j~(\mathrm{mod}~2^{u+5}).\notag
\end{align}
Using expression (\ref{b_l with m=10+12q and q=3+4d}) and the equality\\
$\frac{F_m(i)}{i}=\sum_{j=0}^{6\cdot 2^{u}+12\cdot 2^{u} r-2}\tbinom{6\cdot 2^{u}+12\cdot 2^{u} r-1+j}{1+2j}(-1)^j$, we obtain
\begin{align*}
F_m(2^{u+3}&-1-2k)\\
&=\sum_{j=0}^{6\cdot 2^{u}+12\cdot 2^{u} r-2}\tbinom{6\cdot 2^{u}+12\cdot 2^{u} r-1+j}{1+2j}(2^{u+3}-1-2k)^{1+2j}\\
&\equiv\sum_{j=0}^{6\cdot 2^{u}+12\cdot 2^{u} r-2}\tbinom{6\cdot 2^{u}+12\cdot 2^{u} r-1+j}{1+2j}\big(-(1+2k)^{1+2j}+2^{u+3} (-1)^j\big)\\
&\equiv -F_m(1+2k)+2^{u+3} \frac{F_m(i)}{i}~(\mathrm{mod}~2^{u+5}).
\end{align*}

We know that the sequence $\{F_m(1)~(\mathrm{mod}~2)\}_m$ is periodic of period 3 by Proposition \ref{prop:Fib periodic}. Since $m\equiv 1$~(mod 3), $F_1(1)\equiv 1~(\mathrm{mod}~2)$ implies $F_m(1)\equiv1~(\mathrm{mod}~2)$. Write $F_m(1)=1+2c$ for some $c\in\Z_2$.

We know that the sequence $\{\frac{F_m(i)}{i}\}_m$ is periodic of period 12 by Lemma \ref{lem:Fib input i}. Since $m\equiv 10$~(mod 12), $\frac{F_{10}(i)}{i}=-1$ implies $\frac{F_m(i)}{i}=1$.

Therefore, we obtain
\begin{align*}
F_m(2^{u+3}-1-2k)&\equiv -(2^{u+3}-1-2k+2^{u+4} A)+2^{u+3}(-1)\\
&\equiv -2^{u+4}+1+2k-2^{u+4}A~(\mathrm{mod}~2^{u+5}).
\end{align*}
From expression (\ref{m=10+12q and q=3+4d A,B coeff}), we have that
\[ 1+2k+2^{u+4} B\equiv -2^{u+4}+1+2k-2^{u+4}A~(\mathrm{mod}~2^{u+5}), \]
so $A+B\equiv 1$ (mod 2).

We compute that
\begin{align*}
(2^{u+3}-1-2k&+2^{u+4}A)^{1+2j}\\
&\equiv (2^{u+3}-1-2k)^{1+2j}+\tbinom{1+2j}{1}(2^{u+3}-1-2k)^{2j}2^{u+4}A\\
&\equiv (2^{u+3}-1-2k)^{1+2j}+2^{u+4}A~(\mathrm{mod}~2^{u+5}).
\end{align*}
Therefore,
\begin{align*}
F_m^2&(1+2k)\\
&=F_m(2^{u+3}-1-2k+2^{u+4}A)\\
&=\sum_{j=0}^{6\cdot 2^{u}+12\cdot 2^{u} r-2}\tbinom{6\cdot 2^{u}+12\cdot 2^{u} r-1+j}{1+2j}(2^{u+3}-1-2k+2^{u+4}A)^{1+2j}\\
&\equiv\sum_{j=0}^{6\cdot 2^{u}+12\cdot 2^{u} r-2}\tbinom{6\cdot 2^{u}+12\cdot 2^{u} r-1+j}{1+2j}\big((2^{u+3}-1-2k)^{1+2j}+2^{u+4}A\big)\\
&\equiv F_m(2^{u+3}-1-2k)+2^{u+4}A\cdot F_m(1)\\
&=1+2k+2^{u+4}B+2^4A(1+2c)\\
&\equiv1+2k+2^{u+4}(A+B)\\
&\equiv1+2k+2^{u+4}~(\mathrm{mod}~2^{u+5}).
\end{align*}
So, $F_m^2(1+2k)=1+2k+2^{u+4}+2^{u+5}Q$ for some $Q\in\Z$. Finally,
\begin{equation*}
b_{u+4}(1+2k)=\frac{F_m^2(1+2k)-(1+2k)}{2^{u+4}}=1+2Q\equiv 1~(\mathrm{mod}~2).
\end{equation*}

Because $F_m^2(x)$ is an odd function and $2^{u+4}-1-2k\equiv -(1+2k)$ (mod $2^{u+4}$),
\begin{align*}
b_{u+4}(2^{u+4}-1-2k)&=b_{u+4}\big(-(1+2k)\big)\\
&=\frac{F_m^2\big(-(1+2k)\big)-\big(-(1+2k)\big)}{2^{u+4}}\\
&=-\frac{F_m^2\big((1+2k)\big)-\big((1+2k)\big)}{2^{u+4}}\\
&=-b_{u+4}\big((1+2k)\big)\\
&\equiv 1~(\mathrm{mod}~2).
\end{align*}
Therefore, the cycles $\{1+2k,2^{u+3}-1-2k\}$ and $\{2^{u+3}+1+2k,2^{u+4}-1-2k\}$ strongly grow at level u+4, which completes the proof.

2. For the proof, we divide the cycles $\{(1+4k)2^n,(2^{u+1}-1-4k)2^n\}$ into two types of cycles $\{(1+2k)2^n,(2^{u+1}-1-2k)2^n\}$ and $\{(2^{u+1}+1+2k)2^n,(2^{u+2}-1-2k)2^n\}$ where $k=0,1,\dots,2^{u-1}-1$.

Consider the form
\[ F_m\big((1+2k)2^n\big)=\sum_{j=0}^{6\cdot 2^{u}+12\cdot 2^{u} r-2}\tbinom{6\cdot 2^{u}+12\cdot 2^{u} r-1+j}{1+2j}\big((1+2k)2^n\big)^{1+2j}. \]
For $j=1$, the summand has the valuation
\begin{align*}
\nu_2\Big(\tbinom{6\cdot 2^{u}+12\cdot 2^{u} r}{3}&\big((1+2k)2^n\big)^{3}\Big)\\
&=\nu_2\big(2^{3n+u+1} (1+2k)^3 (1+2r)\\
&\quad \cdot (-1+3\cdot 2^u+6\cdot 2^{u} r)(-1+6\cdot 2^{u}+12\cdot 2^{u} r)\big)\\
&\geq n+u+2.
\end{align*}
For $j\geq 2$,
\begin{multline*}
\tbinom{6\cdot 2^{u}+12\cdot 2^{u} r-1+j}{1+2j}\big((1+2k)2^n\big)^{1+2j}\\
\frac{(2^{u+1}(3+6r)-1+j)\dots(2^{u+1}(3+6r)-1-j)}{(1+2j)!}\big((1+2k)2^n\big)^{1+2j}.
\end{multline*}
In the above expression, the numerator has the factor $2^{u+1}(3+6r)$ and the number of even factors in the numerator, excluding the term $2^{u+1}(3+6r)$, is $j-1$. It is known that for $x=\sum_{i=0}^{s}x_i 2^i$ with $x_i=0$ or 1, $\nu_2(x!)=x-\sum_{i=0}^{s}x_i$. Hence, $\nu_2\big((1+2j)!\big)\leq 2j$. Therefore,
\begin{align}
\nu_2\Big(\tbinom{6\cdot 2^{u}+12\cdot 2^{u} r-1+j}{1+2j}&\big((1+2k)2^n\big)^{1+2j}\Big)\label{m=10+12q,q=3+4d coefficient}\\
&\geq(u+1)+(j-1)+(1+2j)n-2j\notag\\
&=n+u+j(2n-1)\notag\\
&\geq n+u+2\notag.
\end{align}
So,
\begin{align*}
F_m\big((1+2k)2^n\big)\equiv&\tbinom{6\cdot 2^{u}+12\cdot 2^{u} r-1}{1}(1+2k)2^n\\
\equiv&(2^{u+1}-1-2k)2^n~(\mathrm{mod}~2^{n+u+2}),
\end{align*}
\begin{align*}
F_m\big((2^{u+1}-1-2k)2^n\big)\equiv&\tbinom{6\cdot 2^{u}+12\cdot 2^{u} r-1}{1}(2^{u+1}-1-2k)2^n\\
\equiv&(1+2k)2^n~(\mathrm{mod}~2^{n+u+2}),
\end{align*}
\begin{align*}
F_m\big((2^{u+1}+1+2k)2^n\big)\equiv&\tbinom{6\cdot 2^{u}+12\cdot 2^{u} r-1}{1}(1+2k)2^n\\
\equiv&(2^{u+2}-1-2k)2^n~(\mathrm{mod}~2^{n+u+2})~\mathrm{and}
\end{align*}
\begin{align*}
F_m\big((2^{u+2}-1-2k)2^n\big)\equiv&\tbinom{6\cdot 2^{u}+12\cdot 2^{u} r-1}{1}(2^{u+1}-1-2k)2^n\\
\equiv&(2^{u+1}+1+2k)2^n~(\mathrm{mod}~2^{n+u+2}).
\end{align*}
Therefore, $\{(1+2k)2^n,(2^{u+1}-1-2k)2^n\}$ and $\{(2^{u+1}+1+2k)2^n,(2^{u+2}-1-2k)2^n\}$ are cycles of length 2 at level $n+u+2$.

Now we compute the quantity $a_{n+u+2}$ for the above cycles, as defined in (\ref{a_l}).
\begin{align*}
a_{n+u+2}\big((1+2k)2^n\big)&=F_m'\big((1+2k)2^n\big)\cdot F_m'\big((2^{u+1}-1-2k)2^n\big)\\
&\equiv \tbinom{6\cdot 2^{u}+12\cdot 2^{u} r-1}{1}\cdot\tbinom{6\cdot 2^{u}+12\cdot 2^{u} r-1}{1}\\
&\equiv 1~(\mathrm{mod}~4)~\mathrm{and}
\end{align*}
\begin{align*}
a_{n+u+2}\big((2^{u+1}+1+2k)2^n\big)&=F_m'\big((2^{u+1}+1+2k)2^n\big)\cdot F_m'\big((2^{u+2}-1-2k)2^n\big)\\
&\equiv \tbinom{6\cdot 2^{u}+12\cdot 2^{u} r-1}{1}\cdot\tbinom{6\cdot 2^{u}+12\cdot 2^{u} r-1}{1}\\
&\equiv 1~(\mathrm{mod}~4).
\end{align*}

Now we compute the quantity $b_{n+u+2}$ for the above cycles, as defined in (\ref{b_l}).
\begin{align*} b_{n+u+2}\big((1+2k)2^n\big)=\frac{F_m^2\big((1+2k)2^n\big)-(1+2k)2^n}{2^{n+u+2}}.
\end{align*}
We show that $\nu_2\Big(F_m^2\big((1+2k)2^n\big)-(1+2k)2^n\Big)=n+u+2$.
\[ F_m\big((1+2k)2^n\big)=\sum_{j=0}^{6\cdot 2^{u}+12\cdot 2^{u} r-2}\tbinom{6\cdot 2^{u}+12\cdot 2^{u} r-1+j}{1+2j}\big((1+2k)2^n\big)^{1+2j}. \]
For $j=1$, the summand has the valuation
\begin{align*}
\nu_2\Big(\tbinom{6\cdot 2^{u}+12\cdot 2^{u} r}{3}&\big((1+2k)2^n\big)^{3}\Big)\\
&=\nu_2\big(2^{3n+u+1} (1+2k)^3 (1+2r)\\
&\quad \cdot (-1+3\cdot 2^u+6\cdot 2^{u} r)(-1+6\cdot 2^{u}+12\cdot 2^{u} r)\big)\\
&\geq n+u+3.
\end{align*}
For $j=2$, the summand has the valuation
\begin{align*}
\nu_2&\Big(\tbinom{6\cdot 2^{u}+12\cdot 2^{u} r+1}{5}\big((1+2k)2^n\big)^{5}\Big)\\
&=\nu_2\big(2^{5n+u-1}\cdot\frac{3}{5}(1+2k)^5 (1+2r)(-1+3\cdot 2^u+6\cdot 2^{u} r)\\
&\quad \cdot(-1+2^{u+1}+2^{u+2} r)(-1+6\cdot 2^{u}+12\cdot 2^{u} r)\\
&\quad \cdot(1+6\cdot 2^{u}+12\cdot 2^{u} r)\big)\\
&\geq n+u+3.
\end{align*}
For $j\geq 3$, by expression (\ref{m=10+12q,q=3+4d coefficient}), $\nu_2\Big(\tbinom{6\cdot 2^{u}+12\cdot 2^{u} r-1+j}{1+2j}\big((1+2k)2^n\big)^{1+2j}\Big)
\geq n+u+j(2n-1) \geq n+u+3.$
So, we obtain
\begin{align*}
F_m\big((1+2k)2^n\big)&\equiv \tbinom{6\cdot 2^{u}+12\cdot 2^{u} r-1}{1}(1+2k)2^n\\
&\equiv (1+2k)(-1+6\cdot 2^{u}+12\cdot 2^{u} r)2^n~(\mathrm{mod}~2^{n+u+3}).
\end{align*}
Now we compute the following.
\begin{align*}
F_m^2&\big((1+2k)2^n\big)-(1+2k)2^n\\
&\equiv F_m\big((1+2k)(-1+6\cdot 2^{u}+12\cdot 2^{u} r)2^n\big)-(1+2k)2^n\\
&\equiv \sum_{j=0}^{6\cdot 2^{u}+12\cdot 2^{u} r-2}\tbinom{6\cdot 2^{u}+12\cdot 2^{u} r-1+j}{1+2j}\big((1+2k)(-1+6\cdot 2^{u}+12\cdot 2^{u} r)2^n\big)^{1+2j}\\
&\quad-(1+2k)2^n~(\mathrm{mod}~2^{n+u+3}).
\end{align*}
For $j=0$, with the term $(1+2k)2^n$, the summand has the valuation
\begin{align*}
\nu_2&\big(\tbinom{6\cdot 2^{u}+12\cdot 2^{u} r-1}{1}(1+2k)(-1+6\cdot 2^{u}+12\cdot 2^{u} r)2^n-(1+2k)2^n\big)\\
&=\nu_2\big(2^{n+u+2} 3(1+2k)(1+2r)(-1+3\cdot 2^u+6\cdot 2^{u} r)\big)\\
&=n+u+2.
\end{align*}
For $j=1$, the summand has the valuation
\begin{align*}
\nu_2\Big(&\tbinom{6\cdot 2^{u}+12\cdot 2^{u} r}{3}\big((1+2k)(-1+6\cdot 2^{u}+12\cdot 2^{u} r)2^n\big)^{3}\Big)\\
&=\nu_2\big(2^{3n+u+1}\cdot (1+2k)^3 (1+2r)(-1+3\cdot 2^u+6\cdot 2^{u} r) \\
   & \quad \cdot(-1+6\cdot 2^{u}+12\cdot 2^{u} r)^4\big)\\
&\geq n+u+3.
\end{align*}
For $j=2$, the summand has the valuation
\begin{align*}
\nu_2\Big(&\tbinom{6\cdot 2^{u}+12\cdot 2^{u} r+1}{5}\big((1+2k)(-1+6\cdot 2^{u}+12\cdot 2^{u} r)2^n\big)^{5}\Big)\\
&=\nu_2\big(2^{5n+u-1} \cdot\frac{3}{5}(1+2k)^5 (1+2r)(-1+3\cdot 2^u+6\cdot 2^{u} r)\\
&\quad \cdot(-1+2^{u+1}+2^{u+2} r)(-1+6\cdot 2^{u}+12\cdot 2^{u} r)^6\\
&\quad \cdot(1+6\cdot 2^{u}+12\cdot 2^{u} r)\big)\\
&\geq n+u+3.
\end{align*}
For $j\geq 3$, by expression (\ref{m=10+12q,q=3+4d coefficient}),
\begin{align*}
\nu_2\Big(&\tbinom{6\cdot 2^{u}+12\cdot 2^{u} r-1+j}{1+2j}\big((1+2k)(-1+6\cdot 2^{u}+12\cdot 2^{u} r)2^n\big)^{1+2j}\Big)\\
&\geq n+u+j(2n-1)\\
&\geq n+u+3.
\end{align*}	
Combining these, we obtain
\[ \nu_2\Big(F_m^2\big((1+2k)2^n\big)-(1+2k)2^n\Big)=n+u+2. \]
Hence, $b_{n+u+2}\big((1+2k)2^n\big)\equiv 1~(\mathrm{mod}~2)$.

Because $F_m^2(x)$ is an odd function and $(2^{u+2}-1-2k)2^n\equiv -(1+2k)2^n$ $(\mathrm{mod}~2^{n+u+2})$,
\begin{align*}
b_{n+u+2}\big((2^{u+2}-1-2k)2^n\big)&=b_{n+u+2}\big(-(1+2k)2^n\big)\\
&=\frac{F_m^2\big(-(1+2k)2^n\big)-\big(-(1+2k)2^n\big)}{2^{n+u+2}}\\
&=-\frac{F_m^2\big((1+2k)2^n\big)-\big((1+2k)2^n\big)}{2^{n+u+2}}\\
&=-b_{n+u+2}\big((1+2k)2^n\big)\\
&\equiv 1~(\mathrm{mod}~2).
\end{align*}
Therefore, the cycles $\{(1+2k)2^n,(2^{u+1}-1-2k)2^n\}$ and $\{(2^{u+1}+1+2k)2^n,(2^{u+2}-1-2k)2^n\}$ strongly grow at level $n+u+2$, which completes the proof.
\end{proof}

By Proposition \ref{prop:m=10+12q,q=3}, we conclude that the following is true.
\begin{theorem}\label{thm:m=10+12q,q=3}
The minimal decomposition of $\Z_2$ for $F_m(x)$ with $m=10+12q$ and $q=3+4d$ for nonnegative integers $d$ and $q=1+2+\cdots+2^{u-1}+0\cdot 2^u+2^{u+1} r,r\in\Z$ is
\[ \Z_2=\{0\}\bigsqcup\Bigl(\big(\bigcup_{k=0}^{2^{u+2}-1} A_k\big)\cup\big(\bigcup_{n\geq 1}\bigcup_{k=0}^{2^{u}-1} M_{n,k}\big)\Bigr), \]
where 
\begin{align*}
	& A_k=(1+4k+2^{u+4}\Z_2)\cup(2^{u+3}-1-4k+2^{u+4}\Z_2) \ \text{and} \\
	& M_{n,k}=\big((1+4k)2^n+2^{n+u+2}\Z_2\big)\cup \big((2^{u+1}-1-4k)2^n+2^{n+u+2}\Z_2\big).
\end{align*}
Here, $\{0\}$ is the set of a fixed point and $A_k$'s and $M_{n,k}$'s are the minimal components.
\end{theorem}



Postal address of Jung and Kim: Department of Mathematics, Korea University, Anam-ro, Seongbuk-gu, Seoul, 02841, Republic of Korea

     Email address of Jung: \texttt{myunghyun.jung07@gmail.com}

     Email address of Kim: \texttt{kim.donggyun@gmail.com}

      ORCID iD of Kim: 0000-0002-8072-8071

Postal address of Song: Institute of Mathematical Sciences, Ewha Womans University, Ewhayeodae-gil, Seodaemun-gu, Seoul, 03760, Republic of Korea

     Email address of Song: \texttt{khsong0118@ewha.ac.kr}

\end{document}